\documentclass[12pt,a4paper]{article}
\usepackage{amsmath,amssymb,indentfirst,amsthm,latexsym}
\usepackage[usenames]{color}
\usepackage{bm}
\usepackage{geometry}
\def\be{\begin{equation}}
\def\ee{\end{equation}}
\def\ba{\begin{array}}
\def\ea{\end{array}}

\newtheorem{Theorem}{Theorem}
\newtheorem{Lemma}{Lemma}[section]
\newtheorem{Proposition}{Proposition}[section]
\newtheorem{Corollary}{Corollary}
\newtheorem{Remark}{Remark}[section]

\newtheorem{Definition}{Definition}[section]

\numberwithin{equation}{section}
\allowdisplaybreaks

\def\be{\begin{equation}}
\def\ee{\end{equation}}
\def\ben{\begin{equation*}}
\def\een{\end{equation*}}
\def\br{\begin{eqnarray}}
\def\er{\end{eqnarray}}

\title{ Almost global solutions to two classes of 1-d Hamiltonian Derivative Nonlinear Schr\"odinger  equations }

\author{  Jing Zhang \footnotemark[2] \qquad \medskip \\
{\small School of Mathematical sciences,} \\ {\small Shanghai Key Laboratory of Pure Mathematics and Mathematical Practice, }\\{\small
East China Normal University, Shanghai, 200241, China} }

\date{}

\begin{document}

\footnotetext[2]{Email: jzhang@math.ecnu.edu.cn\\
Supported by National Natural Science Foundation of China 11871023}

\maketitle

%\begin{spacing}{2}�м���Ϊ����
\noindent

\begin{abstract}
Consider two kinds of 1-d Hamiltonian  Derivative Nonlinear Schr\"odinger (DNLS) equations with respect to different symplectic forms under periodic boundary conditions. The nonlinearities of these equations depend not only on $(x,\psi,\bar{\psi})$ but also  on $(\psi_x,\bar{\psi}_x)$, which means the nonlinearities of these equations  are unbounded.  Suppose that the nonlinearities   depend on the space-variable $x$ periodically.  
  Under some assumptions,   for most potentials  of these two kinds of Hamiltonian DNLS equations, if the initial value   is smaller than $\varepsilon\ll1$ in  $p$-Sobolev norm, then the corresponding solution to these equations  is also smaller than $2\varepsilon$ during a time interval $(-c\varepsilon^{-r_*},c\varepsilon^{-r_*})$(for any given positive $r_*$). The main methods  are  constructing  Birkhoff normal forms to two kinds of  Hamiltonian systems which have unbounded nonlinearities and   using the special symmetry of the unbounded  nonlinearities of  Hamiltonian functions to obtain a long time estimate of the solution in $p$-Sobolev norm.

\end{abstract}

\noindent
{\bf Keyword.} Derivative Nonlinear Schr\"odinger (DNLS) equations,  Hamiltonian systems,  unbounded,  long time stability,    momentum,  Birkhoff normal form    \par\bigskip

\noindent
{\bf AMS subject classifications.} 37K55, 37J40

% ������Ҫ���ʱ����\sup_{\stackrel{gh}{ghg}}

\section{Introduction}

It is very interesting to research the behavior of the solution in high-index Sobolev norm to nonlinear evolution equations with derivative in nonlinearities during a long time interval.

  Consider a  nonlinear Schr\"odinger 
equation
\begin{eqnarray}\label{001}
{\bf i}{\psi}_t &=&\partial_{xx}\psi + F(x,\psi,\bar{\psi}, \partial_x \psi, \partial_x \bar{\psi}), \ x\in[0,2\pi] 
\end{eqnarray}
under periodic boundary condition 
$$\psi(t,x)= \psi(t,x+2\pi) .$$
 Suppose that $F$ satisfies  
 $$F(x+2\pi,\psi,\bar{\psi},\partial_x\psi,\partial_x\bar{\psi})= F(x,\psi,\bar{\psi}, \partial_x \psi, \partial_x \bar{\psi})\ \mbox{ and }
  \ F(x,0,0,0,0)=0.$$ Under this assumption  $\psi=0$ is an equilibrium solution to equation (\ref{001}). 
I am interested in the  behavior of solutions around $\psi=0$ during a long time interval. 
  
If $F$ only depends on $(x,\psi,\bar{\psi})$ and  vanishes at order $n+1$  about $(\psi,\bar{\psi})$ at the origin ($n$ is a positive integer), local existence theory implies that when initial value $\|\psi(x,0)\|_{H^p}\leq \varepsilon\ll1$ the  corresponding solution $\psi(x,t)$ exists at least over an interval $(-c \varepsilon^{-n}, c\varepsilon^{-n})$ and $\|\psi(x,t)\|_{H^p}$  stays bounded on such
an interval. The problem that I am interested in is  that construct  almost global
solutions  when $F$ depends  on $(\psi_x,\bar{\psi}_x)$. An almost global solution  means that for any given $k>0$, when the initial value $\psi(x,0)$ is smaller than $0<\varepsilon\ll1$,  solution $\psi(x,t)$ is also small in a high index Sobolev norm  for any $t\in(-c \varepsilon^{-k}, c\varepsilon^{-k})$ (refer \cite{klm83}).

When investigation concerns equation (\ref{001}) on a compact manifold,  no dispersion is available. Nevertheless, two ways may be used to obtain
solutions, defined on time-intervals larger than the one given by local existence theory. The first one is  using KAM theory to get small amplitude  periodic or quasi-periodic (hence global)
solutions. A lot of work have been devoted to these questions and readers refer \cite{bbp,bbp1,bbp3,bou94,bou1,Craig,E,zgy,gxy,l-y,l-y2,k-p,ku,ku2000,ku98,Way1}.

The second approach concerns the construction of almost global $H^p$-small solutions
for (\ref{001}) on compact manifold. Use Birkhoff normal form method to improve the order of normal form and then exploit integral principle to get almost global solutions. 
 When nonlinearity $F$ to equation (\ref{001}) depends only on $(\psi,\bar{\psi})$,  small initial  data give rise to global solutions and
keep uniform control of the $p$-Sobolev  norm of  solutions ($p$ large enough), over time-intervals of length
$\varepsilon^{-k}$, for any given positive $k$. This has been initiated by Bourgain \cite{bou}, who
stated  results of almost global existence and uniform control to equation\begin{equation}
{\bf i}\psi_t-\psi_{xx}+V(x)\psi +\frac{\partial H}{\partial \bar{\psi}}(\psi,\bar{\psi})=0,
\end{equation}
  for any typical (with large probability)  $V(x)$.
  Bourgain in \cite{bou1} stated that  for any small typical initial value to equation
  \begin{equation}
{\bf i}\psi_t-\psi_{xx}+v|\psi|^2\psi +\frac{\partial H}{\partial \bar{\psi}}(\psi,\bar{\psi})=0,
\end{equation}
 the solution $\psi$   will  satisfy $$ \|\psi(t)\|_{H^p}<C \varepsilon, \quad \mbox{for any }\ |t|\leq\varepsilon^{-B}.$$More results may be
found in the book of Bourgain \cite{bou3}. 
 Almost global
solutions for Hamiltonian Semi-linear Klein-Gordon Equations (without derivative in nonlinearity) on spheres and Zoll manifolds have been obtained
by Bambusi, Delort, Gr\'ebert and Szeftel in \cite{ba4}. 
Berti and  Delort in \cite{BD17} give 
 almost global existence of solutions for capillarity-gravity water waves equations with periodic spatial boundary conditions. 

 Bambusi and Gr\'ebert \cite{ba1} (see also Bambusi \cite{Bam08} and Gr\'ebert \cite{gr1}) 
  prove an abstract Birkhoff normal form theorem for Hamiltonian partial differential
equations and  apply this theorem to semi-linear equations: nonlinear wave equation, nonlinear
Schr\"odinger  equation on the $d$-dimensional ($d\geq1$) torus with nonlinearities
satisfying a property-tame modulus. In a non-resonant case they deduce that any small
amplitude solution remains very close to a torus for a very long time. In \cite{ba3}  Bambusi researches the NLW equation with nonlinearity function depending on $x$ periodically. 

Faou and Gr\'{e}bert in \cite{gr-f} consider a general class of infinite dimensional reversible differential systems and   prove that if the $p$-Sobolev norm of initial
data is  smaller than $\varepsilon$ ( $0<\varepsilon$ small enough) then the solution is  
bounded by $2\varepsilon$ during time of order $\varepsilon^{-r}$ with $r$ arbitrary. This theorem applies to a class of reversible semi-linear PDEs including
 nonlinear Schr\"odinger equation on the $d$-dimensional torus and a class of coupled NLS equations which is reversible but not
Hamiltonian. Feola and Iandoli  in \cite{F-I} give the long time existence for a large class of fully nonlinear, reversible and parity preserving Schr\"odinger equations on the one dimensional torus.

  Delort
and Szeftel in \cite{De1}, \cite{De06},  Delort in \cite{De4}, \cite{De09}  research  semi-linear Klein-Gordon equation (with derivative in nonlinearity) on   spheres and Zoll manifold,  quasi-linear Klein-Gordon equation on tori and $\mathbb{S}^1$, and obtain that when the initial value is small than $\varepsilon>0$, the corresponding solution  exists when $|t|\leq \varepsilon^{-r}.$

Given  a  DNLS equation   
\begin{equation}\label{1555}
{\bf i}{\psi}_t=\partial_{xx}\psi +V\ast \psi+(\frac{ \partial f(\psi,\bar{\psi})}{\partial \bar{\psi}})_x, \quad x\in[0,2\pi],
\end{equation}
 Yuan and Zhang in \cite{z-y} obtain that for most $V$ the solution to (\ref{1555}) is still smaller than $2\varepsilon$ among  time $|t|\leq \varepsilon^{-r}$ (for any given positive $r$), if the initial value is smaller than $\varepsilon\ll1$.
  The nonlinearity in (\ref{1555}) does not directly  depend on  space variable $x$. 
 In \cite{z-y-1} Yuan and Zhang research the long time behavior of the solution to the perturbed KdV equation the nonlinearity  of which is trigonometric polynomial about $x$. 
  
\vspace{12pt}

In this paper,
I focus on the behavior of solutions during a long time interval to  two types of  Hamiltonian Derivative Nonlinear Schr\"odinger (DNLS) equations which depend on $x$ periodically. One type is of the  following form
\begin{equation}\label{in-1}
{\bf i}{\psi}_t=\partial_{xx}\psi+ V* \psi +
\frac12\partial_x\partial_{\bar{\psi}} F(x,\psi,\bar{\psi})+
\partial_{\psi\bar{\psi}}F(x,\psi,\bar{\psi})\psi_x,
\end{equation}where $V$ belongs to {\footnotesize \begin{equation*}  \Theta^0_m:=\bigg\{
     V(x) 
\in L^2([0,2\pi],\mathbb{R}) \  \bigg|\      \widehat{V}_j \cdot\max\{ 1,|j|^m\}\in [-\frac12,\frac12],\   \widehat{V}_j=\widehat{V}_{-j},\ \forall\ j\in \mathbb{Z}
         \bigg\}
     \end{equation*} }
and  the other type is as follow
 \begin{equation}\label{in-2}
{\bf i}{\psi}_t=\partial_{xx}\psi+ V* \psi + {\bf i}\partial_{x}\big(  \partial_{\bar{\psi}} F(x,\psi,\bar{\psi}) \big),
\end{equation}
where $V$ belongs to  
     {\footnotesize  \begin{equation*}   \Theta^{1}_m:=\left\{
     V\in L^2([0,2\pi],\mathbb{C})  \bigg|\
     \widehat{V}_j\cdot\max\{ 1,|j|^m\}\in [-\frac12,\frac12],\  \forall\ j\in \mathbb{Z}\setminus \{0\}, \  \widehat{V}_0=0
     \right\}.
     \end{equation*} }
Under some assumptions, (\ref{in-1}) becomes into a Hamiltonian equation with respect to
 a   symplectic form  $ 
 w^0: = J_0 d \psi \wedge d \bar{\psi},$ $ J_0^{-1}:=\footnotesize{ \big(
                                                                            \begin{array}{cc}
                                                                              0 & -{\bf i} \\
                                                                              {\bf i} & 0 \\
                                                                            \end{array}
                                                                       \big)}$ and  (\ref{in-2}) is   Hamiltonian   under a symplectic form $w^1 := J_1 d \psi \wedge d \bar{\psi},$ $ ( J_1)^{-1}:= \footnotesize{ \big(
  \begin{array}{cc}
                 0 & \partial_x \\
                 \partial_x & 0 \\
                      \end{array}
                   \big)}$.

When
                   $ \partial_{\bar{\psi}} F(x,0,0)=0$ and $  \partial^2_{\bar{\psi}\psi}F (x,0,0)=0, $
 $\psi=0$ is an equilibrium point of equations  (\ref{in-1}) and (\ref{in-2}). In order to get the  almost global solution around the origin to (\ref{in-1}) and (\ref{in-2}), it is required to research  the  behavior of solutions around $\psi=0$ during a long time interval. 

  The result in \cite{z-y}  holds ture for Hamiltonian  DNLS equation  with  nonlinearity independent of $x$. In other words  the momentum of the corresponding Hamiltonian function equals to zero. This property is important in proving the long time stability result.  In \cite{z-y-1} one researches an unbounded perturbed KdV equation the nonlinearity  of which is a  trigonometric polynomial about $\sin kx$ and $\cos kx$ ($|k|\leq M$), i.e., the momentum of the corresponding Hamiltonian function are bounded.   But  generally, the sets of the momentum of  Hamiltonian functions   to equation (\ref{in-1}) and (\ref{in-2}) under Fourier transformation  may be unbounded.  
  Even if assume that  for any $(\psi,\bar{\psi})$ around origin $F\in H^{\beta}([0,2\pi],\mathbb{R})$ ($\beta$ big enough), the corresponding nonlinear vector field of equations (\ref{in-1}) and (\ref{in-2}) are sitll unbounded.  Denote the part of $F$,  the momentum of which is bigger than  $\delta>0$, as $F_1$. Even if  $\delta$  is very large, the Hamiltonian vecotr field of $F_1$  in equations (\ref{in-1}) and (\ref{in-2})  are still unbounded.   The results and methods  in \cite{z-y} and \cite{z-y-1}   do not work to equations (\ref{in-1}) and (\ref{in-2}), directly.   
 In \cite{De3} one consider quasi-linear Klein-Gordon equation on $S^1$.  The nonlinearities are polynomials and  smooth depend on $x$. Their methods are not suitable to DNLS equations (\ref{in-1}) and (\ref{in-2}). In \cite{F-I}  they consider the reversible and parity preserving Schr\"odinger equation. It is necessary to construct    a long time stability theory to solutions of Hamiltonian DNLS equations (\ref{in-1}) and (\ref{in-2}) around the origin.

Under Fourier transformation, equations (\ref{in-1}) and (\ref{in-2}) are transformed into  two types of  Hamiltonian  systems $\theta\in\{0,1\}$
\begin{equation} \label{mod1}
\left\{ \begin{array}{ll}
\dot{u}_j= -{\bf i}\mbox{sgn}^{\theta}(j) \cdot\partial_{\bar u_j} H^{w_{\theta}}(u,\bar{u}),\\
\\
\dot{\bar{u}}_j=\ \ {\bf i} \mbox{sgn}^{\theta}(j)\cdot   \partial_{ u_j} H^{w_{\theta}}(u,\bar{u}),
\end{array} \right. 
\quad j\in\mathbb{Z}^*:=\left\{
\begin{array}{ll}
\mathbb{Z} \ \mbox{or} \ \mathbb{Z}\backslash\{0\}& \mbox{when } \theta=0\\
\\
\mathbb{Z}\backslash\{0\} &\mbox{when }\theta=1
\end{array}\right.
\end{equation}
with Hamiltonian function 
\begin{equation} \label{mod2}
H^{w_{\theta}}(u,\bar{u})=H_0^{w_{\theta}} + P^{w_{\theta}}(u,\bar{u}),\quad
(u,\bar{u})\in {\cal H}^{p},\quad \theta\in\{0,\ 1\}
\end{equation}
under symplectic form
 \begin{equation}\label{sp1}
w_{\theta} :=\left\{
\begin{array}{lll}
&{\bf
i}\sum_{j\in \mathbb Z^*}\mathrm du_j\wedge \mathrm d\bar{u}_j& \theta=0,\\
\\
&{\bf
i}\sum_{j\in \mathbb Z^*} \mbox{sgn}(j) \mathrm du_j\wedge \mathrm d\bar{u}_j&\theta=1,
\end{array}
\right.
\end{equation}
where  $ H_0^{w_{\theta}} :=\sum_{j\in \mathbb Z^{\theta}}   \omega^{w_{\theta}}_j|u_j|^2$ and 
$\omega^{w_{\theta}}_j:= \left\{
\begin{array}{lll}
 (-j^2+\hat{V}_j)& \theta=0\\
\mbox{sgn}(j)(-j^2+\hat{V}_j)& \theta=1\\
\end{array}
\right.$. $P^{w_{\theta}}(u,\bar{u})$ is a  power series  having $(\beta,\theta)$-type symmetric coefficients semi-bounded by $C_{\theta}>0$ (refer definitions \ref{3.4} and \ref{3.5} in section 3). Note that the coefficients of $P^{w_{\theta}}(u,\bar{u})$ are not bounded.  This leads to  the Hamiltonian vector field of $P^{w_{\theta}}(u,\bar{u})$ being unbounded. See  Proposition \ref{cor2} in section 4.  

The problem of finding almost global solutions around the origin to equations (\ref{in-1}) and (\ref{in-2}) is changed into considering  a long time stability of  solutions around equilibrium point $(u,\bar{u})=0$ of (\ref{mod1}).  

In section 3 Theorem \ref{T22} states that under some assumptions, the solution to the two type of Hamiltonian systems which have $(\beta,\theta)$-type symmetric coefficients $(\theta\in\{0,1\}$) is still smaller than $2\varepsilon$ during a time interval $( -\varepsilon^{-r_*},\varepsilon^{-r_*})$, if its initial value is smaller than $\varepsilon\ll 1$.

  The idea of proving Theorem \ref{T22} is to
combine  Birkhoff normal form method with the property of $(\beta,\theta)$-type symmetric coefficients used to obtain
 energy inequalities. 
 
 Let us introduce the  important steps  in proving Theorem \ref{T22}. 

    First step:  construct a coordination transformation ${\cal T}_{w_{\theta}}^{(r)}$ under which  
the Hamiltonian function $H^{w_{\theta}}(u,\bar{u})$ in  (\ref{mod2}) can be transformed into a new Hamiltonian function
 $$ H^{(r,w_{\theta})}= H^{w_{\theta}}\circ{\cal T}_{w_{\theta}}^{(r)}=H^{w_{\theta}}_0+\underbrace{Z^{(r,w_{\theta})} +{\cal R}^{N(r,w_{\theta})}+{\cal R}^{T(r,w_{\theta})}}_{P^{(r,w_{\theta})}} $$
 with a high degree $(\theta,\gamma,\alpha,N)$-normal form ${ Z}^{(r,w_{\theta})}$ (see definition \ref{5.3}). 
    Because the system (\ref{mod1}) is in an infinite dimension, one can only get a partial normal form. 
  ${\cal R}^{N(r,w_{\theta})}$ is at least 3 order about $(u_j,\bar{u}_j)_{|j|>N}$ ($N$ is large enough) and ${\cal R}^{T(r,w_{\theta})} $ has a zero of high order about $(u,\bar{u})$.  The Hamiltonian vector field of $ {P^{(r,w_{\theta})}}$ is still unbounded. The construction of $ {\cal T}_{w_{\theta}}^{(r)}$ is from solving  Homological equation (refer Lemma \ref{lem2}). 
      Because the perturbation in equation (\ref{mod1}) is unbounded, a strong non resonant condition to frequencies $\{\omega^{w_{\theta}}_j(V)\}$  is needed to keep the transformation $ {\cal T}_{w_{\theta}}^{(r)}$ bounded.  This condition will effort the estimate of sets of potential $V(x)$ and the expression  of $(\theta,\gamma,\alpha,N)$-normal form.  
     Moreover, if $P^{w_{\theta}}(u,\bar{u})$ has $(\beta,\theta)$-type symmetric coefficients, then   $P^{(r,w_{\theta})}(u,\bar{u})$ is still of $(\beta,\theta)$-type symmetric coefficients.

      Second step:
The solution to the new Hamiltonian system  satisfies
\begin{equation}\label{eq1} \frac{d\|u\|^2_p}{dt} =\{ \|u\|_p^2, H^{(r,w_{\theta})}(u,\bar{u}) \}_{w_{\theta}}.
   \end{equation} 
From above equation, it is obvious that estimating  $ \{\|u\|^2_p, H^{(r,w_{\theta})}(u,\bar{u}) \}_{w_{\theta}}$ is the key to get  a long time behavior of the solution.   For a general function $f(u,\bar{u})$ with unbounded coefficients, 
$ \{\|u\|^2_p, f(u,\bar{u}) \}_{w_{\theta}}$ is not bounded even if $ \|u\|_p$ is small enough. 
Fortunately,  $ P^{(r,w_{\theta})}(u,\bar{u})$ has  $(\beta,\theta)$-type symmetric coefficients semi-bounded by $C_{\theta}>0$. 
Studying the Possion bracket of $\|u\|_p^2$ and $P^{(r,w_{\theta})}(u,\bar{u})$ with $(\beta,\theta)$-type symmetric coefficients is an important problem in this paper.  Proposition \ref{2.1} in section 4 and Lemma \ref{6.2} in section 5 state that 
       \begin{equation}\label{eq}|\{ \|u\|_p^2,{\cal R}^{N(r,w_{\theta})}(u,\bar{u})+{\cal R}^{T(r,w_{\theta})}(u,\bar{u})\}_{w_{\theta}}|\prec R^{r+1}
       \end{equation}  and 
       \begin{equation}\label{eq2}|\{ \|u\|_p^2,{Z}^{(r,w_{\theta})}(u,\bar{u})\}_{w_{\theta}}|\prec R^{r+1}
       \end{equation}
       hold true for any $\|u\|_p\leq R\ll1$ and large enough $N$.  
With the help of (\ref{eq}), (\ref{eq2}) and (\ref{eq1}), the long time behavior of solution to the new Hamiltonian system   can be obtained.

 Since the two Hamiltonian DNLS equations have some difference, there still are some differences in the results of existence of almost global solution. The main difference is the sets of the potentials.    From Lemma \ref{6.1}, there exist positive measure subsets $ \tilde{\Theta}^{\theta}_m \subset\Theta^{\theta}_m$ ($\theta\in\{0,1\}$) such that when   $V\in \tilde{\Theta}^{\theta}_m$,  frequencies $ \{\omega^{w_{\theta}}_j(V)\}$ are $(\theta,\gamma,\alpha,N)$-non resonant (see definition \ref{5.3}). 
   When $V\in \Theta_m^0$,  its Fourier coefficients  satisfy $\hat{V}_j=\hat{V}_{-j}\in\mathbb{R},$ for any $ j\in\mathbb{Z},$ which makes the corresponding frequencies satisfying $\omega^{w_0}_j=\omega^{w_0}_{-j}$; while $V\in \Theta_m^1$, $\hat{V}_j$  does not  always equal to $\hat{V}_{-j}$. Thus $\omega^{w_1}_j$  is not related to  $\omega^{w_1}_{-j}$  for any $j\in\mathbb{N}$. The potential sets to equations (\ref{in-1}) and (\ref{in-2}) are different,   because   (\ref{in-1}) and (\ref{in-2}) have different symplectic forms and nonlinearities. To be specific, from the definitions of symplectic structures, the following equations hold  true  for any $j\in\mathbb{Z}\setminus \{0\}$ 
\begin{equation}\label{w0}
\{ u_j\bar{u}_{-j}+\bar{u}_ju_{-j}, \|u\|_p^2\}_{w_0}=0
\end{equation}  and 
\begin{equation}
\label{w1}
\{ u_j\bar{u}_{-j}+\bar{u}_ju_{-j}, \|u\|_p^2\}_{w_1}\neq0.
\end{equation}
 In order to make $|\{H^{w_{\theta}}(u,\bar{u}),\|u\|_p^2\}_{w_{\theta}}|$ being  high order small as $\|u\|_p$ is small,  
when $\theta=0$, from (\ref{w0}) the terms depending on  $\big((u_j\bar{u}_{-j}+\bar{u}_ju_{-j})\big)_{j\in\mathbb{Z}}$ in $H^{w_{\theta}}(u,\bar{u})$  will not need to be eliminated by  symplectic transformations; when $\theta=1$, from (\ref{w1}) it  needs to eliminate the terms depending on  $\big( (u_j\bar{u}_{-j}+\bar{u}_ju_{-j})\big)_{j\in\mathbb{Z}}$ in $H^{w_{\theta}}(u,\bar{u})$. Therefore,  it needs more parameters in the case $\theta=1$ than in the case $\theta=0$ and the sets of potential  $V(x)$ are different to equations (\ref{in-1}) and (\ref{in-2}).

\vspace{6pt}

 The paper is organized as follows:
 The section 2 of this paper is devoted to  introduction of two types of Hamiltonian  DNLS equations with respect to different symplectic forms.  There are many differences between these  two types of equations (see Remark \ref{2.33}). Then I give the main results in this paper, the  existence  of global solutions with small initial values  to these two types of DNLS equations (See Theorem \ref{T1-1} and Theorem \ref{Th1}).

 In the third section I present a definition of $(\beta,\theta)$-type symmetric($\theta\in\{0,1\}$). Using this definition, one can describe the coefficients  of nonlinearities of two types of Hamiltonian DNLS  equations under Fourier transformation. The long time stability  result to infinite dimensional  Hamiltonian systems  owing $(\beta,\theta)$-type symmetric coefficients is given in Theorem \ref{T22}.  Theorem \ref{T1-1} and Theorem \ref{Th1} 
 follow from Theorem \ref{T22}.  

In the fourth section I give two main estimates.    One is  the  estimate of the ${\cal H}^{p-1}$ norm of Hamiltonian vector field of a polynomial $f^{w_{\theta}}(u,\bar{u})$ with $(\beta,\theta)$-type symmetric coefficients under symplectic form $w_{\theta}$ ($\theta\in\{0,1\}$). This estimate is given in Proposition \ref{cor2}.  It is easy to found that  the Hamiltonian vector field of $f^{w_{\theta}}(u,\bar{u})$ is unbounded.  Proposition \ref{2.1}  states that  $|\{f^{w_{\theta}}(u,\bar{u}),\|u\|_p^2\}_{w_{\theta}}|$ is small  when $\|u\|_p$ is small enough and  $f^{w_{\theta}}(u,\bar{u})$ has $(\beta,\theta)$-type symmetric coefficients. Even if  the set of momentum of $f^{w_{\theta}}(u,\bar{u})$ is unbounded, the result still holds true.  The property of having $(\beta,\theta)$-type symmetric coefficients  is  invariant under some operators, such as  truncated  operators $\Gamma^{N}_{\leq 2}$ and $\Gamma^{N}_{>2}$ defined in (\ref{trun1}) and (\ref{trun2}). See Corollary \ref{4.3}.

In the fifth section,  in order to improve the order of Birkhoff  normal forms of Hamiltonian systems under two different symplectic forms, I will find suitable bounded symplectic  transformations (See Theorem \ref{Th2}). These transformations are constructed by solving Homological equations. Since the nonlinear vector fields of Hamiltonian systems are unbounded (see Proposition \ref{cor2}), a stronger non-resonant condition (see definition  \ref{5.3}) is needed.  Under these transformations,  new Hamiltonian systems  are obtained. 
The nonlinearities of the new  Hamiltonian functions still have $(\beta,\theta)$-type symmetric coefficients (See Lemma \ref{sf}). Although the high order normal forms  $Z^{(r,w_{\theta})}(u,\bar{u}) \ (\theta\in\{0,1\})$ in the new Hamiltonian functions are not standard Birkhoff normal forms, from Lemma \ref{6.2} $\{\|u\|_p^2, Z^{(r,w_{\theta})}(u,\bar{u})\}_{w_{\theta}}$ is high order small when $\|u\|_p$ is small enough.  The detail of the proof of Theorem \ref{Th2} is listed in Appendix.

In the sixth section Theorem \ref{T22} is proved by applying Theorem \ref{Th2}, Proposition \ref{2.1}, Corollary \ref{4.3} and Lemma \ref{6.2}, . 

In the seventh section the proofs of  Theorem \ref{T1-1} and Theorem \ref{Th1} are given.  Using Lemma \ref{6.1}, there exists a  positive measure subset $\tilde{\Theta}_m^{\theta} \subset\Theta_m^{\theta}$ ($\theta\in \{0,1\}$) such that  when $V\in \tilde{\Theta}_m^{\theta}$ the eigenvalues of linear operator $\partial_{xx}+V(x)\ast$  are stronger non resonant.

\section{Hamiltonian DNLS equations and main results}

\subsection{ Hamiltonian DNLS equations}
\noindent
  Let
 $$H^p([0,2\pi],\mathbb{C}):=\left\{\ \psi\in L^2([0,2\pi],\mathbb{C}) \ \bigg|\ \frac{\partial^r \psi}{\partial x^r} \in L^2([0,2\pi],\mathbb{C}), \ \forall \ 0\leq r\leq p\ \right\} $$
 be a p-Sobolev space.   The inner product  of the space $L^2([0,2\pi],\mathbb{C})$ is defined as
   $$ \langle
 \zeta, \eta\rangle:= {\bf Re}\int_0^{2\pi}\zeta\cdot\bar{\eta}\ \mathrm dx ,\quad \mbox{for any } \zeta,\ \eta\in L^2([0,2\pi],\mathbb{C}).$$
The important definition of Hamiltonian PDEs is introduced in \cite{ku}. I list it as following.
  Consider  an  evolution equation
 \begin{equation}\label{D1}
 \dot{ \xi}=A\xi+ f(\xi )
 \end{equation}
defined in  symplectic Hilbert scales $(\{H^p([0,2\pi],\mathbb{C})\times H^p([0,2\pi],\mathbb{C})\}, \alpha )$, where $\alpha  $ is a non-degenerate closed 2-form.
\noindent
Equation  (\ref{D1}) is called   a
 {\bf Hamiltonian equation}, if there exists a Hamiltonian function $H(\xi)$ defined in  a domain   $O_p\subset H^p([0,2\pi],\mathbb{C})\times H^p([0,2\pi],\mathbb{C})$   making
  $$\alpha (A\xi+f(\xi), \eta)=-\langle dH(\xi), \eta\rangle \quad\mbox{for any } \xi\in O_p, \ \eta\in TO_p  \ (TO_p \mbox{ is the tangent space of } O_p). $$
The dual space and the tangent space  of $H^p([0,2\pi],\mathbb{C})\times H^p([0,2\pi],\mathbb{C})$ are isometry to  $H^p([0,2\pi],\mathbb{C})\times H^p([0,2\pi],\mathbb{C})$, without  confusion I denote them in the same signal in the following content.
  
\vspace{4pt}  
Denote $d_{A}$ as the {\it order} of the linear operator
$$A:H^p([0,2\pi],\mathbb{C})\times H^p([0,2\pi],\mathbb{C}) \rightarrow H^{p-d_A}([0,2\pi],\mathbb{C})\times H^{p-d_A}([0,2\pi],\mathbb{C})$$
 and $d_f$ as the {\it order} of the mapping
  $$ f:H^p([0,2\pi],\mathbb{C})\times H^p([0,2\pi],\mathbb{C}) \rightarrow H^{p-d_f}([0,2\pi],\mathbb{C})\times H^{p-d_f}([0,2\pi],\mathbb{C}) .$$

 When the nonlinearity of a partial differential equation includes derivative, the corresponding order of the nonlinear vector field is positive. Otherwise, the order is non-positive.
     The following notations ``bounded" and ``unbounded" are given by the signs of the order of the vector field, and readers can refer \cite{l-y}, \cite{z-y}, \cite{z-y-1}.  For the sake of reference I list the definitions again. 
\begin{Definition}
     If   $d_f\leq 0$, call $f$ in (\ref{D1}) {\bf bounded}; If $d_f> 0$, $f$ is called {\bf unbounded}.
 Moreover, If $d_A-1=d_f>0$,  call $f$ {\bf critical unbounded}.
\end{Definition}

In this paper, I focus on  two kinds of Hamiltonian Derivative Nonlinear Schr\"odinger (DNLS)   equations.

\vspace{8pt}
\noindent
{\bf Type I}--DNLS equation has the following form  
\begin{equation}\label{2.23}
{\bf i}{\psi}_t=\partial_{xx}\psi+ V* \psi +{\bf i
}f(x,\psi,\bar{\psi},\psi_x,\bar{\psi}_x),\quad \psi\in H^p([0,2\pi],\mathbb{C})
\end{equation}
  under periodic boundary condition
    \begin{equation}\label{ex1}
\psi(x,t)=\psi(x+2\pi,t),\quad 
    \end{equation}
     where $V$ belongs to
     {\footnotesize \begin{equation}\label{vset}  \Theta^0_m:=\left\{
     V(x)=\sum_{j\in\mathbb{Z}}\widehat{V}_j e^{{\bf i}j x}
\in L^2([0,2\pi],\mathbb{R}) \  \bigg|\     v^{w_0}_j:=\widehat{V}_j\langle j\rangle^m\in [-\frac12,\frac12],\   v^{w_0}_j=v^{w_0}_{-j},\ \forall\ j\in \mathbb{Z}
         \right\}
     \end{equation} } with $m>1/2$, $\langle j\rangle:=\max\{1,\ |j|\}$ and $\bar{\psi}$ is the complex conjugate of $\psi.$

\noindent
Suppose that  there exists a function $F(x,\psi,\bar{\psi})$ such that 
\begin{equation}
f(x,\psi,\bar{\psi},\psi_x,\bar{\psi}_x)=\frac12\partial_x\partial_{\bar{\psi}} F(x,\psi,\bar{\psi})+
\partial_{\psi\bar{\psi}}F(x,\psi,\bar{\psi})\psi_x.
\end{equation}
Moreover, $F$ satisfies assumptions as follows.  
           \begin{description}
       \item[${\bf A_1}:$] $F(x,\xi, \eta)$
     is analytic  about $(\xi, \eta)$ in a  neighborhood of the origin and  satisfies
\begin{equation}\label{gm}
\overline{F(x,\psi,\bar{\psi})}=F(x,\psi,\bar{\psi})\ 
\end{equation}and $F(x,\psi,\bar{\psi})$ vanishes at least at order 2 in $(\psi,\bar{\psi})$ at the origin.   
\item[${\bf A_2}$] For any fixed $(\psi,\bar{\psi})$ a  neighborhood of the origin, $F \in H^{\beta+1}([0,2\pi],\mathbb{C})$ ($\beta$ is a big enough positive real number) satisfies 
$$ F(x+2\pi,\psi,\bar{\psi})=F(x,\psi,\bar{\psi}).$$
     \end{description}
     Then (\ref{2.23})  becomes a Hamiltonian PDE  with a real value Hamiltonian function
     \footnote{Since $F(x,\psi,\bar{\psi})$ satisfies assumptions ${\bf A_1}$-${\bf A_2}$, the following equation holds true for any  $\psi \in H^{p}([0,2\pi],\mathbb{C}) $ fulfilling $\psi(x+2\pi,t)=\psi(x,t)$
     \begin{eqnarray*}
0=\int_{0}^{2\pi} \frac{d F}{d x}(x,\psi,\bar{\psi})d x= \int_0^{2\pi}    \partial_x F(x,\psi,\overline{\psi}) +\partial_{\psi} F
 (x,\psi,\overline{\psi}){\psi}_x +\partial_{\bar{\psi}} F
 (x,\psi,\overline{\psi})\bar{\psi}_x dx,
\end{eqnarray*}     
i.e.,
\begin{eqnarray}\label{add1}
  \int_0^{2\pi}   \frac12 \partial_x F(x,\psi,\overline{\psi}) +\partial_{\psi} F
 (x,\psi,\overline{\psi}){\psi}_x  dx= -\int_0^{2\pi}   \frac12 \partial_x F(x,\psi,\overline{\psi}) 
 -\partial_{\bar{\psi}} F
 (x,\psi,\overline{\psi})\bar{\psi}_x dx.
\end{eqnarray} 
    From (\ref{add1}) and assumptions ${\bf A_1}$-${\bf A_2}$, it follows   
\begin{eqnarray*}
 \overline{\int_0^{2\pi}  {\bf i}\frac12\partial_x F(x,\psi,\overline{\psi}) + {\bf i}\partial_{\psi} F
 (x,\psi,\overline{\psi})\psi_x dx}&= &\int_0^{2\pi}  (-{\bf i})\frac12\partial_x F(x,\psi,\overline{\psi}) - {\bf i}\partial_{\bar{\psi}} F
 (x,\psi,\overline{\psi})\bar{\psi}_x dx\\
 &=&\int_0^{2\pi}  {\bf i}\frac12\partial_x F(x,\psi,\overline{\psi}) + {\bf i}\partial_{\psi} F
 (x,\psi,\overline{\psi})\psi_x dx,
\end{eqnarray*}      
which means that the Hamiltonian  function $H_{(\ref{2.23})}(\psi,\bar{\psi})$    is real.   }
\begin{equation}\label{h.1}
 H_{(\ref{2.23})} (\psi,\bar{\psi})=\int_0^{2\pi} -|\partial_{x} \psi|^2+(V* \psi )\bar{\psi} + {\bf i}\frac12\partial_x F(x,\psi,\overline{\psi}) + {\bf i}\partial_{\psi} F
 (x,\psi,\overline{\psi})\psi_xdx 
 \end{equation} on symplectic space $(H^{p}([0,2\pi],\mathbb{C})\times H^p([0,2\pi],\mathbb{C}),$ $w^0)$, where
 \begin{equation}\label{omega1}
 w^0 = J_0 d \psi \wedge d \bar{\psi},\quad  J_0^{-1}= \left(
                                                                \begin{array}{cc}
                                                                              0 & -{\bf i} \\
                                                                              {\bf i} & 0 \\
                                                                            \end{array}
                                                                          \right).
        \end{equation}
 The corresponding  Hamiltonian
vector of $H_{(\ref{2.23})}(\psi,\bar{\psi})$ under symplectic form $w^0$ is
$$ X^{w^0}_{H_{(\ref{2.23})}}:=\left(- {\bf i}\partial_{\bar{\psi}} H_{(\ref{2.23})},\  {\bf i}\partial_{{\psi}} H_{(\ref{2.23})} \right)^{T}$$ and  the equation (\ref{2.23}) can be written as
\begin{equation}
\left\{
\begin{array}{ll}
\dot{\psi}  = -{\bf i} \partial_{\bar{\psi}}  H_{(\ref{2.23})}(\psi,\bar{\psi}),\\
\\
\dot{\bar\psi} =\ {\bf i} \partial_{\psi} H_{(\ref{2.23})}(\psi,\bar{\psi}).
\end{array}
\right.
\end{equation}

\vspace{10pt}

\noindent
\vspace{8pt}
 {\bf Type II}--DNLS equation has the form as following 
 \begin{equation}\label{NLS-1}
{\bf i}{\psi}_t=\partial_{xx}\psi+ V* \psi + {\bf i}\partial_{x}\big(\frac{\partial F(x,\psi,\bar{\psi})}{\partial\bar{\psi}}\big)
\end{equation}
defined on \begin{equation}\label{H_0}
H_0^{p}([0,2\pi],\mathbb{C}):=\left\{ \psi\in H^p([0,2\pi],\mathbb{C})\  \bigg|\  \int_0^{2\pi} \psi(x,t) dx=0 \right\}
\end{equation} 
  under periodic boundary condition
    \begin{equation}\label{ex1-0}
\psi(x,t)=\psi(x+2\pi,t).
    \end{equation}
      The  potential $V$ belongs to  
     \begin{equation}\label{vset1}  \Theta^{1}_m:=\left\{
     V\in L^2([0,2\pi],\mathbb{C})  \bigg|\
     v^{w_1}_j:=\widehat{V}_j\langle j\rangle^m\in [-\frac12,\frac12],\  \forall\ j\in \mathbb{Z}\setminus \{0\}, \ v^{w_1}_0=\widehat{V}_0=0
     \right\}
     \end{equation} with $m>1/2.$

\noindent
If equation (\ref{NLS-1}) satisfies the following assumptions:

\begin{description}
  \item[${\bf B_1}$]    $F(x,\xi,\eta)$ is analytic at the origin about $(\xi, \eta)\in H_0^{p}([0,2\pi],\mathbb{C})\times H_0^{p}([0,2\pi],\mathbb{C})$ and vanishes at least at order 2 in $(\psi,\bar{\psi})$ at origin. For any $\psi\in H_0^{p}([0,2\pi],\mathbb{C})$, it holds
      $$\overline{F(x,\psi,\bar{\psi})}=F(x,\psi,\bar{\psi}) .$$
        \item[${\bf B_2}$]
  For any fixed $(\psi,\bar{\psi})$ in a neighborhood of the origin,      $F \in H^{\beta+1}([0,2\pi],\mathbb{C})$ ($\beta$ is a big enough positive real number) satisfies 
  $$ F(x+2\pi,\psi,\bar{\psi})=F(x,\psi,\bar{\psi});$$
         \end{description}
then equation (\ref{NLS-1}) becomes into  a Hamiltonian PDE with  a real Hamiltonian
 $$H_{(\ref{NLS-1})}(\psi, \bar{\psi})=\int_{0}^{2\pi}-{\bf i} \partial_{x}\psi   \bar{\psi}-{\bf i}( \partial_{x})^{-1}( V(x)\ast \psi)\cdot \bar{\psi}+F(x,\psi,\bar{\psi})\mathrm dx  $$
under  symplectic space $ ( H_0^{p}([0,2\pi],\mathbb{C})\times H_0^p([0,2\pi],\mathbb{C}), w^1 )$, where
 \begin{equation}\label{omega2}
 w^1 := J_1 d \psi \wedge d \bar{\psi},\quad  ( J_1)^{-1}:= \left(
  \begin{array}{cc}
                 0 & \partial_x \\
                 \partial_x & 0 \\
                      \end{array}
                   \right)
        \end{equation} is a symplectic  from ($w^1$ is a non-degenerate closed two form in space $ H_0^{p}([0,2\pi],\mathbb{C})\times H_0^p([0,2\pi],\mathbb{C})$).

     \noindent  The Hamiltonian
vector $X^{w^1}_{H_{(\ref{NLS-1})}}$ of $H_{(\ref{NLS-1})}(\psi, \bar{\psi})$ equals to
$$ \left( \partial_{x}(\partial_{\bar{\psi}} H_{(\ref{NLS-1})}), \partial_{x}(\partial_{{\psi}} H_{(\ref{NLS-1})})\right)^{T}.$$
 Equation (\ref{NLS-1}) can be written as follow
$$ \left\{\begin{array}{ll}
\dot{\psi}=\partial_x  \partial_{\bar{\psi}}H_{(\ref{NLS-1})}(\psi,\bar{\psi}),\\
\\
\dot{\bar{\psi}}=\partial_x \partial_{\psi}H_{(\ref{NLS-1})}(\psi,\bar{\psi}).
\end{array}\right. $$
   The DNLS equation researched in \cite{z-y} is a special case of equation (\ref{NLS-1}), i.e., $F(x,\psi,\bar{\psi})$ is independent of $x$.

\begin{Remark}\label{2.33}
There are some differences between type I-DNLS equations and type II-DNLS equations:
\begin{itemize}
\item Nonlinearities of these two kinds of DNLS equations are different. It is an essential difference.
\item Symplectic spaces are different.  Type I-DNLS equation is defined in  $(H^{p}([0,2\pi],\mathbb{C})\times H^p([0,2\pi],\mathbb{C}),$ $w^1)$ and type II-DNLS equation is defined in $(H_0^{p}([0,2\pi],\mathbb{C})\times H_0^p([0,2\pi],\mathbb{C}),$ $w^0)$. $w^1$ and $w^0$ are also different.
    \item The Potential $V$ in  type I-DNLS equation belongs to $\Theta^0_m$ and the one in type II-DNLS equation belongs to $\Theta^1_m$. $\Theta^1_m$ is different from  $\Theta^0_m$. When $V\in \Theta^0_m$ it  fulfills $\hat{V}_j=\hat{V}_{-j}\in \mathbb{R}$ which means $\overline{V(x)}=V(x)$; while $V\in \Theta^1_m$, $V$ is a complex valued potential.
          The potential $V$ will directly determine the eigenvalues of the linear operator $-\partial_{xx} +V(x)\ast $. It is clear that  when $V\in \Theta^0_m$ the corresponding eigenvalues $\omega_j$ and $\omega_{-j}$ of the linear operator $-\partial_{xx} +V(x)\ast $ are resonant, when $V\in \Theta^1_m$, they are independent. The measures of   $\Theta^{0}_{m}$ and $\Theta^{1}_{m}$ are  defined as follows
$$ \mbox{meas}(\Theta^{0}_{m}) :=\prod_{j\in\mathbb{N}\cup\{0\}} \mbox{meas}\big\{v_j\in[-\frac12,\frac12]\ \big|\ V \in \Theta^{1}_{m},\ \ \widehat{V}_j\langle j\rangle^m=v_j  \big\} $$
and
$$ \mbox{meas}(\Theta^{1}_{m}) :=\prod_{j\in\mathbb{Z}\setminus\{0\}} \mbox{meas}\big\{v_j\in[-\frac12,\frac12]\ |\ V \in \Theta^{0}_{m},\ \widehat{V}_j |j|^m=v_j  \big\} ,$$
where ``$\mbox{meas}$" means the Lebesgue measure.
\end{itemize}
\end{Remark}

\begin{Remark}

\noindent
\begin{itemize}
  \item If type II-DNLS equation (\ref{NLS-1}) is defined in space $H^{p}([0,2\pi],\mathbb{C})$, it is easy to verify that  for any solution $\psi(x,t)$ to equation (\ref{NLS-1}) the quantity $\int_{0}^{2\pi} \psi(x,t) dx$ is a constant for any $t\in\mathbb{R}$. Set
       $$ \phi:=\psi-c,\quad c:=\int_{0}^{2\pi} \psi(x,t) dx=\int_{0}^{2\pi} \psi(x,0) dx. $$ If $ \psi\in H^p([0,2\pi],\mathbb{C}),$ then   $ \phi\in H_0^p([0,2\pi],\mathbb{C}).$
When  $G(x,\phi,\bar{\phi}):= F(x,\phi+c,\overline{\phi+c})$ satisfies ${\bf B_1}$  and ${\bf B_2}$, then equation (\ref{NLS-1}) becomes into a Hamiltonian equation under symplectic form $w^1$ about $(\phi,\bar{\phi})$.
  \item From  Proposition \ref{cor2} in section 4,  the nonlinearities of type I-DNLS equation (\ref{2.23}) and type II-DNLS equation (\ref{NLS-1})  are  unbounded.

\end{itemize}

\end{Remark}

\subsection{Main result}

The long time behavior of the solutions around equilibrium point  to type I and type II-DNLS Hamiltonian equations are given in this subsection.

\begin{Theorem}\label{T1-1}
Suppose that the equation (\ref{2.23}) satisfies assumptions ${\bf A_1}$-${\bf A_2}$.  For any integer $r_*>1$, there exist an almost full   measure  set $\widetilde{\Theta}^0_m \subset\Theta^0_m$ and $p_*>0$ such that for any fixed
$V \in  \widetilde{\Theta}^0_m $ and any $p$ fulfilling $ (\beta-4)/2>p>p_{*}$,   if
the initial data of the solution  to (\ref{2.23}) satisfies
$$\|\psi(x,0)\|_{H^p([0,2\pi],\mathbb{C})}\leq \varepsilon<\varepsilon_{*},$$
then one has
$$ \|\psi(x,t)\|_{H^p([0,2\pi],\mathbb{C})}<2\varepsilon, \qquad \forall~|t| \prec  \varepsilon^{-r_*-1}.$$
\end{Theorem}

\begin{Theorem}\label{Th1}
 Suppose that equation (\ref{NLS-1}) fulfills assumptions ${\bf B_1}$-${\bf B_2}$.
    For any integer $r_*>1$, there exist a positive  $p_*$ and an almost full  measure  set $\widetilde\Theta^1_m\subset \Theta^1_m$ ($m>\frac12$) such that for any fixed
$V\in \widetilde{\Theta}^1_m$ and  any $p$ fulfilling $ (\beta-4)/2>p>p_{*}$,  
the solution  to (\ref{NLS-1}) satisfies 
$$ \|\psi(x,t)\|_{H_0^{p+1/2}([0,2\pi],\mathbb{C})}<2\varepsilon, \qquad \mbox{for any } \ |t|   \prec \varepsilon^{-r_*-1},$$
  if
the initial value fulfills 
$$\|\psi(x,0)\|_{H_0^{p+1/2}([0,2\pi],\mathbb{C})}<\varepsilon\ll1.$$
\end{Theorem}

\begin{Remark}
As  type I and type II-DNLS equations have many differences (Readers can refer Remark \ref{2.33}), the  proofs of Theorem \ref{T1-1} and  Theorem \ref{Th1}    still have some differences. 
\end{Remark}

\begin{Remark}
 The DNLS equations researched in our paper are not always invariant under gauge transformation.
For example, take
 \begin{equation}\label{f-t}
F(\psi,\bar{\psi})=\psi^3\bar{\psi}+\psi \bar{\psi}^3
 \end{equation} in equation (\ref{NLS-1}), which fulfills assumption ${\bf B_1}$. It is easy to check that  equation (\ref{NLS-1}) with $F$ in (\ref{f-t}) is not invariant under the transformation   $\phi= e^{{\bf i} \theta}\psi$, $\theta\in \mathbb{R}$.
\end{Remark}

\section{Long time stability result to infinite dimension Hamiltonian  systems with $(\beta,\theta)$-type symmetric coefficients ($\theta\in\{0,1\}$)}
\noindent
\subsection{$(\beta,\theta)$-type symmetric coefficients ($\theta\in\{0,1\}$)}
Under Fourier transformation, Hamiltonian DNLS equations with respect to periodic boundary condition can be transformed  into two classes of infinite dimension Hamiltonian systems with ``unbounded" nonlinearities. 
In this section, I will
 introduce long time stability results to two classes of infinite dimension Hamiltonian  systems with ``unbounded" nonlinearities. 
First, give some notations and annotations.
In this paper,    $ {\mathbb Z}^*$ means  $ 
 {\mathbb Z} $ or $
 {\mathbb Z} \setminus\{0\}. $
Denote  weighted Hilbert spaces
 $$ \ell^2_p(\mathbb Z^*,\mathbb C):=\bigg\{u\in \ell^2(\mathbb Z^*,\mathbb C)\ \bigg| \  \|u\|_{p }^2:
 = \sum_{j\in \mathbb Z^*}\langle j\rangle^{2p}\cdot |u_j|^2 <+\infty ,\ \langle j\rangle:
 =\max\{ |j|,\ 1\}\bigg\} ,\ $$
and $ {\mathcal H^p}(\mathbb Z^*,\mathbb C):=\{ (u,v)\in\ell^2_p(\mathbb Z^*,\mathbb C)\times \ell^2_p(\mathbb Z^*,\mathbb C)\ \big|\ v=\bar{u}\ \}$ with norm  
\begin{equation*}
  \|(u,\bar{u})\|_{p }:=\sqrt{\|u\|^2_{p }+\|\bar{u}\|^2_{p }}.
\end{equation*}
Let the neighborhood of the origin with a radius $R$ be noted by
$$B_p(R):=\left\{(u,\bar{u})\in {\cal H}^p(\mathbb Z^*,\mathbb C) \ \big| \ \|(u,\bar{u})\|_{p}<R \right\}.$$
\begin{Definition}
 For any fixed $l,\ k\in\mathbb{N}^{\mathbb{Z}^*},$ 
 call the integer  $$ \sum_{j\in \mathbb Z^*} j (l_j-k_j)$$
    be the {\bf momentum} of the ordered vector $(l,k)$ and denote it as ${\cal M}(l,k).$
 \end{Definition}
Readers can refer this definition in \cite{z-y} and \cite{z-y-1}. 

 \begin{Remark}\label{M}
If ${\cal M}(l,k)=i$, from the definition of momentum, it holds that 
$$ {\cal M}(k,l)=-i.$$
\end{Remark}
\noindent
\begin{Definition}
 Call a   power series  $$f(u,\bar{u})=\sum\limits_{t\geq 3} \sum\limits_{|k+l|=t,  l,k\in\mathbb{N}^{\mathbb Z^*}\atop{{\cal M}(l,k)=i\in M_{f_t}\subset\mathbb{Z}} } f^i_{t,lk} u^{l}\bar{u}^k, \quad (u,\bar{u})\in {\cal H}^p(\mathbb Z^*,\mathbb C) $$  have
{\bf   symmetric  coefficients}, if
for any $l,k$ fulfilling  $ |l+k|=t $ and ${\cal M}(l,k)=i$, the coefficient holds 
$$ \overline{f^i_{t,lk}} =f^{-i}_{t,kl}.$$
Moreover, fixed $\beta>0$, call $f(u,\bar{u})$ has $\beta$-{\bf bounded coefficients bounded by $C>0$}, if  
 $$ |f^i_{t,lk}| \leq  \frac{C^{t-2}}{ \langle i\rangle^{\beta}},$$
 for any $l,k$ satisfying $|l+k|=t$ and ${\cal M}(l,k)=i$. 
\end{Definition}

\begin{Remark}
A power series $f:  {\cal H}^p(\mathbb Z^*,\mathbb C)\rightarrow \mathbb{C}$  is of symmetric coefficients, if and only if $f$ satisfies  
$$ \overline{f(u,\bar{u})}=f(u,\bar{u}),\quad \mbox{for any } (u,\bar{u})\in{\cal H}^p(\mathbb Z^*,\mathbb{C}).$$ Hence, a real-value Hamiltonian function has symmetric coefficients. 
 
\end{Remark}

 Now  define two kinds of power series with  ``unbounded" special symmetric coefficients.

\begin{Definition}\label{3.4}Given $\beta>0$ and $C_f>0$, call   a power series  $$f(u,\bar{u})=\sum\limits_{t\geq 3}\   \sum\limits_{|k+l|=t, l,k\in\mathbb{N}^{\mathbb Z^* }\atop{{\cal M}(l,k)=i\in M_{f_t}\subset\mathbb{Z} } } f^i_{t,lk} u^{l}\bar{u}^k, \quad (u,\bar{u})\in {\cal H}^p(\mathbb Z^* ,\mathbb C) $$   have {\bf $(\beta,0)$-type symmetric coefficients}, if its coefficients have  the following  form
 $$ f^i_{t,lk}:=\sum_{( {l}^0,{k}^0,i^0)\subset {\cal A}_{f^i_{t,lk} }}\  {f}^{i( {l}^0,{k}^0,i^0)}_{t,lk}\big({\cal M}( {l}^0, {k}^0)-\frac{i^0}2\big) ,$$
    where
    $${\cal A}_{f^i_{t,lk}}\subset \{(\tilde{l},\tilde{k},\tilde{i})  \ |\ 0\leq\tilde{l}_j\leq l_j,\ 0\leq\tilde{k}_j\leq k_j,\ \mbox{for any }j\in\mathbb{Z}^*,\  (\tilde{l},\tilde{k})\in\mathbb{N}^{{\mathbb Z^*}}\times \mathbb{N}^{{\mathbb Z^*}},\   \tilde{i}\in\mathbb{Z}\},$$  and   for any
        $(l^0,k^0,i^0)\in {\cal A}_{f^i_{t,lk}}$,   the followings hold true 
 $$(k- k^0,l-{l}^0, i^0-2i)\in{\cal A}_{f^{-i}_{t,kl}},\quad \overline{f^{i(l^0,k^0,i^0)}_{t,lk}}=f^{-i(k-k^0, {l}-l^0, i^0-2i)}_{t,kl}.$$
     Moreover, call $f(u,\bar{u})$ have  {\bf $(\beta,0)$-type symmetric coefficients   semi-bounded by $C_f$}, if  $f(u,\bar{u})$ have  $(\beta,0)$-type symmetric coefficients  and  there exists a constant $C_f>0$ such that for any $l,k\in\mathbb{N}^{\mathbb Z^*}$ with $|l+k|=t $ and ${\cal M}(l,k)=i$, the following inequality holds true 
      \begin{eqnarray}\label{3.1---0}
    \sum_{(l^0,k^0,i^0)\in {\cal A}_{f^i_{t,lk}}} |f^{i(l^0,k^0,i^0)}_{t,lk}|\cdot\max\{ \langle i^0\rangle, \langle i^0-2i\rangle \}  \leq \frac{ C_{f}^{t-2}}{\langle i\rangle^{\beta}}. 
    \end{eqnarray}

  \end{Definition}
Suppose that Type I-DNLS equation satisfies assumption ${\bf A_1}$-${\bf A_2}$. 
Under Fourier transformation, there exists a constant $C>0$ such that the Hamiltonian function  of Type I-DNLS equation have {\bf $(\beta,0)$-type symmetric coefficients semi-bounded by $C$}. See section 7 for details. This symmetric property  is  invariant  under a symplectic transformation.  Refer Lemma \ref{sf} in section 5.  

 \begin{Definition}\label{3.5}Given $\beta>0$ and $C_g>0$,
call a power series
$$ g(u,\bar{u})=\sum\limits_{t\geq 3}\   \sum\limits_{|k+l|=t,\  l,k\in\mathbb{N}^{\mathbb Z^*}\atop{ {\cal M}(l,k)=i\in M_{g_t}\subset\mathbb{Z} }} g^i_{t,lk} u^{l}\bar{u}^k $$ have {\bf $(\beta,1)$-type symmetric coefficients},  if  for any $l,k\in\mathbb{N}^{\mathbb Z^*}$  with $|l+k|=t$ and ${\cal M}(l,k)=i\in M_{g_t}\subset \mathbb{Z}$, its coefficient
has the following form  $$  g^i_{t,lk}:= \tilde{g}^i_{t,lk}\prod_{j\in \mathbb{Z}^*}\langle j\rangle^{\frac12 (l_j+k_j)}  $$
and satisfies 
$$ {\tilde{g}^i_{t,lk}}=\tilde{g}^{-i}_{t,kl}.$$
Moreover, call $g(u,\bar{u})$  have {\bf $(\beta,1)$-type symmetric coefficients semi-bounded by $C_g$}, if $g(u,\bar{u})$  have  $(\beta,1)$-type symmetric coefficients  and there exists a constant $G_g>0$ such that
 \begin{equation}\label{3.1---1}
  |\tilde{g}^i_{t,lk}|\leq \frac{C_g^{t-2}}{ \ \langle i\rangle^{\beta}\ }
 \end{equation} 
 for any $l,k$ fulfilling $|l+k|=t$ and ${\cal M}(l,k)=i.$

\end{Definition}
\begin{Remark}\label{Rem3.2}

 \noindent
 Suppose a power series $f^{w_{\theta}}(u,\bar{u})$ is  of $(\beta,\theta)$-type symmetric coefficients semi-bounded by $C_{\theta}>0$ ($\theta\in\{0,1\}$). Thence, 
\begin{itemize}
                \item 
                   the ``semi-bounded"  does not means the coefficients of $f^{w_{\theta}}(u,\bar{u})$ are bounded,  even if $f^{w_{\theta}}(u,\bar{u})$ is an $r$-degree polynomial;  
\item                
                  the momentum set $M_{f^{w_{\theta}}_t}$ is symmetric, i.e., if $i\in M_{f^{w_{\theta}}_t}$, then $-i\in M_{f^{w_{\theta}}_t}$. And the number of elements in $M_{f^{w_{\theta}}_t}$ satisfies
\begin{equation}\label{m-f-t}
\sharp{M_{f^{w_{\theta}}_t}}\leq 2 \max{M_{f^{w_{\theta}}_t}}+1.
\end{equation}
So does the power series having  $\beta$-bounded symmetric coefficients.
               \item \label{3.4} 
    when  $\theta=0$ for any given $(l,k)\in \mathbb{N}^{\mathbb Z^*}\times \mathbb{N}^{\mathbb Z^*}$ and $i\in \mathbb Z$,
{\footnotesize\begin{eqnarray*}
   \overline{(f^{w_{0}})^i_{t,lk}}&=&\overline{\sum_{( {l}^0,{k}^0,i^0)\subset {\cal A}_{f^i_{t,lk} }}\  (f^{w_{0}})^{i( {l}^0,{k}^0,i^0)}_{t,lk}\big({\cal M}( {l}^0, {k}^0)-\frac{i^0}2\big)}\nonumber\\
   &=& \sum_{( {l}^0,{k}^0,i^0)\subset {\cal A}_{(f^{w_{0}})^i_{t,lk} }}\  (f^{w_{0}})^{-i( k-k^0,l-{l}^0,i^0-2i)}_{t,kl}\big({\cal M}( {l}^0, {k}^0)-\frac{i^0}2\big) \nonumber\\
   &=& \sum_{( k-k^0,l-{l}^0,i^0-2i)\subset {\cal A}_{(f^{w_{0}})^{-i}_{t,kl} }}\  (f^{w_{0}})^{-i( k-k^0,l-{l}^0,i^0-2i)}_{t,kl}\big({\cal M}(k-k^0,l-{l}^0)-(\frac{i^0}2-i)\big) \nonumber\\
   &=&(f^{w_{0}})^{-i}_{t,kl},
   \end{eqnarray*}
   }
the last second equation follows from
\begin{equation*} 
  {\cal M}(l^0,k^0)-\frac{i^0}2={\cal M}(k-k^0,l-l^0)-(\frac{i^0}2-i)
  \end{equation*}
  and $(k-k^0,l-l^0, {i^0}-2i)\in {\cal A}_{f^{-i}_{t,kl}};$

  \noindent
 when $\theta=1$  for any $l,k\in\mathbb{N}^{\mathbb Z^*}$ and any $i\in\mathbb{Z}$,
   \begin{eqnarray}
   \overline{(f^{w_1})^i_{t,lk}}= \overline{(\tilde{f}^{w_1})^i_{t,lk}\prod_{j\in \mathbb{Z}^*}\langle j\rangle^{\frac12 (l_j+k_j)} } =(\tilde{f}^{w_1})^{-i}_{t,kl}\prod_{j\in \mathbb{Z}^*}\langle j\rangle^{\frac12 (k_j+l_j)}=(f^{w_1})_{t,kl}^{-i}.
   \end{eqnarray}
   Hence, the coefficients of $f^{w_{\theta}}(u,\bar{u})$ are symmetric. 
                              \end{itemize}

\end{Remark}

\noindent
Let $({\mathcal H^p}(\mathbb Z^*,\mathbb C),w_{\theta} )$ ($\theta\in\{0,1\}$)   be a symplectic space  endowed with symplectic form
 \begin{equation}\label{sp1}
w_{\theta} :=\left\{
\begin{array}{lll}
&{\bf
i}\sum_{j\in \mathbb Z^*}\mathrm du_j\wedge \mathrm d\bar{u}_j& \theta=0,\\
\\
&{\bf
i}\sum_{j\in \mathbb Z^*} \mbox{sgn}(j) \mathrm du_j\wedge \mathrm d\bar{u}_j&\theta=1.
\end{array}
\right.
\end{equation}  
When $\theta=0$, $\mathbb{Z}^*$ can be either  $\mathbb{Z}$ or $\mathbb{Z}\backslash\{0\}$. When $\theta=1$, $\mathbb{Z}^*$ means $\mathbb{Z}\backslash\{0\}$ only. 
 
  The possion bracket of differential functions $f_1$ and $f_2$ defined in the domain of ${\mathcal H^p}(\mathbb Z^*,\mathbb C)$ under the symplectic form $w_{\theta} \ (\theta\in \{0,1\})$ has the following form
\begin{equation}
\{f_1,f_2 \}_{w_{\theta}} =w_{\theta} (\nabla f_1, \nabla f_2).
\end{equation}
 Given a differential function $f$, its corresponding Hamiltonian vector field under the symplectic form $w_{\theta} $ is defined as
 \begin{eqnarray}\label{I-theta} X^{w_{\theta}}_f: = J_{\theta}  \nabla f , \quad  J_{\theta} :=\left(
                                                 \begin{array}{cc}
                                                   0 & -{\bf i}I_{\theta}    \\
                                                   {\bf i}I_{\theta}   & 0 \\
                                                 \end{array}
                                               \right),\quad \theta\in\{0,1\},
 \end{eqnarray}
 where  $I_0$ is an identity operator on space $\ell^2_p(\mathbb{Z}^*,\mathbb{C})$, and
   for any $u=(u_j)_{j\in \mathbb{Z}^*}\in \ell_p^2(\mathbb{Z}^*,\mathbb{C})$,
 $$I_1 u=\big((I_1 u)_j\big)_{j\in \mathbb{Z}^*},\quad   (I_1u)_j:=\mbox{sgn}(j)\cdot u_j,\quad j\in\mathbb{Z}^*. $$

\noindent

\subsection{Result of   Hamiltonian system with ($\beta,\theta$)-Type symmetric coefficients}
\noindent
 In order to use an uniformly formula  to describe two kinds of Hamiltonian equations, in this paper,   denote $0^0=1$.

Let $\theta\in\{0,1\}$.  Consider  Hamiltonian systems defined in  $( {\mathcal H}^p (\mathbb{Z}^*,\mathbb{C}),w_{\theta})$,   for any $j\in \mathbb{Z}^* $
\begin{equation}\label{uj}
\left\{ \begin{array}{ll}
\dot{u}_j= -{\bf i}\mbox{sgn}^{\theta}(j) \cdot\partial_{\bar u_j} H^{w_{\theta}}(u,\bar{u}),\\
\\
\dot{\bar{u}}_j=\ \ {\bf i} \mbox{sgn}^{\theta}(j)\cdot   \partial_{ u_j} H^{w_{\theta}}(u,\bar{u}),
\end{array}\right.
\end{equation}
with a Hamiltonian function
\begin{equation}\label{H-1}
H^{w_{\theta}}(u,\bar{u})=H_0^{w_{\theta}} + P^{w_{\theta}}(u,\bar{u}),
\end{equation}
where $ H_0^{w_{\theta}} :=\sum_{j\in \mathbb Z^*}   \omega^{w_{\theta}}_j|u_j|^2$.

  \begin{Theorem}\label{T22}
Suppose that  equation (\ref{uj}) satisfies the following assumptions:
 \begin{description}
   \item[${\cal A}_{\theta}$]: $\omega^{w_{\theta}}:=(\omega^{w_{\theta}}_j)_{j\in\mathbb{Z}^*}$, $\ \omega^{w_{\theta}}_j\in \mathbb{R}$.  $\omega^{w_{\theta}}$ satisfies strong non resonant condition.\footnote{See Definition \ref{5.3} in section 5.} 
   \item[${\cal B}_{\theta}$]: $P^{w_{\theta}}(u,\bar{u})$ is a  power series beginning with  at least at order 2 in $(u,\bar{u})$ and has  $(\beta,\theta)$-type symmetric coefficients semi-bounded by $C_{\theta}>0$ ($\beta$ is big enough positive number).
 \end{description}  
  Given integer $r_*>1$, there exists  an integer
$p_{r_*}>0$, for any $p$ fulfilling $(\beta-4)/2>p>p_{r_*}$ there exists $\varepsilon_{r_*,p}>0$ such that  the solution to (\ref{uj}) satisfies 
$$ \|(u(t),\bar{u}(t))\|_{p }<2\varepsilon,\quad \text{for any }\ \ |t|  \preceq \varepsilon^{-r_*-1},$$
  if  the initial data  fulfills 
$$\|(u(0),\bar{u}(0))\|_{p }<\varepsilon<\varepsilon_{r_*,p}.$$
\end{Theorem}

Let us give the basic procedure of   proving Theorem \ref{T22} which consists of the following steps. 

The first step is to construct a bounded  symplectic transformation around the origin under which the nonlinearity of Hamiltonian function (\ref{H-1}) becomes into the sum of the following three parts: one is  a high order $(\theta,\gamma, \alpha, N)$-normal form $Z^{(r_*,w_{\theta})}(u,\bar{u})$\footnote{See Definition  \ref{5.3}.}, one of the others is ${\cal R}^{T(r_*,w_{\theta})}(u,\bar{u})$ which vanishes at $r_*+3$ order of $(u,\bar{u})$ at origin  and the last one denoted as ${\cal R}^{N(r_*,w_{\theta})}(u,\bar{u})$  has zero at least order 3 about high index variable $(u_j,\bar{u}_j)_{|j|>N}$($N$ is big enough).  Moreover, when $P^{w_{\theta}}(u,\bar{u})$ in (\ref{H-1}) has $(\beta,\theta)$-type symmetric coefficients semi-bounded by $C_{\theta}>0$, the new Hamiltonian function is still of $(\beta,\theta)$-type symmetric coefficients semi-bounded by $C(\theta,r)>0$ ($C(\theta,r)$ is defined in Theorem \ref{Th2}). In order to guarantee the boundedness of the symplectic transformation, a strong non-resonant condition is presented in Definition \ref{5.3}.
See Theorem \ref{Th2} in section 5  for details. 

Since (\ref{uj}) is a Hamiltonian system, the following equation holds true  
$$\|u(t)\|_p^2-\|u(0)\|_p^2 
=\int_0^t\frac{d}{d \tau}\|{u}(\tau)\|_p^2d\tau 
=\int_0^t\{ \|{u}\|_p^2 , H^{w_{\theta}}({u},\bar{u})\}_{w_{\theta}}d \tau. $$ Researching the Possion bracket of  Hamiltonian function $H^{w_{\theta}}(u,\bar{u})$ and $\|u\|_p^2$ under the corresponding symplectic form $w_{\theta}$ is important. 
The second step is to estimate the Possion bracket of  function $f^{w_{\theta}}(u,\bar{u})$  and $\|u\|_p^2$. Suppose that $f^{w_{\theta}}(u,\bar{u})$ has $(\beta,\theta)$-type symmetric coefficients semi-bounded by $C_{\theta}>0$. If the momentum of $f^{w_{\theta}}(u,\bar{u})$ are bounded, partial result can be  found in \cite{z-y} and \cite{z-y-1}. When the set of the momentum of $f^{w_{\theta}}(u,\bar{u})$ is unbounded, the corresponding Hamiltonian vector field of $f^{w_{\theta}}(u,\bar{u})$   is small under $H^{p-1}$ norm  but not $H^p$ norm (see Proposition \ref{cor2} and Remark \ref{dj}). In order to deal with it, I will make use of the Hamiltonian structure and  $(\beta,\theta)$-type symmetric coefficients semi-bounded by $C_{\theta}>0$ to get the estimate of Possion bracket of  Hamiltonian function $H^{w_{\theta}}(u,\bar{u})$ and $\|u\|_p^2$ under the corresponding symplectic form $w_{\theta}$. 
See Proposition \ref{2.1} and Corollary \ref{4.3} for details.  By Proposition \ref{2.1} and Corollary \ref{4.3}, $|\{ {\cal R}^{T(r_*,w_{\theta})}(u,\bar{u})+ {\cal R}^{N(r_*,w_{\theta})}(u,\bar{u}),\|u\|_p^2\}_{w_{\theta}}|\prec R^{r_*+1}$ holds true for $\|(u,\bar{u})\|_p\leq R$. From Remark \ref{norm-re}, $(\theta,\gamma,\alpha,N)$-normal form $Z^{(r_*,w_{\theta})}(u,\bar{u})$ is not   a standard Birkhoff normal form. By Lemma \ref{6.2} when  $Z^{(r_*,w_{\theta})}(u,\bar{u})$ has $(\beta,\theta)$-type symmetric coefficients semi-bounded by $C(\theta,r_*)$, for any $\|(u,\bar{u})\|_p<R\ll 1$ and $N$ satisfying (\ref{N-c}), $| \{ Z^{(r_*,w_{\theta})}(u,\bar{u}), \|u\|_p^2\}|\prec R^{r_*+1}$.

\section{Estimate  $\{ f_r^{w_{\theta}}(u,\bar{u}), \|u\|_p^2\}_{w_{\theta}}\ $ and Hamiltonian vector field $X_{f_r^{w_{\theta}}}$  ($f_r^{w_{\theta}}(u,\bar{u})$ has $(\beta,\theta)$-type symmetric coefficients $\theta\in\{0,1\}$)}

Suppose that an $r$-degree homogeneous power series $f_r^{w_{\theta}}(u,\bar{u})$ defined on ${\cal H}^p(\mathbb{Z}^*,\mathbb{C})$ is of $(\beta,\theta)$-type symmetric coefficients semi-bounded by $C_{\theta}>0$. 

First of all I present that the Hamiltonian vector field of   $f_r^{w_{\theta}}(u,\bar{u})$   under  symplectic form $w_{\theta}$ is unbounded with order 1.  See  Proposition \ref{cor2}. 
 
Next, the estimate of the possion bracket of the power series $f_r^{w_{\theta}}(u,\bar{u})$  and $\|u\|_p^2$  is
 given  in Proposition \ref{2.1}.
 
 Last but not least, I introduce truncated operators $ \Gamma^N_{\leq 2}$ and $\Gamma^N_{>2}$, and   estimate the Hamiltonian vector fields of the functions  $  \Gamma^N_{\leq 2}f_r^{w_{\theta}}(u,\bar{u})$, 
  $  \Gamma^N_{>2}f_r^{w_{\theta}}(u,\bar{u})$ in  ${{\mathcal H}^{p-1}}(\mathbb Z^*,\mathbb C)$-norm, 
  $ |\{ \Gamma^N_{\leq 2}f_r^{w_{\theta}}(u,\bar{u}),\|u\|_p^2\}|$ and $ |\{ \Gamma^N_{> 2}f_r^{w_{\theta}}(u,\bar{u}),\|u\|_p^2\}|$  for any $(u,\bar{u})\in {{\mathcal H}^p}(\mathbb Z^*,\mathbb C) $ 
  in Corollary \ref{4.3}.

These results will be used in proving Theorem \ref{T22}.

\begin{Proposition}\label{cor2}
   Suppose that an  $r$-degree  ($r\geq3$) homogeneous polynomial 
  $f^{w_{\theta}}_r(u,\bar{u})$ ($\theta\in\{0,1\}$) defined on $ {{\mathcal H}^p}(\mathbb Z^*,\mathbb C)$   has $( \beta,\theta)$-type symmetric coefficients semi-bounded by $C_{f^{w_{\theta}}}>0$, 
  and $\beta-p\geq2$. Then for any $(u,\bar{u})\in  {{\mathcal H}^p}(\mathbb Z^*,\mathbb C)$ and any $\theta\in\{0,1\}$,
  \begin{equation}\label{1-1}
   \| X^{w_{\theta}}_{f^{w_{\theta}}_r}(u,\bar{u})\|_{p-1}\leq 16C_{f^{w_{\theta}}}^{r-2}r^{p+1}  c^{r-1}  
  \|u\|_{2}^{r -2}\|u\|_{p} .
\end{equation}
 If  an $r$-degree homogeneous polynomial $f_r(u,\bar{u})$ has $\beta$-bounded symmetric coefficients bounded by $C_f>0$, then the following inequality  holds true for any $(u,\bar{u})\in  {{\mathcal H}^p}(\mathbb Z^*,\mathbb C)$ and any $\theta\in\{0,1\}$
\begin{equation}
   \| X^{w_{\theta}}_{f_r}(u,\bar{u})\|_{p}\leq  16C_{f_r}^{r-2}r^{p+1}  c^{r-1}    \|u\|_{2}^{r-2}\|u\|_{p} .
\end{equation}
   \end{Proposition}

{\begin{Remark}\label{dj}
If  an $r$-degree homogeneous polynomial  $f^{w_{\theta}}_r(u,\bar{u}):B_p(R_*)\rightarrow \mathbb{C}$ ($\theta\in\{0,1\}, \ R_*>0$) has $( \beta,\theta)$-type symmetric coefficients semi-bounded by $C_{f^{w_{\theta}}}>0$ ($\theta\in\{0,1\}$), from Proposition \ref{cor2}  it holds  
that
\begin{itemize}
  \item  The Hamiltonian vector field $X^{w_{\theta}}_{f^{w_{\theta}}_r}$  is from $B_p(R_*)$ to ${{\cal H}}^{p-1}(\mathbb{Z}^*,\mathbb{C})$, but not to ${{\cal H}^p}(\mathbb{Z}^*,\mathbb{C}) $. It  means that  $X_{f_r^{w_{\theta}}}^{w_{\theta}}$ is unbounded with order 1.
   \item
There exists a constant $\delta\in(0,1)$ such that   
\begin{eqnarray}
   |f_r^{w_{\theta}}(u,\bar{u})|\!=\! |\langle \nabla_{(u,\bar{u})} f_r^{w_{\theta}}(\delta u,\delta\bar{u}), (u,\bar{u})\rangle|. 
   \end{eqnarray} 
  Together with Cauchy estimate and (\ref{1-1}), one has that 
    the function $f_r^{w_{\theta}}(u,\bar{u})$ $(\theta\in\{0,1\})$ is analytic about $(u,\bar{u})$ on some $B_p(R)\subset {{\cal H}^p}(\mathbb{Z}^{*}, \mathbb{C})$ ($p>1$). 
  {(In this paper, when I mention the ``analyticity" of functions or vector fields,
I take  $u$ and $\bar{u}$ as independent variables).}
\end{itemize}
 \end{Remark}}  
   
\begin{Proposition}\label{2.1}
 Suppose that an $r$-degree ($r\geq3$) homogeneous polynomial  $f^{w_{\theta}}_r(u,\bar{u})$ $(\theta \in \{0,1\})$ has $(\beta,\theta)$-type symmetric coefficients semi-bounded by $C_{f_r^{w_{\theta}}}>0$.
   Then the following inequality holds true for any $(u,\bar{u})\in  {\mathcal H}^p(\mathbb{Z}^*,\mathbb{C})$ ($p\geq2, \ \beta-p\geq2$) 
\begin{equation*}
\big|\{f^{w_{\theta}}_r(u,\bar{u}),\|u\|_p^2\}_{w_{\theta}}  \big|\leq
 {C^{r-2}_{f^{w_{\theta}}_r}  2^{p+1}p r^{p-1}c^{r-1} }\|u\|_{p }^2\|u\|_{2}^{r-2} .
\end{equation*}
\end{Proposition}

Given an integer $N>0$,
  two projection operators  $ \Gamma_{>N}$ and $\Gamma_{\leq N}$ on $\mathbb{N}^{\mathbb Z^*}$
and $ \ell^2_p(\mathbb{Z}^*
,\mathbb{C})$ are defined as follows. 
For any $l=(l_j)_{j\in\mathbb{Z}^*}\in \mathbb{N}^{\mathbb{Z}^*} $, $\Gamma_{>N}l$ and $\Gamma_{\leq N}l\in\mathbb{N}^{\mathbb{Z}^*} $ with
   $$
(\Gamma_{> N}l)_j:=\left\{\begin{array}{ll}
l_j,& |j|>N,\\
0,& |j|\leq N,
  \end{array}\right. \  (\Gamma_{\leq N}l)_j:=\left\{
\begin{array}{ll}
0,& |j|>N,
\\ l_j, & |j|\leq N.
\end{array}
\right. $$
For any  $u=(u_j)_{j\in\mathbb{Z}^*}\in \ell^2_p(\mathbb{Z}^*
,\mathbb{C})$, $\Gamma_{>N}u$, $\Gamma_{\leq N}u\in \ell^2_p(\mathbb{Z}^*
,\mathbb{C})$ with
$$(\Gamma_{>N}u)_j:=\left\{\begin{array}{ll}
u_j,& |j|>N,\\
0,& |j|\leq N,
  \end{array}\right. \ \      (\Gamma_{\leq N}u)_j:=\left\{
\begin{array}{ll}
0,& |j|>N,
\\ u_j, & |j|\leq N.
\end{array}
\right.$$ Now
I will introduce two truncated  operators $\Gamma^{N}_{\leq 2}$ and $\Gamma^{N}_{>2}$
defined as follows. For any power series
$$f(u,\bar{u}):=  \sum_{r\geq3}\sum_{|l+k|=r\atop{{\cal M}(l,k)=i}} f^{i}_{r,lk}u^l\bar{u}^k
 $$  denote
\begin{eqnarray}\label{trun1}
\Gamma^{N}_{\leq 2} f \! (u,\bar{u})\!:= 
\ \sum_{r\geq3}\sum_{|l+k|=r,{\cal M}(l,k)=i\atop{\!|\Gamma_{>N}(l+k)\!|\leq 2, |i|\leq N }}  f^{i}_{r,lk}u^l\bar{u}^k, 
\end{eqnarray}
\begin{eqnarray}\label{trun2}
\Gamma^{N}_{> 2} f(u,\bar{u}):= f(u,\bar{u})- \Gamma^{N}_{\leq 2} f(u,\bar{u}).
\end{eqnarray}

   \begin{Remark}\label{trunction}
Fix a positive integer $N$. Suppose that a power series $f^{w_{\theta}}(u,\bar{u})$
 has  $(\beta,\theta)$-type symmetric coefficients semi-bounded by $C_{f^{w_{\theta}}}>0$ $(\theta\in\{0,1\})$.
 Then $ \Gamma^{N}_{\leq 2}f^{w_{\theta}}(u,\bar{u})$, $\Gamma^{N}_{> 2}f^{w_{\theta}}(u,\bar{u}) $ also  have $(\beta,\theta)$-type symmetric coefficients semi-bounded by $C_{f^{w_{\theta}}}>0.$
\end{Remark}

\begin{Corollary}\label{4.3}Suppose that an  $r$-degree  ($r\geq3$) homogeneous polynomials 
  $f^{w_{\theta}}_r(u,\bar{u})$ ($\theta\in\{0,1\}$) defined on $ {{\mathcal H}^p}(\mathbb Z^*,\mathbb C)$   has $( \beta,\theta)$-type symmetric coefficients semi-bounded by $C_{f^{w_{\theta}}}>0$, 
  and $\beta-p\geq2$. Given an integer $N>0$,    then  
\begin{eqnarray} \label{fgamma}
   \| X^{w_{\theta}}_{\Gamma^{N}_{\leq2}f^{w_{\theta}}_{r}}(u,\bar{u})\|_{p-1} \!\!&\leq&\!\!  16 C_{f^{w_{\theta}}}^{r-2} r^{p+1}  c^{r-1}   
  \|u\|_{2}^{r-2}\|u\|_{p},\nonumber\\
  \| X^{w_{\theta}}_{\Gamma^{N}_{>2}f^{w_{\theta}}_{r}}(u,\bar{u})\|_{p-1}\!\!&\leq& \!\! 16C_{f^{w_{\theta}}}^{r-2} r^{p+1}  c^{r-1}  
  \|u\|_{2}^{r-3} \|\Gamma_{>N} u\|_2\|u\|_{p} \nonumber\\
  &&+16N^{-(\beta-p-\frac12)}C_{f^{w_{\theta}}}^{r-2} r^{p+1}  c^{r-1}
  \|u\|_{2}^{r-2}  \|u\|_{p},\nonumber\\
  \big|\{\Gamma_{\leq2}^Nf^{w_{\theta}}_r(u,\bar{u}),\|u\|_p^2\}_{w_{\theta}}  \big| \!\!&\leq&\!\!
 { C^{r-2}_{f^{w_{\theta}}_r}  2^{p+1}p r^{p-1}c^{r-1} }\|u\|_{p }^2\|u\|_{2}^{r-2},\nonumber\\ 
 \big|\{\Gamma_{>2}^Nf^{w_{\theta}}_r(u,\bar{u}),\|u\|_p^2\}_{w_{\theta}}  \big| \!\!&\leq&\!\!
 { C^{r-2}_{f^{w_{\theta}}_r}  2^{p+1}p r^{p-1}c^{r-1} }\|u\|_{p }^2\|u\|_{2}^{r-3}\|\Gamma_{>N}u\|_{2} \nonumber\\
 &&+ N^{-(\beta-p-\frac12)}{ C^{r-2}_{f^{w_{\theta}}_r}  2^{p+2}p r^{p-1}c^{r-1}   }\|u\|_{p }^2\|u\|_{2}^{r-2}  .
 \end{eqnarray}
 
\end{Corollary}
The proof of Corollary \ref{trun2} is similar with Proposition \ref{cor2}-\ref{2.1} and I omit it.  
 In order to give the proof of Proposition \ref{cor2}-\ref{2.1},  the following Lemmas are needed.   
 
\begin{Lemma}\label{lem4.2}
 If power series  $f^{w_{\theta}}(u,\bar{u})$ and $g^{w_{\theta}}(u,\bar{u})$ have $(\beta,\theta)$-type symmetric coefficients ($\theta\in\{0,1\}$) semi-bounded by   $C_{f^{w_{\theta}}}>0$ and $C_{g^{w_{\theta}}}>0$, respectively, 
 then for any $a,b\in \mathbb{R}$,  $(af^{w_{\theta}}+bg^{w_{\theta}})(u,\bar{u})$ also has $(\beta,\theta)$-type symmetric coefficients  semi-bounded by  
{\small\begin{equation}\label{c-c-c}
C_{af^{w_{\theta}}+bg^{w_{\theta}}}:=\max\{|a|C_{f^{w_{\theta}}}+|b|C_{g^{w_{\theta}}},\ |a|C_{f^{w_{\theta}}}+C_{g^{w_{\theta}}},\ C_{f^{w_{\theta}}}+|b|C_{g^{w_{\theta}}},\ C_{f^{w_{\theta}}}+C_{g^{w_{\theta}}}\}>0.
\end{equation} }
  \end{Lemma}
\begin{proof}
I only give the proof in the case $\theta=0$,  while in the case $\theta=1$ the proof is similar to the case $\theta=0$.
For any $a$, $b\in\mathbb{R}$,
\begin{equation}
(af^{w_{0}}+bg^{w_{0}})(u,\bar{u})=\sum_{t\geq3} \sum_{|l+k|=t\atop{{\cal M}(l,k)=i\in M_{(af^{w_{0}}+bg^{w_{0}})_t}}}(af^{w_{0}}+bg^{w_{0}})_{t,lk}^i u^l\bar{u}^k,
\end{equation}
where $ M_{(af^{w_{0}}+bg^{w_{0}})_t}:=M_{f^{w_{0}}_t}\cup M_{g^{w_{0}}_t}.$
For any $l,k\in \mathbb{N}^{\mathbb Z^*}$ with  ${\cal M}(l,k)=i\in M_{(af^{w_{0}}+bg^{w_{0}})_t}$, the corresponding coefficient of $af^{w_{0}}+bg^{w_{0}}$ has the following form
   $$(af^{w_{0}}+bg^{w_{0}})_{t,lk}^i =\sum_{(l^0,k^0,i^0)\in {\cal A}_{(af^{w_{0}}+bg^{w_{0}})^i_{t,lk}}}(af^{w_{0}}+bg^{w_{0}})_{t,lk}^{i(l^0,k^0,i^0)} ({\cal M}(l^0,k^0)-\frac{i^0}2),$$
   where $$ {\cal A}_{(af^{w_{0}}+bg^{w_{0}})^i_{t,lk}}:= {\cal A}_{(f^{w_{0}})^i_{t,lk}}\cup {\cal A}_{(g^{w_{0}})^i_{t,lk}}$$
   and
   $$ (af^{w_{0}}+bg^{w_{0}})_{t,lk}^{i(l^0,k^0,i^0)}\!:=\!
   \left\{\begin{array}{lll}
\!\!a{f^{w_{0}}}_{t,lk}^{i(l^0,k^0,i^0)}+b{g^{w_{0}}}_{t,lk}^{i(l^0,k^0,i^0)}, & (l^0,k^0,i^0)\in {\cal A}_{(f^{w_{0}})^i_{t,lk}}\cap {\cal A}_{(g^{w_{0}})^i_{t,lk}},\\
\!\! a{f^{w_{0}}}_{t,lk}^{i(l^0,k^0,i^0)}, & (l^0,k^0,i^0)\in {\cal A}_{(f^{w_{0}})^i_{t,lk}} \cap {\cal A}_{(g^{w_{0}})^i_{t,lk}}^C ,\\
 \!\!b{g^{w_{0}}}_{t,lk}^{i(l^0,k^0,i^0)}, & (l^0,k^0,i^0)\in {\cal A}_{(f^{w_{0}})^i_{t,lk}}^C\cap {\cal A}_{(g^{w_{0}})^i_{t,lk}} .\\
\end{array} \right.$$
  It is easy to check that
   $$ \overline{(af^{w_{0}}+bg^{w_{0}})_{t,lk}^{i(l^0,k^0,i^0)}}=(af^{w_{0}}+bg^{w_{0}})_{t,kl}^{-i(k-k^0,l-l^0,i^0-2i)}$$
   and
  \begin{eqnarray*}
 && \sum_{(l^0,k^0,i^0)\in {\cal A}_{(af^{w_{0}}+bg^{w_{0}})^i_{t,lk}}}|(af^{w_{0}}+bg^{w_{0}})_{t,lk}^{i(l^0,k^0,i^0)}|\cdot \max\{\langle i^0\rangle, \ \langle i^0-2i\rangle \}\nonumber\\
 &\leq& \sum_{(l^0,k^0,i^0)\in {\cal A}_{(f^{w_{0}})^i_{t,lk}}}|a(f^{w_{0}})_{t,lk}^{i(l^0,k^0,i^0)}|\cdot\max\{\langle i^0\rangle, \ \langle i^0-2i\rangle\}\nonumber\\
 &&+\sum_{(l^0,k^0,i^0)\in {\cal A}_{(g^{w_{0}})^i_{t,lk}}}|b(g^{w_{0}})_{t,lk}^{i(l^0,k^0,i^0)}|\cdot\max\{\langle i^0\rangle, \ \langle i^0-2i\rangle \}\nonumber\\
 &\leq&   \frac{(C_{af^{w_{0}}+bg^{w_{0}}})^{t-2}}{ \langle i\rangle^{\beta}  },
  \end{eqnarray*}
  where $C_{af^{w_{0}}+bg^{w_{0}}}$ is defined in (\ref{c-c-c}). 
 \end{proof}

\begin{Lemma}\label{f-v} Given  real numbers $  q_i\leq p $ ($1\leq i\leq r$), suppose that
 $ {F}=(F_j)_{j\in\mathbb{Z}^*}$  is an  $r$-multiple linear vector field defined as following
 \begin{eqnarray*}
F_j(u^{(1)},\cdots ,u^{(r)})
:=\sum_{j=j_{(1)}\pm j_{(2)}\cdots \pm j_{(r)}}F_{j, j_{(1)}\cdots j_{(r)}} u_{j_{(1)}}^{(1)}\cdots    u^{(r)}_{j_{(r)}},\ j\in \mathbb{Z}^*
 \end{eqnarray*}
where  $u^{(1)}:=\big(u^{(1)}_{j_{(1)}}\big)_{j_{(1)}\in \mathbb{Z}^*},$ $ \cdots$ $ u^{(r)}:=\big(u^{(r)}_{j_{(r)}}\big)_{j_{(r)}\in \mathbb{Z}^*}\in
\ell^2_p(\mathbb Z^*,\mathbb C)$.
 If there exist a positive constant $C$ and an integer $ n \in \{ 1,\cdots ,r  \}$  such that
\begin{equation}\label{c-1}
|F_{j, j_{(1)}\cdots j_{(r)}}| \leq C  \langle j_{(n)}\rangle^{q_n} \cdot\prod_{t=1,t\neq n}^{r}  \langle j_{(t)}\rangle^{q_t-1}   ,
\end{equation}
 then
\begin{equation*}
\|F(u^{(1)},\cdots ,u^{(r)})\|_{\ell^2(\mathbb{Z}^*,\mathbb{C})}\leq C c^{r-1}\big\|u^{(n)}\big\|_{q_n}\cdot  \prod_{i=1,i\neq n}^r \big\|u^{(i)}\big\|_{q_t},
\end{equation*}
where  $ c:=\sqrt{\sum_{j\in
 \mathbb{Z}} \langle j\rangle^{-2} }.$
\end{Lemma}
\begin{proof}

 Using Young's inequality\footnote{
 Suppose $a\in\ell^p(\mathbb{Z}^*,\mathbb{C})$,  $b\in\ell^q(\mathbb{Z}^*,\mathbb{C})$ and
 $\frac1p+\frac1q=\frac1r+1,$ with $1\leq p,q,r\leq \infty$. Then $\|f\ast g\|_{\ell^r(\mathbb{Z}^*,\mathbb{C})}\leq \|f\|_{\ell^p(\mathbb{Z}^*,\mathbb{C})}\cdot \|g\|_{\ell^q(\mathbb{Z}^*,\mathbb{C})}.$  }  the following inequality holds true  for any $a\in \ell^2(\mathbb{Z}^*,\mathbb{C})$ and $ b\in \ell^1(\mathbb{Z}^*,\mathbb{C})$ 
\begin{equation}\label{wy}
\big\| \big(\sum_{k\in  \mathbb{Z}^* }
a_{j\pm k}b_k\big)_{j\in  \mathbb{Z}^*}\big\|_{\ell^2(\mathbb{Z}^*,\mathbb{C})}\leq
\|a\|_{\ell^2(\mathbb{Z}^*,\mathbb{C})}\cdot \|b\|_{\ell^1(\mathbb{Z}^*,\mathbb{C})}.
\end{equation}
Together with (\ref{c-1}), using (\ref{wy}) repeatedly,
one has  
\begin{eqnarray}\label{jcc}
&&\|F(u^{(1)},\cdots ,u^{(r)})\|_{\ell^2(\mathbb{Z}^*,\mathbb{C})} \nonumber\\
  &\leq&C\big\|\big( \sum_{j=j_{(1)}\pm j_{(2)}\cdots \pm j_{(r)}}   \langle j_{(1)}\rangle^{q_1-1}  | u_{j_{(1)}}^{(1)}|\cdots   \langle j_{(n-1)} \rangle^{q_{n-1}-1}  | u^{(n-1)}_{j_{(n-1)}}|\cdot \langle j_{(n)}\rangle^{q_n}   | u^{(n)}_{j_{(n)}}|  \cdot \nonumber\\
  &&
   \langle j_{(n+1)} \rangle^{q_{n+1}-1}  |u^{(n+1)}_{j_{(n+1)}} |\cdots  \langle j_{(r)}\rangle^{q_{r}-1} |u^{(r)}_{j_{(r)}}| \big)_{j\in\mathbb{Z}^*} \big\|_{\ell^2(\mathbb{Z}^*,\mathbb{C})}\nonumber\\
 &\leq& C \|u^{(n)}\|_{q_n } \cdot  \prod_{1\leq i\leq r, i\neq n}  \big\|( \langle j_{(i)}\rangle^{q_{i}-1}\cdot|u^{(i)}_{j_{(i)}}|)_{j_{(i)}\in\mathbb{Z}^*}
 \big\|_{\ell^1(\mathbb{Z}^*,\mathbb{C})} .
 \end{eqnarray}
Since $q_i\leq p$,   the following inequality holds true for any $u=(u_j)_{j\in\mathbb{Z}^*}\in \ell^2_p(\mathbb Z^*,\mathbb C)$,
  \begin{eqnarray}\label{4.5}
  \|( \langle j\rangle^{q_i-1} |u_j|)_{j\in\mathbb{Z}^*}\|_{\ell^1 (\mathbb{Z}^* ,\mathbb{C})}= \sum_{j\in\mathbb{Z}^*}\big(\langle j\rangle^{q_i} |u_j|  \cdot  \frac{1}{ \langle j\rangle }\big)  
  \leq  \sqrt{ \sum_{j\in\mathbb{Z}}\langle j\rangle^{-2} }\cdot \|u\|_{q_i}=c \|u\|_{q_i}.
   \end{eqnarray}
 In view of   (\ref{4.5}) and (\ref{jcc}), one has 
$$\|F(u^{(1)},\cdots ,u^{(r)})\|_{\ell^2(\mathbb{Z}^*,\mathbb{C})}\leq C c^{r-1}\big\|u^{(n)}\big\|_{q_n}\cdot  \prod_{i=1,i\neq n}^r \big\|u^{(i)}\big\|_{q_i}. $$
\end{proof}

\begin{Corollary}\label{r1}Given integers $p>q \geq 1$ and real number $\rho\geq 2$,
suppose  that there exists a positive number $C_f>0$  such that the coefficients of  an $r$-degree homogeneous polynomial
$$ f(u,\bar{u})=    \sum_{l,k\in \mathbb{N}^{\mathbb{Z}^*},\ i\in M_{f_r}\subseteq \mathbb{Z}  \atop{ |l+k|=r,\  {\cal M}(l,k)=i  } }f^i_{r,l k} u^{l}\bar{u}^{k}  $$ satisfy that for
any $l,k\in \mathbb{N}^{\mathbb{Z}^*}$ with ${\cal M}(l,k)=i \in M_{f_r}$
  \begin{eqnarray}\label{c-2}
  &&|f^i_{r,lk}|
  \leq  \frac{C_f}{ \langle i\rangle^{\rho}} l_j \langle j\rangle^{p-q+1} \big( \sum_{t\in\mathbb{Z}^*   } (l+k-e_j)_t  \langle t\rangle^{(p-q+1)}   \big) \prod_{m\in\mathbb{Z}^*} \langle m\rangle^{(q-1)(l_m+k_m)} .
  \end{eqnarray}
     Then  for any $(u,\bar{u})\in  {{\mathcal H}^p}(\mathbb Z^*,\mathbb C)$,  it satisfies that  
 $$ |f(u,\bar{u})|  \leq  C_f    c^{r-1}    r \|u\|^2_{p}\cdot \|u\|_{q}^{r-2} .$$ 
\end{Corollary}
\begin{Remark}The 
result of Corollary \ref{r1}  still holds true for $\rho>1$.  To simplify the process of proof,  assume $\rho\geq2$. 
\end{Remark}
 \begin{proof}
 By Cauchy estimate,  
 \begin{equation}\label{inner}
 |f(u,\bar{u})|\leq |\langle F,G\rangle|\leq \|F\|_{\ell^2}\cdot \|G\|_{\ell^2},
 \end{equation}
where $G:=(\langle j\rangle^p|u_j|)_{j\in\mathbb{Z}^*}$ and $F(u,\bar{u})=(F_j(u,\bar{u}))_{j\in\mathbb{Z}^*}$
$$ F_j(u,\bar{u}):= \sum_{|l-e_j+k|=r-1,
 \atop{ {\cal M}(l,k)=i\in M_{f_r}\subseteq\mathbb{Z} }} \frac{f^i_{r,l k}}{\langle j\rangle^p }u^{l-e_j}\bar{u}^k . $$
For any $ n\in \{0,\ 1,\cdots,r-1\}$,  there exist an $(r-1)$-multiple linear vector field
$$   \widetilde{F}^n(u^{(1)} ,\cdots u^{(r-1
)} ):=\big( \sum_{|l-e_j|=n,
|k|=r-1-n,\atop{ i\in M_{f_r}\subseteq\mathbb{Z},\ j={\cal M}(l-e_j,k)-i }} \frac{f^i_{r,l k}}{\langle j\rangle^p} \underbrace{u^{(1)}\cdots u^{(n)}}_{n} \underbrace{{u}^{(n+1)}\cdots u^{(r-1)} }_{r-1-n} \big)_{j\in\mathbb{Z}^*}
 $$
such that 
$$F_j(u,\bar{u})= \sum_{n=0}^{r-1}    \widetilde{F}_j^{n}( \underbrace{u,\cdots ,u}_{n}, \underbrace{\bar{u},\cdots ,\bar{u}}_{r-1-n}),\quad \mbox{for any }{j\in \mathbb{Z}^*}, .$$
 By condition (\ref{c-2}),  the coefficients of each $(r-1)$-multiple linear vector fields $\widetilde{F}^{n}$ satisfy the condition (\ref{c-1}) of Lemma \ref{f-v} with $ q_n=p$ and $q_i=q\leq p\ (i\neq n)$.    Hence, by Lemma \ref{f-v}, for any $(u,\bar{u})\in  {\mathcal H}^p(\mathbb Z^*,\mathbb C)$, one has
 \begin{eqnarray}\label{e-v}
  &&\|F(u,\bar{u})\|_{\ell^2} \leq \sum_{n=0}^{r-1}  \| \widetilde{F}^{n}( \underbrace{u,\cdots ,u}_{n}, \underbrace{\bar{u},\cdots ,\bar{u}}_{r-1-n})\|_{\ell^2(\mathbb{Z}^*,\mathbb C)}\nonumber\\
  &\leq & C_f  c^{r-2} r  \|u\|_{p}\cdot \|u\|_{q}^{r-2} \sum_{i\in M_{f_r}\subset \mathbb{Z}} \langle i\rangle^{-\rho} \nonumber\\
  &\leq &C_f  c^{r-1} r\|u\|_{p}\cdot \|u\|_{q}^{r-2}   .
  \end{eqnarray}
   In view of (\ref{inner}) and (\ref{e-v}), the following inequality holds true
  \begin{eqnarray*}
  |f(u,\bar{u})|&\leq& \|F(u,\bar{u})\|_{\ell^2}\cdot \|G(u,\bar{u})\|_{\ell^2}\nonumber\\
   &\leq& \|F(u,\bar{u})\|_{\ell^2}\cdot\|u\|_p\leq C_f c^{r-1} r  \|u\|^2_{p}\cdot \|u\|_{q}^{r-2}.
  \end{eqnarray*}
 \end{proof}

Now give the proof of Proposition \ref{cor2}.
 \begin{proof}

\noindent \noindent In the case $\theta=0$,
\begin{eqnarray}\label{yy-1}
 &&\big\| {X}^{w_0}_{f^{w_0}_r}(u,\bar{u}) \big\|_{p-1}
 = \sqrt{    \|  \nabla_{\bar{u}} f^{w_0}_r(u,\bar{u})\|_{p-1}^2+\| \nabla_{u} f^{w_0}_r(u,\bar{u}) \|_{p-1}^2     }.
   \end{eqnarray}
Note that $ \|  \nabla_{\bar{u}} f^{w_0}_r(u,\bar{u})\|_{p-1}$ equals to  
 equals to $\ell^2(\mathbb Z^*,\mathbb C)$ norm of the following vector field
\begin{small}\begin{equation}\label{lso}
 \bigg(  \sum_{j={\cal M}(l,k-e_j)-i\atop{l,k\in \mathbb{N}^{\mathbb{Z}^*},\ i\in M_{f^{w_0}_r}\subseteq\mathbb{Z}}}   k_j \langle j\rangle^{p-1} \cdot\sum_{( l^0,k^0,i^0)\in{\cal A}_{(f^{w_0})^i_{r,lk}}} (f^{w_0})^{i(l^0,k^0,i^0)}_{r,{l k}}\big( {\cal M}( l^0,k^0)-\frac{i^0}2\big)u^l\bar{u}^{k-e_j} \bigg)_{j\in\mathbb{Z}^*} .
  \end{equation}
  \end{small}
  Accordingly,  $ \|  \nabla_{{u}} f^{w_0}_r(u,\bar{u})\|_{p-1}$ equals to  
 equals to $\ell^2(\mathbb Z^*,\mathbb C)$ norm of the following vector field
\begin{small}\begin{equation}\label{lso-1}
 \bigg(  \sum_{j=i-{\cal M}(l-e_j,k)\atop{l,k\in \mathbb{N}^{\mathbb{Z}^*},\ i\in M_{f^{w_0}_r}\subseteq\mathbb{Z}}}   l_j \langle j\rangle^{p-1} \cdot\sum_{( l^0,k^0,i^0)\in{\cal A}_{(f^{w_0})^i_{r,lk}}} (f^{w_0})^{i(l^0,k^0,i^0)}_{r,{l k}}\big( {\cal M}( l^0,k^0)-\frac{i^0}2\big)u^{l-e_j}\bar{u}^{k} \bigg)_{j\in\mathbb{Z}^*} .
  \end{equation}
  \end{small}
In order to  use  Lemma  \ref{f-v}  to give the $\ell^2$-norm of (\ref{lso}) and (\ref{lso-1}), it is required to  estimate the coefficients of vector fields (\ref{lso}) and (\ref{lso-1}). Note that for any fixed $l,k\in\mathbb{N}^{\mathbb{Z^*}}$ and $j\in\mathbb{Z}^*$ with $(f^{w_0})^{i}_{r,lk}\neq 0$ and  $k_j\neq0$,  the indices satisfy
 \begin{equation}\label{ab}
 \mbox{sgn}(k_j)\cdot j={\cal M}(l,k-e_j)-i.
  \end{equation}
  Since $p>2$, it follows  
  \begin{eqnarray}\label{ab1}
  |j|^{p-1}&\leq& r^{p-1}(\sum_{t\in\mathbb{Z}^*} l_t|t|^{p-1}+ \sum_{t\in\mathbb{Z}^*,t\neq j} k_t|t|^{p-1}+(k_j-1)\cdot |j|^{p-1} +|i|^{p-1} )\nonumber\\
  &\leq& 2r^{p-1}(\sum_{t\in\mathbb{Z}^*} l_t|t|^{p-1}+ \sum_{t\in\mathbb{Z}^*,t\neq j} k_t|t|^{p-1}+(k_j-1)\cdot |j|^{p-1} )\cdot\langle i\rangle^{p-1}  ,
   \end{eqnarray}
the last inequality holds true by the fact that for any integer $a,\ b \geq 1$, 
\begin{equation}\label{a---b}
 a+b\leq ab+1\leq 2ab.
\end{equation}
Furthermore, for any $(l^0,k^0,i^0)\in {\cal A}_{(f^{w_0})_{r,lk}^i}$, it holds that
\begin{equation}\label{zzz1}
0\leq l^0_j\leq l_j,\quad 0\leq k^0_j\leq k_j,\quad \mbox{for any }j\in\mathbb{Z}^*
\end{equation}
and
the momentum of $(l^0,k^0)$ equals to
\begin{equation}\label{4.14-a}
 {\cal M}(l^0,k^0)=\left\{
 \begin{array}{ll}
{\cal M}\big(l^0,k^0-\mbox{sgn}(k_j)e_j\big)-\mbox{sgn}(k_j)\cdot  j ,& k_j^0\neq 0,\\
\\
{\cal M}(l^0,k^0),&k_j^0=0.
 \end{array}
 \right.
 \end{equation}
 Together with (\ref{ab}), (\ref{a---b}), (\ref{zzz1}) and (\ref{4.14-a}),  
 one has 
\begin{eqnarray}\label{M-l^0}
|{\cal M}(l^0,k^0)-\frac{i^0}2|\leq
 \!4\big(\sum_{n\in\mathbb{Z}^*}l_n |n|\!+\!\sum_{n\in\mathbb{Z}^*\atop{n\neq j}}k_n|n|+(k_j-1)|j| \big) \!\cdot\! \max\{ \langle i^0\rangle, \langle i-i^0\rangle \}.
\end{eqnarray}
In view of   (\ref{ab}), (\ref{ab1}),  (\ref{M-l^0})  and  $ f^{w_0}(u,\bar{u})$ having the ($\beta,0$)-type symmetric coefficients semi-bounded by $C_{f^{w_0}}>0$,  one has 
  \begin{eqnarray}\label{co-2}
 &&   \langle j\rangle^{p-1}\cdot\big|\sum_{(l^0,k^0,i^0)\in{\cal A}_{(f^{w_0})^i_{r,lk}}} k_j(f^{w_0})^{i(l^0,k^0,i^0)}_{r,{lk}} \big({\cal M}(l^0,k^0)-\frac{i^0}2\big)\big|\nonumber\\
&\leq& \frac{8C_{f^{w_0}}^{r-2}  r^{p}}{\langle i\rangle^{\beta-p+1}} \big(\sum_{m\in\mathbb{Z}^*,\ l_m+k_m>0} \langle m\rangle^{\mbox{sgn} (l_m+(k-e_j)_m )}\prod_{t\in\mathbb{Z}^* } \langle t\rangle^{(p-1)\mbox{sgn} (l_t+(k-e_j)_t )} \big). 
 \end{eqnarray}
Therefor,  using  (\ref{co-2}) and  Lemma  \ref{f-v} by taking $ q_n=p$ and $q_i=2\ (i\neq n)$, the following inequality holds true 
{\begin{eqnarray*} 
&& \|  \nabla_{\bar{u}} f^{w_0}_r(u,\bar{u})\|_{p-1}\nonumber\\
&\leq&  \bigg\|\big( \sum_{j={\cal M}(l,k-e_j)-i\atop{l,k\in \mathbb{N}^{\mathbb{Z}^*},\ i\in M_{f^{w_0}_r}\subseteq\mathbb{Z},}}  \sum_{( l^0,k^0,i^0)\in{\cal A}_{(f^{w_0})^i_{r,lk}}} \big| k_j j^{p-1} (f^{w_0})^{i(l^0,k^0,i^0)}_{r,{ l k}} \big({\cal M}( l^0,k^0)-\frac{i^0}2\big)u^l\bar{u}^{k-e_j}\big| \big)_{j\in\mathbb{Z}^*}\bigg\|_{\ell^2} \nonumber\\
   &\leq & 8r^{p+1}  c^{r-1} C_{f^{w_0}}^{r-2} 
  \|u\|_{2}^{r-2}\|u\|_{p}   .
   \end{eqnarray*}
   Similarly, one has $$ \|  \nabla_{{u}} f^{w_0}_r(u,\bar{u})\|_{p-1} \leq8r^{p+1}  c^{r-1} C_{f^{w_0}}^{r-2} 
  \|u\|_{2}^{r-2}\|u\|_{p}. $$
 Thus, 
}  \begin{eqnarray*}
  \big\|X^{w_0}_{f^{w_0}}(u,\bar{u})\big\|_{p-1}
\leq    16r^{p+1}  c^{r-1} C_{f^{w_0}}^{r-1}  
   \|u\|_{2}^{r-2}\|u\|_{p} .
\end{eqnarray*}

\noindent In the case $\theta=1$,
\begin{equation}\label{yy-2}
 \big\| {X}^{w_1}_{f^{w_1}_r}(u,\bar{u}) \big\|_{p-1}
 = \sqrt{    \|  \nabla_{\bar{u}} f^{w_1}_r(u,\bar{u})\|_{p-1}^2+\| \nabla_{u} f^{w_1}_r(u,\bar{u}) \|_{p-1}^2     } .
   \end{equation}
Similarly,  $ \|  \nabla_{\bar{u}} f^{w_1}_r(u,\bar{u})\|_{p-1}$ and $ \| \nabla_{u} f^{w_1}_r(u,\bar{u}) \|_{p-1}$ equal to   $\ell^2(\mathbb Z^*,\mathbb C)$-norm of the following vector fields respectively
 \begin{equation}\label{jjg}
 \big(k_j\cdot \langle j\rangle^{p-1} \langle j\rangle^{\frac12} \sum_{j={\cal M}(l,k-e_j)-i\atop{l,k\in \mathbb{N}^{\mathbb{Z}^*},\ i\in M_{f^{w_1}_r}\subseteq\mathbb{Z}}}    (\tilde{f}^{w_1})^{i}_{r,{l k}} \prod_{n\in\mathbb{Z}^*} \langle n\rangle^{\frac12 l_n}\prod_{n\in\mathbb{Z}^*\atop{n\neq j}} \langle n\rangle^{\frac12 k_n} u^l\bar{u}^{k-e_j} \big)_{j\in\mathbb{Z}^*}
 \end{equation}
 and
  \begin{equation} 
 \big(l_j\cdot \langle j\rangle^{p-1} \langle j\rangle^{\frac12} \sum_{j=i-{\cal M}(l-e_j,k)\atop{l,k\in \mathbb{N}^{\mathbb{Z}^*},\ i\in M_{f^{w_1}_r}\subseteq\mathbb{Z}}}    (\tilde{f}^{w_1})^{i}_{r,{l k}} \prod_{n\in\mathbb{Z}^*} \langle n\rangle^{\frac12 l_n}\prod_{n\in\mathbb{Z}^*\atop{n\neq j}} \langle n\rangle^{\frac12 k_n} u^{l-e_j}\bar{u}^{k} \big)_{j\in\mathbb{Z}^*}.
 \end{equation}
  For any  fixed $l,k\in\mathbb{N}^{\mathbb{Z}^*} $ and $j\in\mathbb{Z}^*$ with  $(f^{w_1})^i_{r,lk}\neq0$. ${\cal M}(l,k)=i$ and $k_j\neq0$,  (\ref{ab}) still holds true. By (\ref{ab}) and (\ref{ab1}),  
 the coefficients of the vector field in (\ref{jjg}) are bounded by 
 \begin{eqnarray}\label{jjj}
 &&  \langle j\rangle^{p-\frac12}\cdot | k_j  (\tilde{f}^{w_1})^{i}_{r,{ l k}} |\prod_{n\in\mathbb{Z}^*} \langle n\rangle^{\frac12 l_n}\prod_{n\in\mathbb{Z}^*\atop{n\neq j}} \langle n\rangle^{\frac12 k_n} \nonumber\\
 &\leq&2\frac{C_{f^{w_1}}^{r-2}  r^{p-1}}{   \langle i\rangle^{\beta-p+\frac12} } \big(\sum_{m\in\mathbb{Z}^*} {l_m} \langle m\rangle^{p-\frac12}  +\sum_{m\in\mathbb{Z}^*\atop{m\neq j}} k_m\langle m\rangle^{p-\frac12}+(k_j-1)\langle j\rangle^{p-\frac12}\big) \prod_{t\in\mathbb{Z}^* }  \langle t\rangle^{\frac12(l_t+(k-e_j)_t)} .\nonumber\\
 \end{eqnarray}
Using Lemma  \ref{f-v} and (\ref{jjj}), one has 
 \begin{eqnarray*} 
\|  \nabla_{\bar{u}} f^{w_1}_r(u,\bar{u})\|_{p-1}\leq 8C_{f^{w_1}}^{r-2}r^{p+1}  c^{r-1} 
  \|u\|_{2}^{r-2}\|u\|_{p} . 
  \end{eqnarray*}
Similarly,  $$ \|  \nabla_{{u}} f^{w_1}_r(u,\bar{u})\|_{p-1}\leq 8C_{f^{w_1}}^{r-2}r^{p+1}  c^{r-1} 
  \|u\|_{2}^{r-2}\|u\|_{p}. $$
Thus,    
  \begin{eqnarray*}
  \big\|X^{w_1}_{f^{w_1}}(u,\bar{u})\big\|_{p-1} 
\leq 16r^{p+1}  c^{r-1} C_{f^{w_1}}^{r-2}  
  \|u\|_{2}^{r-2}\|u\|_{p} .
\end{eqnarray*}

\noindent By the same approach, when   $f$ has $\beta$-bounded symmetric coefficients bounded by $C_f>0$, one has  
 \begin{eqnarray*}
 \|X^{w_{\theta}}_f\|_p\leq  16 C_{f }^{r-2} r^{p+1}  c^{r-1}   
  \|u\|_{2}^{r-2}\|u\|_{p} .
  \end{eqnarray*}
 \end{proof}

Next  the proof of Proposition \ref{2.1} is given.
\begin{proof} %{\bf Proof of Lemma \ref{2.1}:}
 Step 1: (delete unbounded part)

 \noindent
 In the case $\theta =0$, since $f^{w_0}_r(u,\bar{u})$ has $(\beta,0)$-type symmetric coefficients, assume that $f^{w_0}_r(u,\bar{u})$ has  the following form
 $$ f^{w_0}_r(u,\bar{u})=  \sum_{|l+k|=r\atop{{\cal M}(l,k)=i\in M_{f^{w_0}_r}\subseteq\mathbb{Z}}}  \sum_{( {l}^0,{k}^0,i^0)\in {\cal A}_{(f^{w_0})^i_{r,lk} }}\  {(f^{w_0})}^{i( {l}^0,{k}^0,i^0)}_{r,lk}\big({\cal M}( {l}^0, {k}^0)-\frac{i^0}2\big)u^l\bar{u}^k. $$
 Under the definition of Possion bracket, it holds that
\begin{eqnarray}
&&\{f^{w_0}_r(u,\bar{u}),\|u\|_p^2\}_{w_0}=\sum_{j\in \mathbb Z^*}   \bigg(\frac{\partial
f^{w_0}_{r}}{\partial u_j}\ {\bf i} \  \frac{\partial \|u\|_p^2}{\partial
\bar{u}_j}\bigg)-\sum_{j\in \mathbb Z^*}  \bigg(
\frac{\partial f^{w_0}_{r}}{\partial \bar{u}_j}\ {\bf i}  \ \frac{\partial
\|u\|_p^2}{\partial u_j}\bigg)\nonumber\\
&=& {\bf i} \sum_{|l+k|=r\atop{{\cal M}(l,k)=i\in M_{f^{w_0}_r}}}(l_j-k_j)\langle j\rangle^{2p} \sum_{( {l}^0,{k}^0,i^0)\in {\cal A}_{(f^{w_0})^i_{r,lk} }}\  (f^{w_0})^{i( {l}^0,{k}^0,i^0)}_{r,lk}\big({\cal M}( {l}^0, {k}^0)-\frac{i^0}2\big) u^l\bar{u}^k. \nonumber
\end{eqnarray}
For the sake of convenience, rewrite  $\{f^{w_0}_r(u,\bar{u}),\|u\|_p^2\}_{w_0}$ as the sum of the following two parts
{\footnotesize
$$O^+(u,\bar{u}):={\bf i} \sum_{|l+k|=r\atop{{\cal M}(l,k)=i\atop{i\in M_{f^{w_0}_r} \subseteq\mathbb{Z}}}}(l_j-l_j^0-k_j+k_j^0) \langle j\rangle^{2p}  \sum_{( {l}^0,{k}^0,i^0)\in {\cal A}_{(f^{w_0})^i_{r,lk} }}\  (f^{w_0})^{i( {l}^0,{k}^0,i^0)}_{r,lk}\big({\cal M}( {l}^0, {k}^0)-\frac{i^0}2\big) u^l\bar{u}^k,$$
}
\begin{small}
$$O^-(u,\bar{u}):={\bf i} \sum_{|l+k|=r\atop{{\cal M}(l,k)=i\atop{i\in M_{f^{w_0}_r} \subseteq\mathbb{Z}}}}(l^0_j-k^0_j )\langle j\rangle^{2p} \sum_{( {l}^0,{k}^0,i^0)\in {\cal A}_{(f^{w_0})^i_{r,lk} }}  (f^{w_0})^{i( {l}^0,{k}^0,i^0)}_{r,lk}\big({\cal M}( {l}^0, {k}^0)-\frac{i^0}2\big) u^l\bar{u}^k.$$
\end{small}
Since the coefficients of $f^{w_0}_r(u,\bar{u})$ are $(\beta,0)$-type symmetric,  
take complex conjugation to $O^-(u,\bar{u})$ and obtain
\begin{eqnarray}\label{f-f}
&&\overline{O^-(u,\bar{u})}\nonumber\\
 &=&\overline{{\bf i}\sum_{i\in M_{f^{w_0}_r}} \sum_{|l+k|=r\atop{{\cal M}(l,k)=i}}( l_j^0- k_j^0)\langle j\rangle^{2p}  \sum_{( {l}^0,{k}^0,i^0)\in {\cal A}_{(f^{w_0})^i_{t,lk} }}\  (f^{w_0})^{i( {l}^0,{k}^0,i^0)}_{t,lk}\big({\cal M}( {l}^0, {k}^0)-\frac{i^0}2\big) u^l\bar{u}^k }\nonumber\\
 &=& {\bf i}\sum_{-i\in M_{f^{w_0}_r}} \sum_{|l+k|=r\atop{{\cal M}(k,l)=-i}}( k^0_j - l_j^0)\langle j\rangle^{2p} \sum_{( k-k^0, l-{l}^0,i^0-2i)\in {\cal A}_{(f^{w_0})^{-i}_{t,kl} }}\  (f^{w_0})^{-i( k-{k}^0,l-{l}^0,i^0-2i)}_{t,kl}\nonumber\\
 &&\cdot\big({\cal M}( k-k^0,l-{l}^0)-(\frac{i^0}2-i)\big) u^k\bar{u}^l  \nonumber\\
 &=&O^+(u,\bar{u}).
 \end{eqnarray} Together with (\ref{f-f}), rewrite   $\{f^{w_0}_r(u,\bar{u}),\|u\|_{p }^2\}_{w_0}$ as the following 
\begin{eqnarray*}
 \{f^{w_0}_r(u,\bar{u}),\|u\|_{p }^2\}_{w_0}&=&
2{\bf Re} O^+(u,\bar{u})\nonumber\\
&=&A^0(u,\bar{u})+ A^1(u,\bar{u})+A^2(u,\bar{u}),
\end{eqnarray*} 
where
 \begin{eqnarray}\label{a0}
A^0(u,\bar{u})\!:=-\! {\bf Re}\! \sum_{|l+k|=r\atop{{\cal M}(l,k)=i\atop{i\in M_{f^{w_0}_r}\subseteq\mathbb{Z}}}}{\bf i}(l_j\!-\!l^0_j\!-\!k_j\!+\!k^0_j)\! \langle j\rangle^{2p} \sum_{\! ( {l}^0,{k}^0,i^0)\!\in \!{\cal A}_{(f^{w_0})^i_{r,lk} }}\! \! i^0\! (f^{w_0})^{i( {l}^0,{k}^0,i^0)}_{t,lk} u^l\bar{u}^k;  
\end{eqnarray}
\begin{eqnarray}\label{a1}
A^1(u,\bar{u}):&=&2 {\bf Re} \sum_{|l+k|=r\atop{{\cal M}(l,k)=i\in M_{f^{w_0}_r}\subseteq\mathbb{Z}}}{\bf i}(l_j-l^0_j-k_j+k^0_j)\langle j\rangle^{p} \nonumber\\
&&\cdot\sum_{( {l}^0,{k}^0,i^0)\in {\cal A}_{(f^{w_0})^i_{r,lk} }}\  (f^{w_0})^{i( {l}^0,{k}^0,i^0)}_{r,lk}\sum_{t\in\mathbb{Z}^*} (l^0_t-k^0_t)t(\langle j\rangle^{p}-\langle t\rangle^{p} )  u^l\bar{u}^k
\end{eqnarray}
and
\begin{eqnarray}\label{1-8}
A^2(u,\bar{u})&:=&2  {\bf Re} \sum_{|l+k|=r\atop{{\cal M}(l,k)=i\in M_{f^{w_0}_r}\subseteq\mathbb{Z}}}{\bf i}(l_j-l^0_j-k_j+k^0_j)\langle j\rangle^{p}\nonumber\\
 &&\cdot \sum_{( {l}^0,{k}^0,i^0)\in {\cal A}_{(f^{w_0})^i_{r,lk} }}\  (f^{w_0})^{i( {l}^0,{k}^0,i^0)}_{r,lk}\sum_{t\in\mathbb{Z}^*} (l^0_t-k^0_t)\cdot t \cdot \langle t\rangle^{p} u^l\bar{u}^k.
 \end{eqnarray}
 The estimate of $\{f_r^{w_0}(u,\bar{u}),\|u\|_p^2\}_{w_0}$ follows the estimates of $A^0(u,\bar{u})$, $A^1(u,\bar{u})$ and $A^2(u,\bar{u})$.
In fact, I can not estimate $A^2(u,\bar{u})$ by Corollary \ref{r1}  directly, because the coefficients of $A^2(u,\bar{u})$ are not satisfy condition (\ref{c-2}). Fortunately  the bad  unbounded part (not satisfy the condition (\ref{c-2})) in $A^2(u,\bar{u})$ can be handled by   $(\beta,0)$-type symmetric property of $f^{w_0}_r(u,\bar{u})$. Then $A^2(u,\bar{u})$ is transformed  into
a new  form,  the  coefficients of which satisfy (\ref{c-2}). Thus, the estimate of $A^2(u,\bar{u})$   can be obtained by Corollary \ref{r1}.

\noindent
 Now   the details of deleting the unbounded terms in $A^2(u,\bar{u})$ are given in the follows.
  For any $(l^0,k^0,i^0)\in{\cal A}_{(f^{w_0})^i_{r,lk}}$, it holds that
 \begin{eqnarray}\label{A2l}
  (l_t^0-k_t^0)\cdot t&=&\underbrace{\sum_{n\neq j}\big((k-k^0)_n-(l-l^0)_n\big)\cdot n+i-\sum_{n\neq t}(l_n^0-k_n^0)\cdot n}_{R^{w_0}(l,k,t,j,i)}\nonumber\\
 &&+ \big((k-k^0)_j-(l-l^0)_j\big)\cdot j.
 \end{eqnarray}
Using (\ref{A2l}) \begin{eqnarray}\label{A2}
&&A^2(u,\bar{u})\nonumber\\
&=&2{\bf Re}\sum_{i\in M_{f^{w_0}_r}\subseteq\mathbb{Z},} \sum_{|l+k|=r\atop{{\cal M}(l,k)=i}}{\bf i}(l_j-l^0_j-k_j+k^0_j)\langle j\rangle^{p} \nonumber\\
&&\quad\cdot \sum_{( {l}^0,{k}^0,i^0)\in {\cal A}_{(f^{w_0})^i_{r,lk} }}\  {(f^{w_0})}^{i( {l}^0,{k}^0,i^0)}_{r,lk} \sum_{t\in\mathbb{Z}^*} R^{w_0}(l,k,t,j,i) \langle t\rangle^{p}  u^l\bar{u}^k\nonumber\\
&+&2{\bf Re}\sum_{i\in M_{f^{w_0}_r}\subseteq \mathbb{Z},} \sum_{|l+k|=r\atop{{\cal M}(l,k)=i}}{\bf i}(l^0_t-k^0_t)\langle t\rangle^{p}  \nonumber\\
&&\cdot\sum_{( {l}^0,{k}^0,i^0)\in {\cal A}_{(f^{w_0})^i_{r,lk} }}\  {(f^{w_0})}^{i( {l}^0,{k}^0,i^0)}_{r,lk} \sum_{t\in\mathbb{Z}^*}  \big((k-k^0)_j- (l-l^0)_j\big)\cdot j \langle j\rangle^{p}   u^l\bar{u}^k.
\end{eqnarray}
Since the two parts in the right side of  (\ref{A2}) are real value functions,  they are invariant under complex conjugation. Taking complex conjugation to the second part of the right side of (\ref{A2}) and using that fact that $f_r^{w_0}(u,\bar{u})$ has $(\beta,0)$-type symmetric coefficients, it leads to  
\begin{eqnarray}\label{con}
&&\overline{2{\bf Re}\sum_{i\in M_{f^{w_0}_r}\subseteq\mathbb{Z}} \sum_{|l+k|=r\atop{{\cal M}(l,k)=i}}{\bf i}(l^0_t-k^0_t)\langle t\rangle^{p}}\nonumber\\
 &&\cdot \overline{\sum_{( {l}^0,{k}^0,i^0)\in {\cal A}_{(f^{w_0})^i_{r,lk} }}\  {(f^{w_0})}^{i( {l}^0,{k}^0,i^0)}_{r,lk} \sum_{t\in\mathbb{Z}^*}  (k_j-k^0_j- l_j+l^0_j)\cdot j \langle j\rangle^{p} u^l\bar{u}^k}\nonumber\\
&=&-2{\bf Re}\sum_{-i\in M_{f^{w_0}_r}\subseteq\mathbb{Z}} \sum_{|l+k|=r\atop{{\cal M}(k,l)=-i}}{\bf i}(l^0_t-k^0_t)\langle t\rangle^{p}\nonumber\\
  &\cdot&\sum_{( k-k^0,l-{l}^0,i^0-i)\in {\cal A}_{(f^{w_0})^{-i}_{r,kl} }}\  (f^{w_0})^{-i( k-k^0,l-{l}^0,i^0-i)}_{r,kl} \sum_{t\in\mathbb{Z}^*}  (k_j-k^0_j- l_j+l^0_j)\cdot j \langle j\rangle^{p}   u^k\bar{u}^l\nonumber\\
&=&-A^{2}(u,\bar{u}).
\end{eqnarray}
Together with (\ref{A2}) and   (\ref{con}), it follows
\begin{eqnarray}\label{a2}
 A^2({u,\bar{u}})
&=&{\bf Re}\sum_{i\in M_{f^{w_0}_r}\subseteq\mathbb{Z},} \sum_{|l+k|=r\atop{{\cal M}(l,k)=i}}{\bf i}(l_j-l^0_j-k_j+k^0_j)\langle j\rangle^{p} \nonumber\\
 &&\cdot\sum_{( {l}^0,{k}^0,i^0)\in {\cal A}_{(f^{w_0})^i_{r,lk} }}\  (f^{w_0})^{i( {l}^0,{k}^0,i^0)}_{r,lk} \sum_{t\in\mathbb{Z}^*} R^{w_0}(l,k,t,j,i)\langle t\rangle^{p}u^l\bar{u}^k.
\end{eqnarray}

\vspace{12pt}

\noindent
In the case $\theta=1$,   
\begin{eqnarray}\label{4.30}
\{f^{w_1}_r(u,\bar{u}),\|u\|_{p}^2\}_{w_1}&=&Q^+(u,\bar{u})+Q^-(u,\bar{u}),
\end{eqnarray}
where
\begin{eqnarray*}
&&Q^+(u,\bar{u}):= \sum_{|l+k|=r \atop{{\cal M}(l,k)=i\in M_{f^{w_1}_r}}}{\bf i}\mbox{sgn}(j) l_j\cdot \langle j\rangle^{2p} \prod_{n\in\mathbb{Z}^*}\langle n\rangle^{\frac12(l_n+k_n)} (\tilde{f}^{w_1})^{i}_{r,{l k} }  u^{l} \bar{u}^{k},\nonumber\\
&& Q^-(u,\bar{u}):=- \sum_{|l+k|=r \atop{{\cal M}(l,k)=i\in M_{f^{w_1}_r}}}{\bf i}\mbox{sgn}(j)   k_j\cdot \langle j\rangle^{2p} \prod_{n\in\mathbb{Z}^*}\langle n\rangle^{\frac12(l_n+k_n)} (\tilde{f}^{w_1})^{i}_{r,{l k} }  u^{l} \bar{u}^{k}.
 \end{eqnarray*}
Since the coefficients of $f^{w_1}_r(u,\bar{u})$ are $(\beta,1)$-type symmetric, it holds that
\begin{eqnarray}\label{o^+}
\overline{Q^-(u,\bar{u})}=
 Q^+(u,\bar{u}).
\end{eqnarray}
 From (\ref{4.30}) and (\ref{o^+}), one has 
 \begin{eqnarray}\label{4.32-a}
&&\{f^{w_1}_r,\|u\|_{p}^2\}_{w_1}=2{\bf Re} Q^+(u,\bar{u})\nonumber\\
&=&2{\bf Re} \sum_{i\in M_{f^{w_1}_r}} \sum_{|l+k|=r \atop{{\cal M}(l,k)=i}}{\bf i}\mbox{sgn}(j)   l_j\cdot \langle j\rangle^{2p} \prod_{n\in\mathbb{Z}^*}\langle n\rangle^{\frac12(l_n+k_n)} (\tilde{f}^{w_1})^{i}_{r,{l k} }  u^{l} \bar{u}^{k}\nonumber\\
&=&2{\bf Re} \sum_{i\in M_{f^{w_1}_r}}{\bf i}\sum_{|l+k|=r\atop{{\cal M}(l,k)=i}} (\tilde{f}^{w_1})_{r,lk}^{i}\langle j\rangle^{2p-1}l_j\cdot j\prod_{n\in\mathbb{Z}^*}\langle n\rangle^{\frac{l_n+k_n}2} u^l\bar{u}^k,
 \end{eqnarray}
the last equation is obtained by $\mbox{sgn}(j) |j|=j$.  For any $l,k\in \mathbb{N}^{\mathbb Z^*}$ and any $j\in\mathbb{Z}^*$ with ${\cal M}(l,k)=i\in M_{f_r^{w_1}}$ and $l_j\neq 0$,  one has  that
 \begin{equation}\label{sjs}
 l_j\cdot j=i-{\cal M}(l-l_j e_j,k)=i-\sum_{n\in\mathbb{Z}^*\atop{n\neq j}}l_n\cdot n+\sum_{n\in \mathbb{Z}^*}k_n\cdot n.
 \end{equation}
  Together with (\ref{4.32-a}) and (\ref{sjs}),  
  \begin{eqnarray}
 \{f^{w_1}_r(u,\bar{u}),\|u\|_p^2\}_{w_1}=E(u,\bar{u})+B(u,\bar{u})+D(u,\bar{u}),
\end{eqnarray}
where {\footnotesize
 \begin{eqnarray}\label{E()}
    E(u,\bar{u}):=2  {\bf Re}   \sum_{|l+k|=r\atop{{\cal M}(l,k)=i\in M_{f^{w_1}_r}}} {\bf i} (\tilde{f}^{w_1})_{r,lk}^{i}\mbox{sgn}(l_j) \langle j\rangle^{2p-1}(\sum_{n\in\mathbb{Z}^*}k_n \cdot n) \prod_{m\in\mathbb{Z}^*}\langle m\rangle^{\frac12(l_m+k_m)}     u^l\bar{u}^k,
    \\
\label{B()}  B(u,\bar{u}):=-2 {\bf Re}   \sum_{|l+k|=r\atop{{\cal M}(l,k)=i\in M_{f^{w_1}_r}}} {\bf i}(\tilde{f}^{w_1})_{r,lk}^{i} \mbox{sgn}(l_j)\langle j\rangle^{2p-1}( \sum_{n\in\mathbb{Z}^*\atop{n\neq j}}l_n \cdot n )\prod_{m\in\mathbb{Z}^*}\langle m\rangle^{\frac12(l_m+k_m)}    u^l\bar{u}^k,  \\
\label{D()}  D(u,\bar{u}):=2 {\bf Re}  \sum_{|l+k|=r\atop{{\cal M}(l,k)=i\in M_{f^{w_1}_r}}} {\bf i}(\tilde{f}^{w_1})_{r,lk}^{i}\mbox{sgn}(l_j) \langle j\rangle^{2p-1}i\prod_{m\in\mathbb{Z}^*}\langle m\rangle^{\frac12(l_m+k_m)} u^l\bar{u}^k.
            \end{eqnarray}
}
 In order to estimate of $E(u,\bar{u})$ and $B(u,\bar{u})$,    $E(u,\bar{u})$ is rewritten as the sum of $E^1(u,\bar{u})$ and $E^2(u,\bar{u})$, where {\footnotesize \begin{eqnarray}\label{e1}
    \label{E1()} E^1(u,\bar{u}):
     = 2 {\bf Re} \sum_{|l+k|=r\atop{{\cal M}(l,k)=i\atop{i\in M_{f_r^{w_1}}}}} {\bf i}(\tilde{f}^{w_1})_{r,lk}^{i} \mbox{sgn}(l_j)\langle j\rangle^{p-\frac12}(\langle j\rangle^{p-\frac12} -\langle n\rangle^{p-\frac12}) \sum_{n}nk_n  \prod_{m\in\mathbb{Z}^*}\langle m\rangle^{\frac12(l_m+k_m)}     u^l\bar{u}^k , \\
\label{E2()}    E^2(u,\bar{u}): = 2 {\bf Re}\sum_{|l+k|=r\atop{{\cal M}(l,k)=i\atop{i\in M_{f_r^{w_1}}}}} {\bf i} (\tilde{f}^{w_1})_{r,lk}^{i} \mbox{sgn}(l_j)\langle j\rangle^{p-\frac12} (\sum_{n}k_n \cdot n\langle n\rangle^{p-\frac12}) \prod_{m\in\mathbb{Z}^*}\langle m\rangle^{\frac12(l_m+k_m)}   u^l\bar{u}^k
    \end{eqnarray}}
    and  $B(u,\bar{u})$ is rewritten as the sum of $B^1(u,\bar{u})$ and $B^2(u,\bar{u})$, where
   {\footnotesize \begin{eqnarray}\label{b1}
    \label{B1()} B^1(u,\bar{u})\!:=\!-2 {\bf Re}  \sum_{|l+k|=r\atop{{\cal M}(l,k)=i\atop{i\in M_{f_r^{w_1}}}}} {\bf i}(\tilde{f}^{w_1})_{r,lk}^{i}\mbox{sgn}(l_j) \langle j\rangle^{p-\frac12} \sum_{n\neq j}l_n \cdot n (\langle j\rangle^{p-\frac12}-\langle n\rangle^{p-\frac12}) \prod_{m\in\mathbb{Z}^*}\langle m\rangle^{\frac12(l_m+k_m)}    u^l\bar{u}^k , \\
    \label{B2()}  B^2(u,\bar{u}):=-2 {\bf Re} \sum_{|l+k|=r\atop{{\cal M}(l,k)=i\in M_{f_r^{w_1}}}}\!\! {\bf i}(\tilde{f}^{w_1})_{r,lk}^{i}\mbox{sgn}(l_j) \langle j\rangle^{p-\frac12}( \sum_{n\neq j}l_n \cdot n \langle n\rangle^{p-\frac12} )\prod_{m\in\mathbb{Z}^*}\langle m\rangle^{\frac12(l_m+k_m)}     u^l\bar{u}^k .  \mbox{     }
    \end{eqnarray}}
  Using the $(\beta,1)$-type symmetric property of $f^{w_1}_r(u,\bar{u})$,   delete the bad unbounded parts (not satisfy the condition (\ref{c-2})) in $B^2(u,\bar{u})$ and $E^2(u,\bar{u})$ in the followings.
    For any $l,k\in\mathbb{N}^{\mathbb{Z}^*}$ and any $j,n\in\mathbb{Z^*}$ with $k_n\neq 0$, $l_j\neq 0$ and ${\cal M}(l,k)=i$, the following equation holds true 
   \begin{equation}\label{kk}  k_n\cdot n=\underbrace{i+\sum_{t\in\mathbb{Z}^*\atop{t\neq j}}l_t\cdot t-\sum_{t\in\mathbb{Z}^*,\atop{t\neq n}}k_t\cdot t}_{R^{w_1}_+(l,k,n,j,i)}+ l_j\cdot j.
   \end{equation}
    Take  (\ref{kk}) into $E^2(u,\bar{u})$,  one obtains that 
    \begin{eqnarray}\label{non}
    &&E^2(u,\bar{u})\nonumber\\
    &=&-2{\bf Re}  \sum_{|l+k|=r\atop{{\cal M}(l,k)=i\in M_{f^{w_1}_r}}} {\bf i}(\tilde{f}^{w_1})_{r,lk}^{i} \langle j\rangle^{p-\frac12} \sum_{n\in \mathbb{Z}^*}R^{w_1}_+(l,k,n,j,i)\langle n\rangle^{p-\frac12} \prod_{m\in\mathbb{Z}^*}\langle m\rangle^{\frac12(l_m+k_m)}  u^l\bar{u}^k \nonumber\\
    &&-2 {\bf Re}  \sum_{|l+k|=r\atop{{\cal M}(l,k)=i\in  M_{f^{w_1}_r}}} {\bf i} (\tilde{f}^{w_1})_{r,lk}^{i} \langle j\rangle^{p-\frac12}\sum_{n\in \mathbb{Z}^*} l_j\cdot j\langle n\rangle^{p-\frac12} \prod_{m\in\mathbb{Z}^*}\langle m\rangle^{\frac12(l_m+k_m)} u^l\bar{u}^k.
    \end{eqnarray}
  Take complex conjugation to the second part of the right side of (\ref{non}) and get
    \begin{eqnarray}\label{322}
    &&E^2(u,\bar{u})\nonumber\\
    &=&-2 {\bf Re}  \sum_{|l+k|=r\atop{{\cal M}(l,k)=i\in  M_{f^{w_1}_r}}} {\bf i} (\tilde{f}^{w_1})_{r,lk}^{i} \langle j\rangle^{p-\frac12} \sum_{n\in\mathbb{Z}^*}R^{w_1}_+(l,k,n,j,i)|n|^{p-\frac12} \prod_{m\in\mathbb{Z}^*}\langle m\rangle^{\frac12(l_m+k_m)}    u^l\bar{u}^k \nonumber\\
    &&-2\overline{{\bf Re}   \sum_{|l+k|=r\atop{{\cal M}(l,k)=i\in M_{f^{w_1}_r}}}{\bf i} (\tilde{f}^{w_1})_{r,lk}^{i} \langle j\rangle^{p-\frac12}\sum_{n\in\mathbb{Z}^*} l_j\cdot j\langle n\rangle^{p-\frac12} \prod_{m\in\mathbb{Z}^*}\langle m\rangle^{\frac12(l_m+k_m)} u^l\bar{u}^k }\nonumber\\
    &=&-2 {\bf Re}   \sum_{|l+k|=r\atop{{\cal M}(l,k)=i\in  M_{f^{w_1}_r}}} {\bf i}(\tilde{f}^{w_1})_{r,lk}^{i} \langle j\rangle^{p-\frac12} \sum_{n\in\mathbb{Z}^*}R^{w_1}_+(l,k,n,j,i)\langle n\rangle^{p-\frac12} \prod_{m\in\mathbb{Z}^*}\langle m\rangle^{\frac12(l_m+k_m)}     u^l\bar{u}^k \nonumber\\
    &&+2  {\bf Re}  \sum_{|l+k|=r\atop{{\cal M}(l,k)=i\in  M_{f^{w_1}_r}}} {\bf i} (\tilde{f}^{w_1})_{r,kl}^{-i} \langle j\rangle^{p-\frac12}\sum_{n\in\mathbb{Z}^*} l_j\cdot j\langle n\rangle^{p-\frac12} \prod_{m\in\mathbb{Z}^*}\langle m\rangle^{\frac12(l_m+k_m)}  u^k\bar{u}^l \nonumber\\
    &=&-2 {\bf Re}  \sum_{|l+k|=r\atop{{\cal M}(l,k)=i\in  M_{f^{w_1}_r}}}{\bf i} (\tilde{f}^{w_1})_{r,lk}^{i} \langle j\rangle^{p-\frac12} \sum_{n\in\mathbb{Z}^*}R^{w_1}_+(l,k,n,j,i)\langle n\rangle^{p-\frac12} \prod_{m\in\mathbb{Z}^*}\langle m\rangle^{\frac12(l_m+k_m)}   u^l\bar{u}^k \nonumber\\
    && -E^2(u,\bar{u}),
    \end{eqnarray}
    the last equation is obtained by  the coefficients of $f^{w_1}_r(u,\bar{u})$ being $(\beta,1)$-type symmetric.
  Equation (\ref{322}) leads to
  {\small  \begin{eqnarray}\label{e2}
   && E^2(u,\bar{u})\nonumber\\
   &=&\!\!\!\!\!\!-  {\bf Re}   \!\!\!\sum_{|l+k|=r\atop{{\cal M}(l,k)=i\in M_{f^{w_1}_r}}}\!\!\!{\bf i} (\tilde{f}^{w_1})_{r,lk}^{i} \langle j\rangle^{p-\frac12} \sum_{n\in\mathbb{Z}^*}\!\!\! R^{w_1}_+(l,k,n,j,i)\langle n\rangle^{p-\frac12} \prod_{m\in\mathbb{Z}^*}\langle m\rangle^{\frac12(l_m+k_m)}      u^l\bar{u}^k. 
    \end{eqnarray} }
  Similarly, by the fact
    $$ l_n\cdot n=\underbrace{i+\sum_{t\in \mathbb{Z}^*}k_t \cdot t-\sum_{t\neq j,n}l_t \cdot t}_{R_-^{w_1}(l,k,n,j,i)}-l_j\cdot j, $$
    the following equation  holds true 
    {\small\begin{eqnarray}\label{b2}
    B^2(u,\bar{u}) 
   =-{\bf Re} \sum_{|l+k|=r\atop{{\cal M}(l,k)=i\in M_{f^{w_1}_r}}} {\bf i}(\tilde{f}^{w_1})_{r,lk}^{i} \langle j\rangle^{p-\frac12} \sum_{n\neq j}R_-^{w_1}(l,k,n,j,i) \langle n\rangle^{p-\frac12} \prod_{m\in\mathbb{Z}^*}\langle m\rangle^{\frac12(l_m+k_m)}   u^l\bar{u}^k.\nonumber\\
    \end{eqnarray}}
Summarize this step, it satisfies that  
\begin{equation}
\begin{array}{cc}
\{f^{w_0}_r,\|u\|^2_p\}_{w_0}\\
 \Updownarrow \\
 {\bf Re} O^+(u,\bar{u})=
 \end{array}
 \\
 \left\{
\begin{array}{lll}
A^0(u,\bar{u})\ \ \mbox{defined in }\ (\ref{a0})\\
+\\
A^1(u,\bar{u})\ \ \mbox{defined in }\ (\ref{a1})\\
+\\
{A^2(u,\bar{u})\ in \ (\ref{1-8})}\xrightarrow{ (\beta,0 )-\mbox{type symmetrical}} A^2(u,\bar{u})\  in\ (\ref{a2}).
 \end{array}
\right.\nonumber
\end{equation}
and
\begin{equation}
\begin{array}{cc}
\{f^{w_1}_r,\|u\|^2_p\}_{w_1}\\
 \Updownarrow \\
 {\bf Re} Q^+(u,\bar{u})=
 \end{array}
 \\
 \left\{
\begin{array}{lll}
B(u,\bar{u})\ \  in \ (\ref{B()})=\left\{
\begin{array}{lcl}
B^1(u,\bar{u}) \ in \ (\ref{B1()})\\
+\\
B^2(u,\bar{u}) \ in \ (\ref{B2()}) \xrightarrow{(\beta,1)} B^2(u,\bar{u})\ in\ (\ref{b2})\\
\end{array}\right.\\
+\\
E(u,\bar{u})\ \ in\ (\ref{E()})=\left\{
\begin{array}{lcl}
E^1(u,\bar{u}) \ in \ (\ref{E1()})\\
+\\
E^2(u,\bar{u}) \ in \ (\ref{E2()})\xrightarrow{(\beta,1)} E^2(u,\bar{u})\ in\ (\ref{e2})\\
\end{array}\right.\\
+\\
{D(u,\bar{u}) } \ in \ (\ref{D()}).
 \end{array}
\right.\nonumber
\end{equation}

\vspace{4pt}
\noindent Step 2:Estimate  $A^0(u,\bar{u})$-$A^{2}(u,\bar{u})$, $B^1(u,\bar{u})$, $B^2(u,\bar{u})$, $E^1(u,\bar{u})$, $E^2(u,\bar{u})$ and $D(u,\bar{u})$.

\vspace{4pt}

\noindent  It is clear that $A^0(u,\bar{u})$
     can be written as an inner product of the vector fields  $G:=(\langle j\rangle^p \bar{u}_j )_{j\in\mathbb{Z}} $ and
  $ F:=(F_j)_{j\in\mathbb{Z}} $, where
      \begin{eqnarray}\label{F_j}
   F_j(u,\bar{u}):=   \sum_{|l +k-e_j|=r-1 \atop{{\cal M}(l , k)=i\in  M_{f^{w_0}_r}\subseteq\mathbb{Z} }} F^i_{j,l k-e_j}  u^{l} \bar{u}^{k-e_j},
   \end{eqnarray}
and
\begin{eqnarray}
    F^i_{j,l k-e_j}:=  \big(l_j-l_j^0-k_j+k_j^0\big) \langle j\rangle^{p}\sum_{( l^0,k^0,i^0)\in {\cal A}_{(f^{w_0})^i_{r, l k}}} i^0\cdot  (f^{w_0})^{i( l^0,k^0,i^0)}_{r,{ l k} } .
   \end{eqnarray}
Noting  the fact the momentum of $(f^{w_0})_{r,lk}^iu^l\bar{u}^k$ being $i$ and $0\leq l_j^{0}\leq l_j$, $0\leq k_j^{0}\leq k_j$, the following equation holds true 
\begin{eqnarray}\label{sgnkj}
 (l_j-l_j^0-k_j+k_j^0)\cdot j=i-(l_j^0-k_j^0)\cdot j-\sum_{t\neq j}(l_t-k_t)\cdot t.
\end{eqnarray}
Using the fact that $|x|^p$ ($p\geq2$) is convex function and (\ref{a---b}), it follows that
  \begin{eqnarray}\label{l_t}
 |i-(l_j^0-k_j^0)\cdot j-\sum_{t\neq j}(l_t-k_t)\cdot t |^p \leq
2 r^{p-1}\big(\sum_{t\neq j}  (l_t+k_t) |t|^p  + (l_j^0+k_j^0) |j|^p \big)\cdot \langle i\rangle^p.
  \end{eqnarray}
In view of (\ref{sgnkj}) and (\ref{l_t}), it holds that
  \begin{eqnarray}\label{c_0}
  |l_j-l_j^0-k_j+k^0_j|\langle j\rangle^p\leq  2 r^{p-1}(\sum_{t\neq j}  (l_t+k_t)\langle t\rangle^p  + (l_j^0+k_j^0)\langle j\rangle^p\big)\cdot  \langle i\rangle^p.
  \end{eqnarray}
Since $f_r^{w_0}$ has $(\beta,0)$-type symmetric coefficients semi-bounded by $C_{f^{w_0}}$,  together with    (\ref{c_0}), the coefficients of vector field $F$ in (\ref{F_j})  are bounded by the following
 {\small \begin{eqnarray}
  |F^i_{j, l k-e_j} |
  \leq 2r^{p-1}C^{r-2}_{f^{w_0}} \big(\sum_{t\neq j}  (l_t+k_t)\langle t\rangle^p + (l_j^0+k_j^0)\langle j\rangle^p  \big)\cdot \frac{1}{\langle i\rangle^{\beta-p}}  \nonumber.
    \end{eqnarray}}
   Then by Corollary \ref{r1}, the following inequality holds true 
    \begin{eqnarray}
     |A^0(u,\bar{u})|\leq |\langle F,G\rangle|\leq \|F\|_{\ell^2}\cdot \|G\|_{\ell^2}=\|F\|_{\ell^2}\cdot\|u\|_{p}\leq 2{C^{r-2}_{f^{w_0}_r}r^{p} c^{r-1} }\|u\|^2_p\|u\|_2^{r-2}.
     \end{eqnarray}
Using the same method, one has    
 \begin{eqnarray}
    |D(u,\bar{u})|
    \leq  2{ C^{r-2}_{f^{w_1}_r}   r^{p-1}c^{r-1} }\|u\|_{p }^2\|u\|_{2 }^{r-2}.
    \end{eqnarray}

\noindent
In order to estimate  $A^1(u,\bar{u})$,  the following inequality is given 
    for any $j,\ m\in\mathbb{Z}^*$ 
 \begin{eqnarray}\label{tu}
&&\big||j|^a-|m|^a\big|\leq \big| \int_{0}^1 \frac{d|m+\theta(j-m)|^a}{d\theta} d\theta\big|\nonumber\\
&\leq & \int_0^1 a|j-m|\cdot |m+\theta(j-m)|^{a-1}d \theta\nonumber\\
&\leq &2^{a-2}a(|m|^{a-1} |j-m|+|j-m|^a),
\end{eqnarray} with $a\geq2$. 

\noindent
Take $a=p$ into (\ref{tu}). Given  $l,k\in \mathbb{N}^{\mathbb{Z}^*}$ fulfilling  ${\cal M}(l,k)=i$ with $k_j\neq 0$ and ${l}^0_m\neq 0$, together with (\ref{a---b}),  it holds that
\begin{eqnarray}\label{j-m}
 \big||j|^p-|m|^p\big|
&\leq&  2^{p-1}p\bigg( |m|^{p-1} \big(\sum_{n\neq j} k_n |n | +\sum_{n\neq m} l_n |n|  \big)\cdot \langle i\rangle \nonumber\\
&& +(r-2)^{p-1} (\sum_{n\neq j} k_n |n |^p +\sum_{n\neq m}  l_n  |n|^p)\cdot \langle i\rangle^p)\bigg) .
\end{eqnarray}
The similar inequality holds in the case $k_j\neq0$, $k^0_m\neq0$ ($m\neq j$).

\noindent 
  Since
 {\small\begin{eqnarray}\label{A1}
 |A^1(u,\bar{u})|&\leq& \big|2\sum_{i\in  M_{f^{w_0}_r}}\sum_{|l +k|=r \atop{i={\cal M}(l,k) }}  (l_j-l_j^0-k_j+k_j^0)\cdot  \langle j\rangle^{p}  \sum_{(l^0,k^0,i^0)\in {\cal A}_{(f^{w_0})^i_{r,lk}}\atop{m\in \mathbb{Z}^*}} (f^{w_0})^{i(l^0,k^0,i^0)}_{r,{l k} }  \nonumber\\
  && \cdot\big( l^0_m m( \langle j\rangle^p - \langle m\rangle^p )-k^0_m m( \langle j\rangle^p - \langle m\rangle^p) \big)  u^{l} \bar{u}^{k}\big|,
 \end{eqnarray}}
take the right side of (\ref{A1}) as an inner product of vectors $F$ and $G$, where
 $$F:=( 2 \langle j\rangle^p \bar{u}_j)_{j\in\mathbb{Z}^*} ,\quad G:= (\sum_{i\in  M_{f^{w_0}_r}}\sum_{|l +k-e_j|=r-1 \atop{j={\cal M}(l, k-e_j)-i}}  (G_j)^i_{r-1, l k-e_j}  u^{l} \bar{u}^{k-e_j})_{j\in\mathbb{Z}^*} $$
  and
{\small\begin{eqnarray}
&&   (G_j)^i_{r-1, l k-e_j}:=\nonumber\\
&&(l_j-l_j^0-k_j+k_j^0)  \sum_{( l^0,k^0,i^0)\in {\cal A}_{(f^{w_0})^i_{r,l k}} \atop{{m\in\mathbb{Z}^*} }}{(f^{w_0})}^{i(l^0,k^0,i^0)}_{r,{l k}}  \big(l^0_m m( \langle j\rangle^p - \langle m\rangle^p )-k^0_m m( \langle j\rangle^p - \langle m\rangle^p)  \big) .
  \nonumber
   \end{eqnarray}}
By (\ref{j-m}) and $f^{w_0}_r(u,\bar{u})$ having $(\beta,0)$-type symmetric coefficients semi-bounded by $C_{f^{w_0}}$,  the coefficients of $G_j$ are bounded by
 \begin{eqnarray*}
 &&|(G_j)^i_{r-1,l k-e_j}|\nonumber\\
  &\leq&\frac{C^{r-2}_{f^{w_0}_r}}{ \langle i\rangle^{\beta-p}}  2^{p-1}p(r-2)^{p-1} (\prod_{t\in\mathbb{Z}^*} \langle t\rangle^{\mbox{sgn}\big(l_t+(k-e_j)_t\big)} \sum_{m\in\mathbb{Z}^*} \langle m\rangle^{ (p-1)\mbox{sgn}\big((k-e_j)_m+l_m\big)}) .
\end{eqnarray*}
Using  {Corollary  \ref{r1}}, one has 
$$ |A^1(u,\bar{u})|\leq \|G\|_{\ell^2} \cdot \|F\|_{\ell^2}\leq    { C^{r-2}_{f^{w_0}_r}  2^{p}p(r-2)^{p-1}c^{r-1} }\|u\|_{p }^2\|u\|_{2 }^{r-2} . $$
By   the same method,      $E^1(u,\bar{u})$ and $B^1(u,\bar{u})$ satisfy the following inequalities 
    \begin{eqnarray*}
    |B^1(u,\bar{u})|\leq     {C^{r-2}_{f^{w_1}_r}  2^{p}p(r-2)^{p-1}c^{r-1} }\|u\|_{p }^2\|u\|_{2 }^{r-2}
    \end{eqnarray*}
    and
    \begin{eqnarray*}
    |E^1(u,\bar{u})|\leq   { C^{r-2}_{f^{w_1}_r}  2^{p}p(r-2)^{p-1}c^{r-1} }\|u\|_{p }^2\|u\|_{2 }^{r-2} .
    \end{eqnarray*}

\noindent Since  $A^2(u,\bar{u})$, $B^2(u,\bar{u})$ and $E^2(u,\bar{u})$ can be  estimated by the same method, I only give the details of estimate of $A^2(u,\bar{u})$ in (\ref{a2}).
 $$ |A^2(u,\bar{u})|\leq \|F\|_{\ell^2}\cdot \|G\|_{\ell^2},$$
 where $F=(F_j)_{j\in\mathbb{Z}^*}$ and $G=(G_j)_{j\in\mathbb{Z}^*}$  with $F_j=\langle j\rangle^p u_j$ and
   \begin{eqnarray}
G_j(u,\bar{u})&:=& \sum_{i\in M_{f^{w_0}_r}} \sum_{|l+k|=r\atop{j={\cal M}(l,k-e_j)-i}}{\bf i}\big((l-l^0)_j-(k-k^0)_j\big) \nonumber\\
 &&\cdot\sum_{( {l}^0,{k}^0,i^0)\in {\cal A}_{(f^{w_0})^i_{r,lk} }}\  {(f^{w_0})}^{i( {l}^0,{k}^0,i^0)}_{r,lk} \sum_{t\in\mathbb{Z}} R^{w_0}(l,k,t,j,i)\langle t\rangle^p    u^l\bar{u}^{k-e_j}.\nonumber
\end{eqnarray}
From (\ref{a---b}) and (\ref{kk}), one has  
  $$|R^{w_0}(l,k,t,j,i)|\leq2\big(\sum_{n\neq j}\big|(k-k^0)_n-(l-l^0)_n\big|\cdot \langle n\rangle+\sum_{n\neq t}|l_n^0-k_n^0|\cdot \langle n\rangle\big)\cdot \langle i\rangle .$$ Then   the coefficients of $G$ are bounded by
  \begin{eqnarray}
 && |(G_j)_{r,l k-e_j}^i|\nonumber\\
 &\leq &  2\frac{C^{r-2}_{f^{w_0}_r}}{\langle i\rangle^{\beta-1}} \big(\sum_{n\neq j}\big|(k-k^0)_n-(l-l^0)_n\big|\cdot \langle n\rangle +\sum_{n\neq t}|l_n^0-k_n^0|\cdot \langle n\rangle \big) \langle t\rangle^p  .\nonumber
  \end{eqnarray}
  Using Corollary \ref{r1},
  it holds that
  $$ |A^2(u,\bar{u})|\leq  2{C^{r-2}_{f^{w_0}_r}c^{r-1}  r }\|u\|_p\|u\|_2^{r-1} . $$
Similarly, 
  $$ |B^2(u,\bar{u})|,\ |E^2(u,\bar{u})|\leq 2 {C^{r-2}_{f^{w_1}_r}c^{r-1}  r} \|u\|_p\|u\|_2^{r-1} .$$
Thus, the following inequalities hold true 
 \begin{eqnarray*}
|\{f^{w_0}_r(u,\bar{u}),\|u\|_p^2\}_{w_0}|&\leq & |A^0(u,\bar{u})|+|A^1(u,\bar{u})|+|A^2(u,\bar{u})|\nonumber \\
 &\leq &  { C^{r-2}_{f^{w_0}_r}  2^{p+1}p r^{p-1}c^{r-1} }\|u\|_{p }^2\|u\|_{2}^{r-2} 
\end{eqnarray*}
and
\begin{eqnarray*}
|\{f^{w_1}_r(u,\bar{u}),\|u\|_p^2\}_{w_1}|&\leq & |B^1(u,\bar{u})|+ |B^2(u,\bar{u})|+|D(u,\bar{u})| +|E^1(u,\bar{u})|+ |E^2(u,\bar{u})|\nonumber\\
&\leq &  {C^{r-2}_{f^{w_1}_r}  2^{p+1}p r^{p-1} c^{r-1} }\|u\|_{p }^2\|u\|_{2}^{r-2}.
\end{eqnarray*}

\end{proof}

\section{Birkhoff Normal form and non resonant condition}
 
\subsection{$(\theta,\gamma, \alpha, N)$-normal form  }
\noindent
In order to guarantee the boundedness of the symplectic transformation, it is required a strong non resonant condition. 
Given integers $N>0$ and $r\geq3$, let
\begin{equation}\label{E-theta}
E_{r,N}:=\left\{(l,k)\ |\ l,k\in\mathbb{N}^{\mathbb Z^*}, \ 3\leq|l+k|=t\leq r,\ |\Gamma_{>N}(l+k)|\leq 2\right\}.
\end{equation}

\begin{Definition}\label{5.3}
Given $\gamma>0$, $\alpha>0$,  $\theta\in\{0,1\}$ and $N,r\in \mathbb{N}$, frequencies  $\omega=(\omega_j)_{j\in \mathbb{Z}^*}$  is said to be {\bf $r$-degree $(\theta,\gamma,\alpha,N)$-non resonant}, if for any $(l,k) $ belongs to
 {\footnotesize \begin{eqnarray}\label{o-set} {O}^{w_{\theta}}_{r,N}:=\left\{(l,k)\in E_{r,N} \ \left|\  \begin{array}{lrllr}\mbox{when }\ | \Gamma_{>N} (l+k)|< 2,\\  \theta\sum_{j\in\mathbb{Z}^*} |l_j-k_j|+(1-\theta)\sum_{j\in\mathbb{Z}^*} |l_j+l_{-j}-k_j-k_{-j}|\neq 0\\
\\
 \mbox{when }\ | \Gamma_{>N} (l+k)|= 2, \\ \sum_{|j|>N}|l_j+l_{-j}-k_j-k_{-j}|\neq 0\\
 \end{array}\right.
    \right\},\end{eqnarray}}
it satisfies
$$ |\langle \omega, \  I_{\theta}(l-k)\rangle|>\frac{\gamma M_{l,k}}{N^{\alpha}},$$
where
\begin{eqnarray}\label{*27}
M_{l,k}:=
 \max\big(\{\ |j| \ \big|\ k_j\neq 0\ \mbox{or}\ l_j\neq 0,\ k,\ l \in \mathbb{N}^{\mathbb{Z}^*}\}\cup\{N\}\big) .
 \end{eqnarray}

\end{Definition}

\noindent
With the $r$-degree $(\theta, \gamma,\alpha,N)$-non resonant condition,   a symplectic transformation will be obtained.  Under this transformation the Hamiltonian function $H(u,\bar{u})$ are transformed into  the sum of an $r$-degree normal form and a remainder term. However this $r$-degree normal form  is not a standard Hamiltonian Birkhoff normal form (A standard $2r$-degree Hamiltonian Birkhoff normal form in   variables $(u, \bar{u})$ is an $r$-degree polynomial which only depends on variables $(|u_j|^2)_{j\in \mathbb{Z}^*} $). Now I introduce a definition to describe this normal form.

\begin{Definition}\label{jr}
Given  $\gamma>0$, $\alpha>0$ and an integer $N>0$, call an $r$-degree  polynomial
$$f(u,\bar{u})=\sum_{r\geq t\geq 3}  \sum_{|l+k|=t\atop{ {\cal M}(l,k)=i\in M_{f_t} }} f^i_{t,l k}u^{l}\bar{u}^k$$
  $\bm{(\theta,\gamma, \alpha, N)}$\textbf{-normal form} with respect to $\omega \in \mathbb R^{\mathbb Z^*}$, if for any $(l,k)\in E_{r,N}$ ($E_{r,N}$ is defined in (\ref{E-theta})) with   ${\cal M}(l,k)=i\in M_{f_t}$ and $f^i_{t,lk}\neq 0$,  it
   satisfies that
 \begin{eqnarray*}
|\langle \omega, \ I_{\theta}(k-l)\rangle |\leq  \frac{\gamma M_{k,l} }{N^{\alpha }},
 \end{eqnarray*}
where $\langle \ , \   \rangle$ is the inner product of space $\ell^2(\mathbb{Z}^*
,\mathbb{C})$, $I_{\theta}$ and $M_{l,k}$ are defined in (\ref{I-theta}) and (\ref{*27}).

\end{Definition}\par

\begin{Remark}\label{B-h}
 Let $f(u,\bar{u})$ be an $r$-degree polynomial. For any given $\gamma>0$, $\alpha>0$, $\theta\in\{0,1\}$, integers $N>0$ and $\omega \in \mathbb R^{\mathbb Z^*}$, denote
$$\Gamma^{\omega}_{(\theta,\gamma,\alpha,N)} f(u,\bar{u}):=\sum_{r\geq t\geq 3,}\sum_{|l+k|=t,{\cal M}(l,k)=i} (\Gamma_{(\theta,\gamma,\alpha,N)} f)^i_{t,l k}u^{l}\bar{u}^k   $$ with
 $$(\Gamma^{\omega}_{(\theta,\gamma,\alpha,N)}f)^i_{t,l k}:=\left\{\begin{array}{ll} f^i_{t,lk},&\ \mbox{if }\ l, k\ \mbox{fulfills }\ |\langle \omega, \  I_{\theta}(l-k)\rangle|\leq \frac{\gamma M_{l,k} }{N^{\alpha}}\\
\\
0,& \mbox{the other cases} \end{array} \right. $$ as   $\bm{(\theta,\gamma, \alpha, N)}$\textbf{-normal form} of $f(u,\bar{u})$  with respect to $\omega $. Moreover,  suppose that  $f^{w_{\theta}}(u,\bar{u})$ has  $(\beta,\theta)$-type symmetric coefficients semi-bounded by $C_{\theta}>0$ ($\theta\in\{0,1\}$). So does $\Gamma^{\omega}_{(\theta,\gamma,\alpha,N)} f^{w_{\theta}}(u,\bar{u})$.
\end{Remark}

\begin{Remark}\label{norm-re}
Assume  that $\omega=(\omega_j)_{j\in\mathbb{Z}^*}$ is an $r_*$-degree $(\theta,\gamma,\alpha,N)$-non resonant frequencies  and $f^{w_{\theta}}(u,\bar{u})$ is an $r_*$-degree $(\theta,\gamma,\alpha,N)$-normal form with respect to $\omega$. Then $f^{w_{\theta}}(u,\bar{u})$ has the following form
\begin{eqnarray*}
f^{w_{\theta}}(u,\bar{u}):= {\cal A}^{w_{\theta}}(u,\bar{u}) +{\cal B}^{w_{\theta}}(u,\bar{u})+{\cal C}^{w_{\theta}}(u,\bar{u}),
\end{eqnarray*}
where

{\small \begin{eqnarray}
{\cal A}^{w_{\theta}}(u,\bar{u})&:=& \sum_{3\leq r\leq r_*} \sum_{(l,k)\in \Omega^{\theta}_{{\cal A}_r^N},i={\cal M}(l,k)} (f^{w_{\theta}})_{r,lk}^{i} u^l\bar{u}^k,\label{5.4}\\
{\cal B}^{w_{\theta}}(u,\bar{u})&:=& \sum_{3\leq r\leq r_*} \sum_{(l,k)\in \Omega^{\theta}_{{\cal B}_r^N},i={\cal M}(l,k) } (f^{w_{\theta}})_{r,lk}^{i} u^l\bar{u}^k,\\
{\cal C}^{w_{\theta}}(u,\bar{u})&:=& \sum_{3\leq r\leq r_*} \sum_{(l,k)\in \Omega^{\theta}_{{\cal C}_r^N},i={\cal M}(l,k) } (f^{w_{\theta}})_{r,lk}^{i} u^l\bar{u}^k,\label{5.6}\\ \nonumber
\end{eqnarray}}
and
$$\Omega^{\theta}_{{\cal A}_r^N}:=\left\{ \begin{array}{ll}
\{(l,k)\in E_{r,N}\ \big|\   |\Gamma_{>N} (l+k)|<2,\ \sum_{j\in\mathbb{Z}^*} |l_j+l_{-j}-k_j-k_{-j}|=0 \},& \theta=0\\
\\
\{(l,k)\in E_{r,N}\ \big|\  |\Gamma_{>N} (l+k)|<2,\ \sum_{j\in\mathbb{Z}^*} |l_j-k_j|=0 \},& \theta=1
\end{array}\right.
$$
 \begin{equation*}
\Omega^{\theta}_{{\cal B}_r^N}:= \left\{(l,k)\in E_{r,N}\ \big|\
 |\Gamma_{>N}(l+k)|=2,\  \mbox{there exists } |j_0|>N\ \mbox{such that }l_{j_0}=k_{j_0}=1\right\},
 \end{equation*} 
\begin{equation*}
\Omega^{\theta}_{{\cal C}_r^N}:= \left\{(l,k)\in E_{r,N}\ \big|\   |\Gamma_{>N}(l+k)|=2,\
 \mbox{there exists } |j_0|>N\ \mbox{such that }l_{-j_0}=k_{j_0}=1\right\}.
\end{equation*}
\end{Remark}

\begin{Lemma}\label{6.2} Let $\omega=(\omega_j)_{j\in\mathbb{Z}^*}$ be an $r$-degree $(\theta,\gamma,\alpha,N)$ non-resonant frequency. Suppose that $f_r^{w_{\theta}}(u,\bar{u}) $ is an  $r$-degree ($r\geq3$) homogeneous  $(\theta,\gamma,\alpha,N)$-normal form   with respect to $\omega$  and  has $(\beta,\theta)$-type symmetric coefficients semi-bounded by $C_{f_r^{w_{\theta}}}>0$ ($ \beta-p\geq2$).
Then  for any $(u,\bar{u})\in {{\mathcal H}^p}(\mathbb Z^*,\mathbb C)$ it has
\begin{equation}\label{f<2}
|\{f_r^{w_{\theta}}(u,\bar{u}), \|u\|_p^2  \}_{w_{\theta}}|\leq  20r^{p+1}c^{r-1} C^{r-2}_{f_r^{w_{\theta}}} N   \|\Gamma_{>N} u\|_2\cdot \|u\|_2^{r-3}\cdot \|u\|_p 
\end{equation}
\end{Lemma}
\begin{Remark}
Although   $f_r^{w_{\theta}}(u,\bar{u})$ in Lemma \ref{6.2} is at most 2 degree about $(\Gamma_{>N} u,\Gamma_{>N}\bar{u})$,  it still satisfies an inequality  similar to  (\ref{fgamma}) in  Corollary \ref{4.3}. But  for 
 the general polynomials being at most 2 degree about $(\Gamma_{>N} u,\Gamma_{>N}\bar{u})$, this inequality dose not always hold. 
\end{Remark}
\begin{proof}[Proof of Lemma \ref{6.2}]
From  Remark \ref{norm-re},
\begin{eqnarray}
f_r^{w_{\theta}}(u,\bar{u})={\cal A}^{w_{\theta}}_r(u,\bar{u})+
 {\cal B}_r^{ w_{\theta}}(u,\bar{u})+{\cal C}_r^{ w_{\theta}}(u,\bar{u}),
 \end{eqnarray}
 where  ${\cal A}^{w_{\theta}}_r(u,\bar{u})$, ${\cal B}_r^{ w_{\theta}}(u,\bar{u})$ and ${\cal C}_r^{ w_{\theta}}(u,\bar{u}) $ are  defined in (\ref{5.4})-(\ref{5.6}) in Remark \ref{norm-re}.

 \vspace{4pt}
 \noindent Step 1: Calculate $\{ {\cal A}_r^{w_{\theta}}(u,\bar{u}) ,\|u\|_p^2\}_{w_{\theta}} $.
 \vspace{4pt}

 \noindent
 It is easy to verify that    
 \begin{eqnarray}\label{cala}
 \{ {\cal A}_r^{w_0}(u,\bar{u}) ,\|u\|_p^2\}_{w_0}
= \sum_{j\in \mathbb{Z}^*} \sum_{(l,k)\in\Omega^0_{{\cal A}_r^N} }{\bf i}(l_j+l_{-j}-k_j-k_{-j}) \langle j\rangle^{2p}(f^{w_0})_{r,lk}^{i} u^l\bar{u}^k=0;
 \end{eqnarray}
and  
 \begin{eqnarray}\label{cala-1}
 \{ {\cal A}^{w_1}_r(u,\bar{u}) ,\|u\|_p^2\}_{w_1}= \sum_{j\in \mathbb{Z}^*} \sum_{(l,k)\in\Omega^1_{{\cal A}_r^N} }{\bf i}\ \mbox{sgn}(j)\cdot(l_j-k_j) \langle j\rangle^{2p}(f^{ w_1})_{r,lk}^{i} u^l\bar{u}^k=0.
 \end{eqnarray}

\vspace{4pt}
 \noindent Step 2: Estimate $\{ {\cal B}_r^{w_{\theta}}(u,\bar{u}) ,\|u\|_p^2\}_{w_{\theta}} $.
 \vspace{4pt}

 \noindent
Since the function $ {\cal B}_r^{ w_{\theta}}(u,\bar{u})$  depends   on  $(u_j,\bar{u}_j)_{|j|\leq N} $ and $(|u_j|^2)_{|j|>N}$, the following equation holds true 
$$  \{ {\cal B}^{w_{\theta}}_r(u,\bar{u}),\sum_{|j|> N}\langle j\rangle^{2p} |u_j|^2\}_{w_{\theta}}=0.$$ Thus, 
\begin{eqnarray}
&&\{{\cal B}^{ w_{\theta}}_r(u,\bar{u}),\|u\|_p^2 \}_{w_{\theta}}=\{ {\cal B}^{ w_{\theta}}_r(u,\bar{u}),\sum_{|j|\leq N}\langle j\rangle^{2p} |u_j|^2\}_{w_{\theta}} .
        \end{eqnarray}
From the definition of $\{ , \}_{w_{\theta}}$ and the structure of ${\cal B}^{w_{\theta}}_r(u,\bar{u})$,   $\{{\cal B}^{w_{\theta}}_r(u,\bar{u}),\sum_{|j|\leq N}\langle j\rangle^{2p} |u_j|^2 \}_{w_{\theta}}$  still depends on $(|u_j|^2)_{|j|>N}$ and $(\Gamma_{\leq N} u,\Gamma_{\leq N}\bar{u}) $. To be more specific,
\small{\begin{eqnarray}\label{calB-u}
\{ {\cal B}^{ w_{\theta}}_r(u,\bar{u}),\sum_{|j|\leq N}\langle j\rangle^{2p} |u_j|^2\}_{w_{\theta}}
= 2{\bf Re} \sum_{(l,k)\in \Omega^{\theta}_{{\cal B}_r^N}, |j|\leq N}  {\bf i}\
 \mbox{sgn}^{\theta}(j)  l_{j} \langle j\rangle^{2p} \big(f^{w_{\theta}}\big)^i_{r,lk} u^l\bar{u}^k.
\end{eqnarray}}
   For any no zero term $u^l\bar{u}^k$ of the right side of (\ref{calB-u}),  its index $(l,k)$  satisfies
   \begin{eqnarray}\label{M(l,k)}
{\cal M}(l,k)={\cal M}(\Gamma_{\leq N}l,\Gamma_{\leq N}k) =\sum_{|j|\leq N}(l_j-k_j)\cdot j =i.
\end{eqnarray}
From (\ref{M(l,k)}), for any $|j|\leq N$ with $l_{j}\neq 0$, it satisfies that  
\begin{equation}\label{ghgh}
 j =\sum_{|t|\leq N,t\neq j}(l_t-k_t)\cdot t +(l_j-1-k_j)\cdot j-i.  
 \end{equation}
Moreover, given $p\geq2$, by (\ref{a---b}) and (\ref{ghgh}), the following inequalities hold true 
\begin{eqnarray}\label{m-l-k-0}
 | j|^p\leq 2
 r^p\big( \sum_{|t|\leq N,t\neq j}(l_t+k_t)\cdot |t|^p +(l_j-1+k_j)\cdot |j|^p\big)\cdot  \langle i\rangle^p  
 \end{eqnarray}
 and
\begin{eqnarray}\label{N-p}
 | j|^{2p}&\leq& |j|\cdot |j|^{p-\frac12}\cdot |j|^{p-\frac12}  \nonumber\\
 &\leq&2Nr^{p-\frac12}|j|^{p-\frac12}\big( \sum_{|t|\leq N,t\neq j}(l_t+k_t)\cdot |t|^{p-\frac{1}{2}} +(l_j-1+k_j)\cdot |j|^{p-\frac12}\big)\cdot \langle i\rangle^{p-\frac12} .
 \end{eqnarray}
 I will estimate the coefficients of $\{{\cal B}_r^{w_{\theta}}(u,\bar{u}),\|u\|_p^2 \}_{w_{\theta}}$. When $\theta =0$, for any $(l^0,k^0,i^0)\in{\cal A}_{(f^{w_0})^i_{r,lk}} $ with $(l,k)\in\Omega^{\theta}_{{\cal B}_r^N}$, it holds that
{ \begin{eqnarray}\label{m-l-k}
|{\cal M}(l^0,k^0)-\frac{i^0}2|&\leq& \sum_{|j|\leq N}(l^0_j+k^0_j)\cdot |j|+ \sum_{|j_0|>N} |\mbox{sgn}(l_{j_0}+k_{j_0})j_0|+|\frac{i^0}2|\nonumber\\
&\leq& (r-2)N+\sum_{|j_0|>N}|\mbox{sgn}(l_{j_0}+k_{j_0})j_0|+|i^0|\nonumber\\
&\leq & 4  (r-2)N\cdot \big(\sum_{|j_0|>N}|\mbox{sgn}(l_{j_0}+k_{j_0})j_0|\big)\cdot\langle i^0\rangle
\end{eqnarray}}
the last inequality hold by (\ref{a---b}). 
From (\ref{m-l-k-0}) and (\ref{m-l-k}), the coefficients of $\{{\cal B}_r^{w_{0}}(u,\bar{u}),\|u\|_p^2 \}_{w_{0}}$ are bounded by
{\small\begin{eqnarray}
 8\frac{C^{r-2}_{f^{w_0}}}{\langle i\rangle^{\beta-p}} r^{p+1}N l_j \langle j\rangle^p \big( \!\!\sum_{|t|\leq N,t\neq j}(l_t+k_t)\!\cdot\! | t|^p +(l_j-1+k_j)\cdot | j|^p\big)  
\!\cdot\!\big(\sum_{|j_0|>N}|\mbox{sgn}(l_{j_0}+k_{j_0})j_0|\big) .
\end{eqnarray}}Similarly, in the case $\theta =1$, using (\ref{N-p}), the coefficients of $\{{\cal B}_r^{w_{1}}(u,\bar{u}),\|u\|_p^2 \}_{w_{1}}$ are bounded by
\begin{equation}
 8\frac{C^{r-2}_{f_r^{w_1}}}{\langle i\rangle^{\beta -p+\frac12} } r^pN l_j|j|^{p-\frac12}\big( \sum_{|t|\leq N,t\neq j}(l_t+k_t)\cdot |t|^{p-\frac12} +(l_j-1+k_j)\cdot |j|^{p-\frac12}\big) \prod_{t\in \mathbb{Z}^*}|t|^{\frac{l_t+k_t}2}.
\end{equation}
By Corollary \ref{r1}, it holds that
\begin{eqnarray}\label{5.20}
|\{ {\cal B}^{w_{\theta}}_r(u,\bar{u}),\|u\|_p^2\}_{w_{\theta}}|\leq  8  r^{p+1} N c^{r-1}{C^{r-2}_{f_r^{w_{\theta}} }}  \|\Gamma_{>N}u\|_2^2\ \|u\|_p^{r-2}.\end{eqnarray}

\vspace{4pt}
 \noindent Step 3: Estimate $\{ {\cal C}_r^{w_{\theta}}(u,\bar{u}) ,\|u\|_p^2\}_{w_{\theta}} $.
 \vspace{4pt}

 \noindent

When $\theta=0$,
\begin{eqnarray}
&&|\{{\cal C}_r^{w_0}(u,\bar{u}),\|u\|_p^2\}_{w_0}|\nonumber\\
&=&\{ {\cal C}_r^{w_0}(u,\bar{u}), \sum_{|t|\leq N}|u_t|^2 \langle t\rangle^{2p}\}_{w_0}
+\{ {\cal C}_r^{w_0}(u,\bar{u}), \sum_{|t|> N}|u_t|^2 \langle t\rangle^{2p}\}_{w_0}\nonumber\\
&=& \{  {\cal C}_r^{w_0}(u,\bar{u}), \sum_{|t|\leq N}|u_t|^2 \langle t\rangle^{2p}\}_{w_0}+0,
\end{eqnarray}
the last equation  holds by the fact that
$$ \{u_{j}\bar{u}_{-j}, |j|^{2p}(|u_j|^2+|u_{-j}|^2)\}_{w_0}=0.$$
It is easy to verify that $\{  {\cal C}_r^{w_{0}}(u,\bar{u}), \sum_{|t|\leq N}|u_t|^2 \langle t\rangle^{2p}\}_{w_0}$ is still dependent  on $(u_{j_0}\bar{u}_{-j_0})_{|j_0|>N}$ and $(\Gamma_{\leq N} u, \Gamma_{\leq N}\bar{u})$. Using the method of estimate $\{{\cal B}_r^{ (w_{\theta})}(u,\bar{u}), \sum_{|t|\leq N}|u_t|^2 \langle t\rangle^{2p} \}_{w_{\theta}}$, the estimate of $\{  {\cal C}_r^{w_0}(u,\bar{u}), \sum_{|t|\leq N}|u_t|^2 \langle t\rangle^{2p}\}_{w_0} $ is obtained.

When $\theta=1$,
  \begin{eqnarray}
&&|\{{\cal C}_r^{ w_1}(u,\bar{u}),\|u\|_p^2\}_{w_1}|\nonumber\\
&=& \{  {\cal C}_r^{w_{1}}(u,\bar{u}),   \sum_{|t|\leq N}|u_t|^2 \langle t\rangle^{2p}\}_{w_1} +\{{\cal C}_r^{w_{1}}(u,\bar{u}), \sum_{|t|> N}|u_t|^2 \langle t\rangle^{2p}\}_{w_1}.\nonumber
\end{eqnarray}
Using the method of estimate $\{{\cal B}_r^{w_{\theta}}(u,\bar{u}), \sum_{|t|\leq N}|u_t|^2 \langle t\rangle^{2p} \}_{w_{\theta}}$ in step 2, the estimate of

\noindent
$  \{  {\cal C}_r^{w_{1}}(u,\bar{u}),   \sum_{|t|\leq N}|u_t|^2 \langle t\rangle^{2p}\}_{w_1}$ can be obtained.
  The   estimate of
   $\{{\cal C}_r^{w_{1}}(u,\bar{u}), \sum_{|t|> N}|u_t|^2 \langle t\rangle^{2p}\}_{w_1}$ will be obtained by the following. For any nonzero term of  ${\cal C}_r^{w_1}(u,\bar{u})$ with index $(l,k)$,  there exists $|j_0|>N$ with $l_{j_0}=1$, $k_{-j_0}=1$ (or $l_{-j_0}=1$, $k_{j_0}=1$) such that
\begin{eqnarray}\label{6.41-1}
2|j_0| = |\sum_{|t|\leq N} (l_t-k_t)\cdot t-i|,
\end{eqnarray}
which follows from $ {\cal M}(l,k)=i$. From the relation (\ref{6.41-1}), using (\ref{a---b}) it holds that
\begin{equation}\label{jkb1}
|j_0|\leq \frac12\big(\sum_{|t|\leq N}(l_t+k_t)\cdot |t|+|i|\big)\leq  \langle i\rangle \big( \sum_{|t|\leq N}(l_t+k_t)\cdot |t|\big) \leq rN \langle i\rangle
\end{equation}
and
\begin{equation}\label{jkb}
\langle j_0\rangle^{p-\frac12} \leq  2 r^{p-\frac12}\big(\sum_{|t|\leq N} (l_t+k_t)\cdot
\langle t\rangle^{p-\frac12}\big) \cdot \langle i\rangle^{p-\frac12}.
\end{equation}
By  (\ref{jkb1}) and (\ref{jkb}),  the coefficients of \begin{eqnarray}
 |\{ {\cal C}_r^{w_{1}}(u,\bar{u}), \sum_{|t|> N}|u_t|^2 \langle t\rangle^{2p}\}_{w_1}|= |  {\bf Re}  \sum_{(l,k)\in \Omega^1_{{\cal C}_r^N} } {\bf i}\langle j_0\rangle^{2p} (\tilde{f}^{ w_1})_{r,lk}^{i} \prod_{t\in \mathbb{Z}^*}|t|^{\frac{l_t+k_t}{2}} u^l\bar{u}^k|
\end{eqnarray}  are smaller than
\begin{eqnarray}
&&2\langle j_0\rangle^{p-\frac12} r^{p+\frac12}N\frac{C^{r-2}_{f_r^{w_1}}}{\langle i\rangle^{\beta-p+\frac12}}\big(\sum_{|n|\leq N}( l_n + k_n)\langle n\rangle^{p-\frac12}\big) \prod_{t\in \mathbb{Z}^*}|t|^{\frac{l_t+k_t}{2}} .
\end{eqnarray}
Using Corollary \ref{r1}, it holds that 
\begin{eqnarray}\label{5.23}
|\{ {\cal C}^{ w_{\theta}}_r(u,\bar{u}),\|u\|_p^2 \}_{w_{\theta}}|\leq 12r^{p+1}c^{r-1} C^{r-2}_{f_r^{w_{\theta}}} N   \|\Gamma_{>N} u\|_2\cdot \|u\|_2^{r-3}\cdot \|u\|_p.
\end{eqnarray}
Summing (\ref{cala}), (\ref{cala-1}), (\ref{5.20}) and (\ref{5.23}), inequality (\ref{f<2})
 is obtained.
 \end{proof}

\subsection{ Birkhoff normal form theorem}
In this subsection, construct  a coordinate transformation under which  the Hamiltonian system (\ref{uj}) will have an $r_*+3$ degree  $(\theta,\gamma, \alpha, N)$-normal form, for any given positive $r_*$.

\begin{Theorem}[Birkhoff normal form theorem]\label{Th2}Suppose that  system (\ref{uj})   satisfies assumptions
${\cal A}_{\theta}$-${\cal B}_{\theta}$ and 
$ P^{w_{\theta}}(u,\bar{u})$ in $H^{w_{\theta}}(u,\bar{u})$ defined in (\ref{H-1})   has  $(\beta,\theta)$-type symmetric coefficients semi-bounded by $C_{\theta}>0$ ($\beta$ is big enough).
Given $\alpha>1$, $0<\gamma\ll1$ and integer $r_*> 0$, take $p$ satisfying  $(\beta-4)/2>p>2+4\alpha(r_*+1)r_*^2$. There exist a positive real number $\widetilde{R}>0 $ and  a Lie-transformation ${\cal T}_{w_{\theta}}^{(r_*)}$: $B_p(\widetilde{R}/3)\rightarrow B_p(\widetilde{R})$ such that: \par
For any $R<\widetilde{R}$ and any integer $N$ fulfilling
\begin{equation}\label{N-c}
( {R^{r_*-2} \gamma^{r_*+1}})^{-\frac{1}{p-2-2\alpha(r_*+1)}} \leq N\leq (\gamma R^{-\frac{1}{(r_*+1)(r_*+2)}})^{\frac{1}{2\alpha}},
\end{equation}  the transformation ${\cal T}_{w_{\theta}}^{(r_*)}$
 puts Hamiltonian $H^{w_{\theta}}$
into
 \begin{equation*}
 H^{(r_*,w_{\theta})} :=H^{w_{\theta}}\circ {\cal T}^{(r_*,w_{\theta})}=H_0^{w_{\theta}}+Z^{(r_*,w_{\theta})}+{\cal R}^{N(r_*,w_{\theta})}+{\cal R}^{T(r_*,w_{\theta})}
 \end{equation*}
 which satisfies that
\begin{itemize}
\item[1)]  Both $Z^{(r_*,w_{\theta})}$ and $\mathcal R^{N(r_*,w_{\theta})}$ are $(r_*+3)$-degree  polynomials
and $\mathcal R^{T(r_*,w_{\theta})}$ is a power series  which starts with $r_*+4$ degree polynomial.
All of them have $(\beta,\theta)$-type  symmetric
coefficients  semi-bounded by $C({\theta},r_*):=C_{\theta} \big(\frac{2^{\beta}N^{2\alpha}}{\gamma}\big)^{(r_*+1)}$.  
\item[2)] The polynomial $Z^{(r_*,w_{\theta})}(u,\bar{u})$ is $r_*+3$-degree $(\theta, \gamma, \alpha,  N)$-normal form with respect to $\omega^{w_{\theta}}$.

 \item[3)] The polynomial  ${\mathcal R}^{N(r_*,w_{\theta})}(u,\bar{u}):=\sum_{r=3}^{r_*+3}\Gamma_{>2}^Ng_{r+4}^{(r,w_{\theta})}$, where $g_{r+4}^{(r,w_{\theta})}$ is an $(r+4)$-degree homogeneous polynomial with $(\beta,\theta)$-type symmetric coefficients semi-bounded by $C({\theta},r)$;

\item[4)] The canonical Lie-transformation ${\cal T}^{(r_*)}_{w_{\theta}}$ satisfies
\begin{equation}\label{sss}
\sup_{(u,\bar{u})\in B_p(R/3)}\|{\cal T}_{w_{\theta}}^{(r_*)}(u,\bar{u})-(u,\bar{u})\|_{p} \leq C(\theta, p,r_*)R^{2-\frac{1}{2(r_*+1)^2}},
\end{equation}
where $C(\theta,p,r_*)$ is a constant dependent on $ \theta,p$ and $r_*$. 
\end{itemize}
\end{Theorem}

\subsection{Important Lemmas in the Proof of Theorem \ref{Th2}}
In order to prove Theorem \ref{Th2}, it need  not only to construct  a bounded canonical
transformation under which the Hamiltonian $H^{w_{\theta}}$ in (\ref{H-1}) has an $r_*+3$-degree $(\theta,\gamma,\alpha, N)$-normal form, but also to show that the new Hamiltonian function has  $(\beta, \theta)$-type symmetric coefficients semi-bounded by $C(\theta,r_*)$. First, let us review the definition of canonical transformation.

\begin{Definition}Call a map  $\phi:{\cal H}^p(\mathbb{Z}^*,\mathbb{C}) \ni (u,\bar{u})\rightarrow (\xi,\bar{\xi})\in {\cal H}^p(\mathbb{Z}^*,\mathbb{C}) $  canonical transformation under a symplectic form $w_{\theta}$ (or a symplectic change of coordiantes), if
$\phi$ is a diffeomorphism and   preserves the Poisson bracket, i.e.
$ \{f,g\}_{w_{\theta}}\circ
      \phi=\{f\circ \phi, g\circ \phi\}_{w_{\theta}}.$
\end{Definition}

 A convenient way of constructing canonical transformations is as followings.
  Let
  $\Phi_{S^{w_{\theta}}}^t$ be the
flow generated by a regular function $S^{w_{\theta}}(u,\bar{u})$ defined in $ {\cal H}^p(\mathbb{Z}^*,\mathbb{C}) $ with respect to the symplectic structure $w_{\theta}$. $\Phi_{S^{w_{\theta}}}^t|_{t=0}=id$. If $\Phi_{S^{w_{\theta}}}^t$ is well defined up to $t=1$, then the
map $ \Phi_{S^{w_{\theta}}}^t|_{t=1}$ is called a Lie transformation  associated
to $S^{w_{\theta}}(u,\bar{u})$ under symplectic form $w_{\theta}$.  $\Phi_{S^{w_{\theta}}}^1$ is canonical.

Given a regular function $g$, the new function $g\circ {\Phi}_{S^{w_{\theta}}}^t$ satisfies
 $$\frac{d^n}{dt^n}(g\circ \Phi_{S^{w_{\theta}}}^t) =\underbrace{\{\{g, S^{w_{\theta}}\}_{w_{\theta}},\cdots   \}_{w_{\theta}}}_{n \  \  times}\circ \Phi_{S^{w_{\theta}}}^t. $$
Thus the Taylor expansion of $g\circ \Phi_{S^{w_{\theta}}}^t$ in  the   variable $t$ is
 $$g\circ \Phi_{S^{w_{\theta}}}^t= \sum_{\nu=0}^{\infty} g_{( \nu, S^{w_{\theta}})} \circ \Phi_{S^{w_{\theta}}}^t\big|_{t=0} t^{\nu}= \sum_{\nu=0}^{\infty} g_{( \nu, S^{w_{\theta}})}   t^{\nu}, $$
 where
 \begin{equation}\label{d}
 g_{(0,S^{w_{\theta}})} :=g,\quad g_{( \nu,S^{w_{\theta}})}:=\frac{1}{\ \nu\ }\{g_{( \nu-1, S^{w_{\theta}} )},S^{w_{\theta}}\}_{w_{\theta}},\ \nu\geq1.
\end{equation}
Take $t=1$ and it follows that
  $$ g\circ \Phi_{S^{w_{\theta}}}^t|_{t=1}=\sum_{\nu=0}^{\infty}g_{( \nu, S^{w_{\theta}})}. $$

In this paper,  denote $\prec$ as $\leq \tilde{C}\cdot$, where $\tilde{C}>0$ is independent of $R$. To improve the order of the $(\theta,\gamma,\alpha,N)$-normal form of $H^{w_{\theta}}$, it  needs to solve a linear equation to find a  suitable generated function  $S^{w_{\theta}}$ under symplectic form $w_{\theta}$. The following lemma is to do this with respect to $w_{\theta}$-Possion bracket.

\begin{Lemma}\label{lem2}{\bf (Homological Equation)}
Given an integer $N>0$, real numbers $\gamma>0$ and $ \alpha>0$, suppose that an $r$-degree homogeneous polynomial $g^{w_{\theta}}(u,\bar{u})$
 has $(\beta,\theta)$-type symmetric  coefficients semi-bounded by $C_{g^{w_{\theta}}}>0$ ($\theta\in\{0,1\}$).
   Then there exists an unique
       $ S^{w_{\theta}} (u,\bar{u})  $
   such that
\begin{equation}\label{47}
\{H_0^{w_{\theta}},S^{w_{\theta}}(u,\bar{u})\}_{w_{\theta}}+ \Gamma^{\omega_{\theta}}_{(\theta,\gamma,\alpha,N)} \Gamma_{\leq 2}^Ng^{w_{\theta}}(u,\bar{u})= \Gamma_{\leq 2}^N g^{w_{\theta}} (u,\bar{u}),
\end{equation}
where $H_0^{w_{\theta}} :=\sum_{j\in \mathbb Z^*}   \omega^{w_{\theta}}_j|u_j|^2$ with $\omega^{w_{\theta}}_j\in\mathbb{R}$. Moreover for any $(u,\bar{u})\in {\cal H}^p(\mathbb Z^*,\mathbb C)$ the Hamiltonian vector of $S^{w_{\theta}}(u,\bar{u})$  holds
\begin{equation*}
\|X_{S^{w_{\theta}}}^{w_{\theta}}(u,\bar{u}) \|_{p}\leq 4 r^{p+1}C^{r-2}_{g^{w_{\theta}}}   c^{r-1} \frac{
N^{\alpha}}{{\gamma}} \|u\|_{p}\|u\|_{2}^{r-2}.
\end{equation*}

\end{Lemma}
\begin{proof} By the definition of Possion bracket $\{\ ,\ \}_{w_{\theta}}$, the solution  $S^{w_{\theta}}(u,\bar{u})$ of (\ref{47}) is still an $r$-degree
homogeneous polynomial and has the following form  
\begin{equation}\label{32-0}
S^{w_{\theta}}(u,\bar{u})=\sum_{i\in  M_{S^{w_{\theta}}_r},\ } \sum_{|l+k|=r,{\cal M}(l,k)=i } (S^{w_{\theta}})^i_{r,lk}   u^l\bar{u}^k
\end{equation}
with undetermined coefficients.
 Since $g^{w_{\theta}}(u,\bar{u})$ has $(\beta,\theta)$-type symmetric coefficients semi-bounded by $C_{g^{w_{\theta}}}$, by  Remark  \ref{trunction} and Remark \ref{B-h}, $ \Gamma^{\omega_{\theta}}_{(\theta,\gamma,\alpha,N)}\Gamma_{\leq 2}^N g^{w_{\theta}}(u,\bar{u})$ is  an  $r$-degree
$(\theta,\gamma,\alpha, N)$-normal form of $\Gamma_{\leq 2}^Ng^{w_{\theta}}(u,\bar{u})$ with   $(\beta,\theta)$-type symmetric  coefficients semi-bounded by $C_{g^{w_{\theta}}}$, and its coefficients have the following form
 \begin{equation}\label{48}
 (\Gamma^{\omega_{\theta}}_{(\theta,\gamma,\alpha,N)} \Gamma_{\leq 2}^N g^{w_{\theta}})^{i}_{r,lk}:=\left\{
\begin{array}{lll}
(\Gamma_{\leq 2}^N g^{w_{\theta}})^{i }_{r,lk},& \mbox{ if } \ \  |\langle \omega^{w_{\theta}}, \  I_{\theta}(l-k)\rangle|\leq\frac{\gamma
M_{l,k}}{N^{\alpha}},\\
\\
0,& \mbox{ if   }\ \ |\langle \omega^{w_{\theta}}, \  I_{\theta}(l-k)\rangle|>\frac{\gamma
M_{l,k}}{N^{\alpha}},\
\end{array}\right.
\end{equation}
where $ M_{l,k}$ is defined in Definition \ref{5.3}.
Take (\ref{32-0}) into  equation (\ref{47}) and get that for any $i\in M_{g^{w_{\theta}}}$ and any $l,k\in
\mathbb{N}^{\mathbb{Z}^*}$ with $|l+k|=r$ and ${\cal M}(l,k)=i$, 
\begin{equation}\label{32}
-{\bf i}\langle \omega^{w_{\theta}}, \  I_{\theta}(l-k)\rangle (S^{w_{\theta}})^{i}_{r,lk}+(\Gamma^{\omega_{\theta}}_{(\theta,\gamma,\alpha,N)} \Gamma_{\leq 2}^Ng^{w_{\theta}})^{i }_{r,lk}=(\Gamma_{\leq 2}^Ng^{w_{\theta}})^{i}_{r,lk},
\end{equation}
 which means that the coefficients of $S^{w_{\theta}}(u,\bar{u})$ has the following form 
\begin{equation}\label{42}
 (S^{w_{\theta}})^{i}_{r,lk}=\left\{\ \begin{array}{cl}
 -\frac{(\Gamma_{\leq 2}^N g^{w_{\theta}})^{i}_{r,lk}}{\ {\bf i}\langle \omega^{w_{\theta}}, \  I_{\theta}(l-k)\rangle\ },& \mbox{ when}\quad
|\langle \omega^{w_{\theta}}, \  I_{\theta}(l-k)\rangle |>\frac{\gamma
M_{l,k}}{N^{\alpha}}\ \\
\\
0,& \mbox{when}\quad  |\langle \omega^{w_{\theta}}, \  I_{\theta}(l-k)\rangle |\leq\frac{\gamma
M_{l,k}}{N^{\alpha}}
\end{array}\right.
\end{equation}
and satisfy that
\begin{eqnarray}\label{5.13}
\overline{(S^{w_{\theta}})^{i}_{r,lk}}= \overline{
-\frac{(\Gamma_{\leq 2}^N g^{w_{\theta}})^{i}_{r,lk}}{\ {\bf i}\langle \omega^{w_{\theta}}, \  I_{\theta}(l-k)\rangle\ }}
  = -\frac{(\Gamma_{\leq 2}^N g^{w_{\theta}})^{-i}_{r,kl}}{\ {\bf i}\langle \omega^{w_{\theta}}, \  I_{\theta}(k-l)\rangle\ }=(S^{w_{\theta}})^{-i}_{r,kl},
 \end{eqnarray}
 the second equality holds by  $\Gamma_{\leq 2}^N g^{w_{\theta}}(u,\bar{u})$ having  symmetric coefficients from Remark \ref{trunction} and $\omega^{w_{\theta}}_j\in\mathbb{R}$ ($j\in\mathbb{Z}^*$).

 \noindent
 The norm of Hamiltonian vector field $X_{S^{w_{\theta}}}^{w_{\theta}}$
\begin{eqnarray}\label{yy}
&& \big\| {X}^{w_{\theta}}_{S^{w_{\theta}}}(u,\bar{u}) \big\|_{p }
 = \sqrt{    \|   \nabla_{\bar{u} } S^{w_{\theta}}(u,\bar{u})\|_{p}^2+\|   \nabla_{ {u} } S^{w_{\theta}}(u,\bar{u}) \|_{p}^2     }\nonumber\\
&\leq&\big\|\big(   \sum_{j={\cal M}(l,k-e_j)-i\atop{i\in M_{S^{w_{\theta}}}\atop{|l+k-e_j|=r-1}} }   k_j  (S^{w_{\theta}})^{i }_{r,{l k}}  u^l\bar{u}^{k-e_j} \big)_{j\in\mathbb{Z}^*} \big\|_{p} +\big\| \big( \sum_{j=-{\cal M}(l-e_j,k)-i\atop{ |l-e_j+k|=r-1\atop{i\in M_{S^{w_{\theta}}}}}}  l_j (S^{w_{\theta}})^{i}_{r,{l k }}  u^{l-e_j}\bar{u}^k \big)_{j\in\mathbb{Z}^*} \big\|_{p}\nonumber
   \end{eqnarray}
equals to    the $\ell^2  $ norm of  the vector fields
$$ Q^{w_{\theta}}_1:=
\bigg(
   \sum_{i\in M_{S^{w_{\theta}}}}\sum_{j={\cal M}(l,k-e_j)-i\atop{|l+k-e_j|=r-1} }  \langle j\rangle^{p}  k_j \cdot (S^{w_{\theta}})^{i }_{r,{l k}}  u^l\bar{u}^{k-e_j} \bigg)_{j\in\mathbb{Z}^*}$$
and
$$ Q^{w_{\theta}}_2:=\bigg(
 \sum_{i\in M_{S^{w_{\theta}}}}\sum_{j=-{\cal M}(l-e_j,k)-i\atop{|l-e_j+k|=r-1}}  \langle j\rangle^{p}  l_j \cdot (S^{w_{\theta}})^{i}_{r,{l k }}  u^{l-e_j}\bar{u}^k \bigg)_{j\in\mathbb{Z}^*}.$$
 When $\theta =0$, for any $l,k\in \mathbb{N}^{\mathbb Z^*}$ with ${\cal M}(l,k)=i$ and $k_j\neq0$, by (\ref{a---b}), it holds 
\begin{eqnarray}\label{5.16}
 |j|^p\leq  2r^{p-1}\big(\sum_{t\in\mathbb{Z}^*} l_t |t|^p+ \sum_{t\in\mathbb{Z}^*, \ {t\neq j}} k_t|t|^p+ (k_j-1)|j|^p\big)
 \cdot \langle i\rangle^p.
\end{eqnarray}
For any $(l^0,k^0,i^0)\subset {\cal A}_{f_{r,lk}^i}$, by (\ref{a---b}) the following inequality holds
\begin{eqnarray}\label{AA}
|{\cal M}(l^0,k^0)-\frac{i^0}2|\leq 2r M_{l,k}\cdot \langle i_0 \rangle.
\end{eqnarray}
By (\ref{42}) and (\ref{AA}), the coefficients of $S^{w_0}$ satisfy that
\begin{eqnarray}\label{5.17}
 &&\frac{ \sum_{(l^0,k^0,i^0)\in {\cal A}_{(g^{w_0})^{i}_{r,lk}} } \big|(\Gamma^N_{\leq 2}g^{w_0})^{i(l^0,k^0,i^0)}_{r,lk} ({\cal M}(l^0,k^0)-\frac{i^0}2)\big|       }{|\langle \omega^{w_0}, \  I_{0}(l-k)\rangle|} \nonumber\\
  &\leq&  \frac{N^{\alpha}}{\gamma M_{l,k}}\sum_{(l^0,k^0,i^0)\in {\cal A}_{(g^{w_0})^{i}_{r,lk}}} |(\Gamma^N_{\leq 2} g^{w_0})^{i(l^0,k^0,i^0)}_{r,lk}|\cdot  2r M_{l,k}\cdot \langle i_0 \rangle \leq 2r \frac{N^{\alpha}C^{r-2}_{g^{w_0}}}{\gamma   \langle i\rangle^{\beta} }   .
\end{eqnarray}
From    (\ref{5.16}) and (\ref{5.17}),
 the coefficients of vector fields $Q^{w_0}_1(u,\bar{u})$ and $Q^{w_0}_2(u,\bar{u})$ fulfill
 \begin{eqnarray}\label{66}
 && \langle j\rangle^p\cdot |k_j (S^{w_0})^i_{r,lk}|,\quad  \langle j\rangle^p\cdot |l_j (S^{w_0})^i_{r,lk}|\nonumber\\
  &\leq& 2r^{p+1} \big(\sum_{t\in\mathbb{Z}^*} l_t  \langle t\rangle^p  + \sum_{t\in\mathbb{Z}^*\atop{t\neq j}} k_t \langle t\rangle^p + (k_j-1) \langle j\rangle^p\big)  \frac{N^{\alpha}C^{r-2}_{g^{w_0}}}{\gamma \langle i\rangle^{\beta-p}  }.
   \end{eqnarray}
 By (\ref{66}), using Corollary \ref{r1},  it holds that 
\begin{eqnarray}
 \big\| {X}^{w_0}_{S^{w_0}}(u,\bar{u}) \big\|_{p}\leq \|Q_1^{w_0}\|_{\ell^2} +\|Q_2^{w_0}\|_{\ell^2}
\leq 4r^{p+1} c^{r-1} \frac{N^{\alpha}C^{r-2}_{g^{w_0}}}{\gamma } \|u\|_p \|u\|_2^{r-2} .
   \end{eqnarray}
When $\theta=1$, in order to estimate the $\ell^2$-norm of $Q^{w_1}_1(u,\bar{u})$ and $Q^{w_1}_2(u,\bar{u})$, let us consider the coefficients of $Q_1^{w_1}(u,\bar{u})$ and $Q_2^{w_1}(u,\bar{u})$ firstly.  For any $i\in M_{S^{w_1}}$ and any $l,k\in \mathbb{N}^{{\mathbb Z}^*} $ satisfying
$|l+k|=r,\ {\cal M}(l,k)=i$ and $k_j\neq0 $ (or $l_j\neq0$), using (\ref{42}),
 the coefficients of $u^l\bar{u}^{k-e_j}$ in $Q^{w_1}_1$ are bounded by the following
  \begin{eqnarray}\label{coe}
    2\frac{r^{p+\frac12} C^{r-2}_{g^{w_1}_r} N^{\alpha} }{\gamma \langle i\rangle^{\beta-p+\frac12}} (\sum_{t} l_t|t|^{p-\frac12} +\sum_{t\neq j} (k-e_j)_t|t|^{p-\frac12} )   \prod_{t\in\mathbb{Z}^*} |t|^{\frac12(l_t+(k-e_j)_t)}
  \end{eqnarray}
and the coefficients of $u^{l-e_j}\bar{u}^k$ in $Q_2^{w_1}(u,\bar{u})$ are bounded by
  \begin{eqnarray}\label{coe-1} 
  2\frac{r^{p+\frac12} C^{r-2}_{g^{w_1}_r} N^{\alpha} }{\gamma  \langle i\rangle^{\beta-p+\frac12}} (\sum_{t\neq j} (l-e_j)_t|t|^{p-\frac12} +\sum_{t} k_t|t|^{p-\frac12}  )   \prod_{t\in\mathbb{Z}^*} |t|^{\frac12((l-e_j)_t+k_t)} .
  \end{eqnarray}
By Corollary \ref{r1} and (\ref{coe})-(\ref{coe-1}),   the following estimate is obtained 
$$ \|X^{w_1}_{S^{w_1}}(u,\bar{u})\|_{p}\leq \|Q_1^{w_1}\|_{\ell^2} +\|Q_2^{w_1}\|_{\ell^2} \leq 4 r^{p+\frac12}C^{r-2}_{g^{w_1}_r}\frac{N^{\alpha}}{\gamma} c^{r-1} \|u\|_p\|u\|_2^{r-2} .$$

\end{proof}

The following Lemma shows that the Possion bracket of  an $\tilde{r}$-degree homogeneous polynomial $f^{w_{\theta}}(u,\bar{u}) $ with $(\beta, \theta)$-type symmetric coefficients semi-bounded by $C_{f^{w_{\theta}}_{\tilde{r}}}>0$ and the solution  $S^{\omega_{\theta}}(u,\bar{u})$ to equation (\ref{47}) is still of $(\beta, \theta)$-type symmetric coefficients. Moreover, its coefficients satisfy some inequalities. 
\begin{Lemma}\label{sf}
  Let an $\tilde{r}$-degree homogeneous polynomial $f^{w_{\theta}}(u,\bar{u}) $ ($\theta\in\{0,1\}$)   have $(\beta,\theta)$-type symmetric coefficients semi-bounded by $C_{f^{w_{\theta}}_{\tilde{r}}}>0$.
      Then the possion bracket of $f^{w_{\theta}}(u,\bar{u})$ and the solution $S^{w_{\theta}}(u,\bar{u}) $ to equation (\ref{47}) under the symplectic form $w_{\theta}$ is an $(\tilde{r}+r-2)$-degree  homogeneous polynomial  with $(\beta,\theta)$-type symmetric coefficients  and it holds that
   \begin{itemize}
     \item   when $\theta=0$, for any $l',\ k',\ i$ fulfilling $|l'+k'|=\tilde{r}+r-2$ and ${\cal M}(l',k')=i$,  it holds that  
   \begin{eqnarray*}
   \sum_{(l'^0,k'^0,i^0)\in {\cal A}_{(f^{w_0}_{(1,S^{w_0})})^i_{\tilde{r}+r-2,l'k'} }}  |(f^{w_0}_{(1,S^{w_0})})^{i(l'^0,k'^0,i^0)}_{\tilde{r}+r-2,l'k'} |\cdot \max\{\langle i^0\rangle ,  \langle i^0-2i\rangle\}\nonumber\\
  \leq 2^{\beta+2}r^2c(\tilde{r} +1)   \frac{C^{\tilde{r}-2}_{f^{w_0}_{\tilde{r}}}C^{r-2}_{g^{w_0}_{r}}(2N+1)^{r-2}N^{\alpha+1}}{ \gamma\langle i\rangle^{\beta}   }.
   \end{eqnarray*}
      \item when $\theta=1$ the following inequality holds  
     true
\begin{eqnarray*}
|({f}^{w_1}_{(1,S^{w_1})})^i_{\tilde{r}+r-2,l'k'}|\leq 2^{\beta+1}cr(\tilde{r} +1) \frac{C^{\tilde{r}-2}_{f^{w_1}_{\tilde{r}}} C^{r-2}_{g^{w_1}_{r}}(2N+1)^{r-2}N^{\alpha}}{\gamma \langle i\rangle^{\beta}} \prod_{t\in\mathbb{Z}^*} |t|^{\frac{l'_t+k'_t}{2}}.\end{eqnarray*}
   \end{itemize}

   \end{Lemma}
  \begin{Remark}\label{jjd}
  Under the same assumptions of Lemma \ref{sf},
     for any integer  $\nu\geq 1$,   $f^{w_{\theta}}_{(\nu,S^{w_{\theta}})}$  is an  $\big(\tilde{r}+\nu(r-2)\big)$-degree  homogeneous polynomial with $(\beta,\theta)$-type symmetric coefficients.
   {\begin{itemize}
    \item When $\theta=0$,
    the following inequality holds 
    \begin{eqnarray*}
     &&\sum_{(l^0,k^0,i^0)\in {\cal A}_{ (f^{w_{0}}_{(\nu,S^{w_{0}})})^{i}_{\tilde{r}+(r-2)\nu,lk} }} |(f^{w_{0}}_{(\nu,S^{w_{0}})})^{i(l^0,k^0,i^0)}_{\tilde{r}+(r-2)\nu,lk}|\cdot \max\{\langle i^0\rangle,\ \langle i^0-2i\rangle \}\nonumber\\
     &\leq& C^{\tilde{r}-2}_{f^{w_{0}}} \big(2^{\beta+2}r^2c (2N+1)^{(r-2) }\frac{N^{\alpha+1}}{\gamma}   C^{r-2}_{g^{w_{0}}}\big)^{\nu}\frac{ \prod_{n=0}^{\nu-1}  \big(\tilde{r}+n(r-2)+1\big)}{\langle i\rangle^{\beta}\nu!}  ;\nonumber\\
     \end{eqnarray*}
    \item When $\theta=1$,
     it holds that
   {\small $$\!\!|({f}^{w_{1}}_{(\nu,S^{w_{1}})})^{i}_{\tilde{r}+(r-2)\nu,lk}|\! \leq\! \frac{C^{\tilde{r}-2}_{f^{w_{1}}_{\tilde r}}}{\ \nu!\langle i\rangle^{\beta} \ }\big(2^{\beta+1}r c (2N)^{(r-2)}\frac{N^{\alpha}}{\gamma}   C^{r-2}_{g^{w_{1}}}\big)^{\nu}\! \prod_{n=0}^{\nu-1}\!  \big(\tilde{r}+n(r-2)+1\big)\!\prod_{t\in\mathbb{Z}^*}\! |t|^{\frac{l'_t+k'_t}{2}} .$$
 } \end{itemize}
  }\end{Remark}

\noindent 
Before  proving Lemma \ref{sf},  I denote a set of indexes and give a Lemma to count the number of this set. This Lemma is used to  prove Lemma \ref{sf}. 

\noindent
For any $(l',k')\in\mathbb{N}^{\mathbb{Z}^*  }\times\mathbb{N}^{\mathbb{Z}^* }$   and any $i'\in \mathbb{Z}$,
 let
 \begin{eqnarray*}
 && \Omega(l',k',i' ):= \left\{ \ \big((l,k,i_1),\ (L,K,i_2),\ j  \big)\ \left| \begin{array}{ll}
 & l,k,L,K\in \mathbb{N}^{\mathbb{Z}^*},\ i_1,\ i_2\in\mathbb{Z},\ j\in\mathbb{Z}^*;\\
 & \mbox{satisfying } {\bf A,  \ B,\  D1 } \ \mbox{or}\  {\bf A,\
 B,\ D2}
 \end{array}\right. \right\},
 \end{eqnarray*}
 where
{\small\begin{description}
  \item[A:] $|l+k|= \tilde{r},\ |L+K|= {r},\ |\Gamma_{>N}(L+K)|\leq 2;$
  \item[B:]  ${\cal M}(l,k)=i_1,\ {\cal M}(L,K)=i_2, \ i'=i_1+i_2;$
  \item[D1:] $(l-e_j)+L=l', \ k+(K-e_j)=k',\ \mbox{with}\ l_j>0\  \mbox{and}\ K_j>0;$
  \item[D2:] $l+(L-e_j)=l', \ (k-e_j)+K=k',\ \mbox{with}\ L_j>0\ \mbox{and}\ k_j>0.$
  \end{description}}
From the definition of set $\Omega(l',k',i' )$, if element  $((l,k,i_1),\ (L,K,i_2), j)\in \Omega(l',k',i')$, then $((k,l,-i_1),\ (K,L,-i_2), j)\in \Omega(k',l',-i')$.

\begin{Lemma}\label{jrm2}
  Fix $\beta\geq2$. For any given $ l',k'\in \mathbb{N}^{\mathbb{Z}^*}$ with   $|l'+k'|=r+\tilde{r}-2$ and ${\cal M}(l',k')=i'$,  it holds
   \begin{eqnarray*}
&& \sum_{((l,k,i_1), (L,K,i_2), j)\in\Omega(l',k',i')}\frac{K_jl_j}{\langle i_2-i'\rangle^{\beta} \cdot \langle i_2\rangle^{\beta} },\
 \sum_{((l,k,i_1), (L,K,i_2), j)\in\Omega(l',k',i')}\frac{k_jL_j}{\langle i_2-i'\rangle^{\beta} \cdot \langle i_2\rangle^{\beta}}\\
 &\leq &\frac{2^{\beta+1}}{\langle i'\rangle^{\beta}}cr(\tilde{r}+1)(2N+1)^{r-2}  .
\end{eqnarray*}
\end{Lemma}
\begin{proof}
 Consider the non-zero components of vectors $k'$ and $l'$. For example, $e_j$ has only one non-zero component with index $j$, being 1; Taking multiplicity into account, 
 regard that $ke_j$($k$ is a positive integer) has $k$ non-zero components  whose values are 1  and  their indexes
 are $j$.
So $(l',k')$ with $|l'+k'|=r+\tilde{r}-2$ have $r+\tilde{r}-2$ non-zero components whose values are $1$.

Denote   $$\Omega_{j,i_2}(l',k',i' ):= \left\{\big((l,k,i_1),(L,K,\tilde{i}),t  \big)\in \Omega(l',k',i' ) \  | \  t=j\ \mbox{and}\  \tilde{i}=i_2
   \right\}. $$ It follows
$$\Omega(l',k',i')
= \bigcup_{j\in\mathbb{Z}^* , } \bigcup_{i_2\in\mathbb{Z}} \Omega_{j,i_2}(l',k',i').$$

   \noindent
The element in $\Omega_{j,i_2}(l',k',i')$ is unique determined, if $(l-e_j,k )$ is fixed. The estimate of  $\tilde{\sharp}\Omega_{j,i_2}(l',k',i')$ is obtained as follows.

 \noindent
 In the case $|j|\leq N$, since $|\Gamma_{>N}({L}+{K})|\leq 2$,  there are  at least  $r-3$  non-zero components  of $(l',k')$ coming from
  $(L, K)$ with the indexes being bounded by $N$ and the choices of that is smaller than  $(2N+1)^{r-3}$.  As for the remaining three components of $(L,K)$ whose values are 1,
one  of their positions is $j$ with $|j|\leq N$;
One position among the other two can be selected from  the rest non-zero components of $(l',k')$ and the choices is $\tilde{r}+1$;  The last one may be determined by the
fact that ${\cal M}(L,K)=i_2$. It holds  
\begin{eqnarray}\label{first-1}
&&\sum_{((l,k,i_1),(L,K,i_2),j)\in \bigcup_{|j|\leq N}\bigcup_{i_2\in\mathbb{Z}}\Omega_{j,i_2}(l',k',i'),}
 \frac{ K_jl_j}{ \langle i_1\rangle^{\beta}\cdot \langle i_2\rangle^{\beta}}  \nonumber\\
&\leq&\!\! \frac{r}{\langle i'\rangle^{\beta} }\sum_{\ |j|\leq N,\ }  \sum_{\Omega_{j,i_2}(l',k',i'),\ } \sum_{i_2\in \mathbb{Z}}\frac{l_j \langle i'\rangle^{\beta}}{\langle i_1\rangle^{\beta}\cdot\langle i_2\rangle^{\beta}}\nonumber\\
 &\leq&\!\! \frac{r}{ \langle i' \rangle^{\beta}}\sum_{|j|\leq N,\ } \! \sum_{\Omega_{j,i_2}(l',k',i'),\ }\! \sum_{i_2\in \mathbb{Z}}\frac{l_j2^{\beta-1}( \langle i_1\rangle^{\beta}  + \langle i_2\rangle^{\beta} ) }{ \langle i_1\rangle^{\beta} \cdot \langle i_2\rangle^{\beta} }\!\leq\!\frac{rc2^{\beta}}{\langle i' \rangle^{\beta}}  (\tilde{r}+1)(2N+1)^{r-2} .
\end{eqnarray}
 In the case $|j|>N$,  there are at least $r-2$ value-1 components of $(l',k')$ coming from $(L,K)$ whose
 indexes are
bounded by $N$, and  there are at most $(2N+1)^{r-2}$ choices; One position of the last two value-1
 components of $(L,K)$ is chosen  from the rest $\tilde{r}$ non-zero components of $(l',k')$ and the choices is $\tilde{r}$;
 The position of the last  component of $(L,K)$ is determined by the  momentum  of $(L,K)$ being $i_2$.
It holds 
\begin{eqnarray}\label{first-2}
&&\sum_{((l,k,i_1),(L,K,i_2),j)\in\bigcup_{|j|> N}\bigcup_{i_2\in\mathbb{Z}}\Omega_{j,i_2}(l',k',i'),}
\frac{ K_jl_j}{  \langle i_1\rangle^{\beta} \cdot  \langle i_2\rangle^{\beta}}  \nonumber\\
 &\leq& \frac{r}{ \langle i' \rangle^{\beta} }\sum_{|j|> N,\ } \sum_{i_2\in \mathbb{Z},\ } \sum_{\Omega_{j,i_2}(l',k',i'),\ } \frac{l_j  \langle i' \rangle^{\beta} }{\langle i_2\rangle^{\beta} \cdot \langle i_1\rangle^{\beta} }\nonumber\\
 &\leq& \frac{r}{\langle i' \rangle^{\beta} }\sum_{|j|> N,\ } \sum_{i_2\in \mathbb{Z},\ } \sum_{\Omega_{j,i_2}(l',k',i'),\ }
  \frac{l_j2^{\beta-1}(  \langle i_1\rangle^{\beta} + \langle i_2\rangle^{\beta} ) }{ \langle i_1\rangle^{\beta} \cdot \langle i_2\rangle^{\beta} }\leq \frac{rc2^{\beta}}{ \langle i' \rangle^{\beta}}  \tilde{r}(2N+1)^{r-2} .
\end{eqnarray}
The result is obtained from (\ref{first-1}) and (\ref{first-2}).  
 \end{proof}

In the following, the proof of the Lemma \ref{sf} is given. 
\begin{proof}
 By the definition of $\{\cdot,\cdot\}_{w_{\theta}}$, the following equation holds 
 \begin{eqnarray*}
f^{w_{\theta}}_{(1,S^{w_{\theta}})}=\{f^{w_{\theta}},S^{w_{\theta}}\}_{w_{\theta}}= 
\sum_{|l'+k'|=r+\tilde{r}-2\atop{{\cal M}(l',k')=i'\in M_{f^{w_{\theta}}_{(1,S^{w_{\theta}})}}}} (f^{w_{\theta}}_{(1,S^{w_{\theta}})})_{r+\tilde{r}-2,l'k'}^{i'} u^{l'}\bar{u}^{k'}
\end{eqnarray*}
with
 $
 M_{f^{w_{\theta}}_{(1,S^{w_{\theta}})}}:=\{i=i_1+i_2 \ | \ i_1\in M_{f^{w_{\theta}}}\subset\mathbb{Z},\ i_2\in M_{S^{w_{\theta}}}\subset\mathbb{Z}\}
 $
and
\begin{eqnarray*}
(f^{w_{\theta}}_{(1,S^{w_{\theta}})})^{i'}_{\tilde{r}+r-2,l',k'}:
&=&\sum_{j\in\mathbb{Z}^*\atop{i'=i_1+i_2}}{\bf i}\mbox{sgn}^{\theta}(j) \sum_{(l-e_j)+L=l'\atop{k+(K-e_j)=k'\atop{{\cal M}(l,k)=i_1,{\cal M}(L,K)=i_2}}}  l_jK_j (f^{w_{\theta}})^{i_1}_{\tilde{r}, l k}   (S^{w_{\theta}})^{i_2}_{r,L K} \nonumber\\ 
&&-\sum_{j\in\mathbb{Z}^*\atop{i'=i_1+i_2}}{\bf i}\mbox{sgn}^{\theta}(j)\sum_{ l+(L-e_j)=l'\atop{(k-e_j)+K=k'\atop{{\cal M}(l,k)=i_1,{\cal M}(L,K)=i_2}}}  L_jk_j (f^{w_{\theta}})^{i_1}_{\tilde{r}, l k}    (S^{w_{\theta}})^{i_2}_{r,LK}   .
\end{eqnarray*}
 
\noindent
  In the case $\theta =0$, I will
   give the  exact definition of ${\cal A}_{(f^{w_0}_{(1,S^{w_0})})^{i'}_{\tilde{r}+r-\!2,l' k'}}\!$ and $(\!f^{w_0}_{(1,S^{w_0})}\!)_{\tilde{r}+r-2,l'k' }^{i'(l'^0,k'^0,i'^0)}$ and prove that 
   the coefficients $(\!f^{w_0}_{(1,S^{w_0})})^{i'}_{\tilde{r}+r-2,l'k'}$ can be rewritten as the following form
   $$(f^{w_0}_{(1,S^{w_0}��)})^{i'}_{\tilde{r}+r-2,l'k'}:= \sum_{(l'^0,k'^0,i'^0)\in {\cal A}_{(f^{w_0}_{(1,S^{w_0})})^{i'}_{\tilde{r}+r-2,l' k'}}} (f^{w_0}_{(1,S^{w_0})})_{\tilde{r}+r-2,l'k' }^{i'(l'^0,k'^0,i'^0)}\big( {\cal M}(l'^0,k'^0)-\frac{i'^0}2 \big)$$
   and satisfy
   $$ \overline{(f^{w_0}_{(1,S^{w_0})})_{\tilde{r}+r-2,l'k' }^{i'(l'^0,k'^0,i'^0)}}=(f^{w_0}_{(1,S^{w_0})})_{\tilde{r}+r-2,k'l' }^{-i'(k'-k'^0,l'-l'^0,i'^0-2i')}. $$
 In order to describe the set ${\cal A}_{(f^{w_0}_{(1,S^{w_0}��)})^{i'}_{\tilde{r}+r-2,l'k'}}$ clearly,
 for any fixed $  \big((l,k,i_1),\ (L,K,i_2),\ j  \big)\in \Omega^{w_{0}}(l',k',i' )$,  define a mapping $D$ on set $ {\cal A}_{(f^{w_0})_{\tilde{r},lk}^{i_1}}$, for any $(l^0,k^0,i_1^0)\in {\cal A}_{(f^{w_0})_{\tilde{r},lk}^{i_1}}$,
{\footnotesize\begin{eqnarray}\label{5.51}
D(l^0,k^0,i_1^0):=\left\{
\begin{array}{lll}
(l^0,k^0,i^0_1) & \begin{array}{ll}&\mbox{when} \ (l',k')= \big((l-e_j)+L,k+(K-e_j)\big) \\
 &\mbox{and}\ l_j^0=0.
 \end{array}\\
(l^0-e_j+L,k^0+K-e_j,i^0_1+2i_2) &  \begin{array}{ll}&\mbox{when}\  (l',k')=\big((l-e_j)+L,k+(K-e_j)\big) \\ &\mbox{and}\ l_j\geq l_j^0>0.
\end{array}\\
(l^0 ,k^0,i^0_1 ) & \begin{array}{ll}&\mbox{when}\  (l',k')=\big(l+(L-e_j),(k-e_j)+K\big) \\
 &\mbox{and}\ k_j>k_j^0\geq 0.
 \end{array}\\
(l^0-e_j+L ,k^0+K-e_j ,i^0_1+2i_2)  & \begin{array}{ll} &\mbox{when}\ (l',k')= \big(l+(L-e_j),(k-e_j)+K\big)
\\ &\mbox{and}\ k_j=k_j^0>0.
\end{array}
\\
\end{array}
 \right.
\end{eqnarray}
}

\noindent Base on the set $\Omega(l',k',i')$ and the map $D$,  denote  
\begin{eqnarray*}
&&{\cal A}_{(f^{w_0}_{(1,S^{w_0})})_{\tilde{r}+r-2,l'k'}^{i'}}:=\bigcup_{ ((l,k,i_1),\ (L,K,i_2),\ j )\in \Omega(l',k',i') }{D{\cal A}}_{(f^{w_0})_{\tilde{r},lk}^{i_1} }
\end{eqnarray*}
and
\begin{eqnarray*}
{\cal A}_{(f^{w_0}_{(1,S^{w_0})})^{-i'}_{\tilde{r}+r-2,k'l'}}:=\bigcup_{ ((k,l,-i_1),\ (K,L,-i_2),\ j)\in \Omega(k',l',-i') } D{\cal A}_{(f^{w_0})_{\tilde{r},kl}^{-i_1}},
\end{eqnarray*}
where $$ D {\cal A}_{(f^{w_0})_{\tilde{r},lk}^{i_1}}:=\{\ D(l^0,k^0,i_1^0) \ |\ (l^0,k^0,i_1^0)\in  {\cal A}_{(f^{w_0})_{\tilde{r},lk}^{i_1}} \}.$$ 
 It is easy to check that $D$ is not an  inverse mapping from $  {\cal A}_{(f^{w_0})_{\tilde{r},lk}^{i_1}}$ to
 $ D {\cal A}_{(f^{w_0})_{\tilde{r},lk}^{i_1}}$.
Denote  
 \begin{small}\begin{eqnarray}\label{jss}
(f^{w_0}_{(1,S^{w_0} )})^{i'(l'^0,k'^0,i'^0)}_{\tilde{r}+r-2,l'k' }\!:=\!\sum_{(l^0,k^0,i^0)\in D^{-1}(l'^0,k'^0,i'^0)    }{\bf i}(l_jK_j-L_jk_j)  ( f^{w_0})_{\tilde{r},l k}^{i_1 (l^0,k^0,i^0 )}   (S^{w_0})^{i_2}_{r,LK},
   \end{eqnarray}
   \end{small}
   where  $D^{-1}(l'^0,k'^0,i'^0):=\{   (l^0,k^0,i^0)\in {\cal A}_{(f^{w_0})^{i^0}_{\tilde{r},lk}} \ |\ D(l^0,k^0,i^0)=(l'^0,k'^0,i'^0) \}.$  
  For any $(l'^0,k'^0,i'^0)\in {\cal A}_{(f^{w_0}_{(1,S^{w_0})})^{i'}_{\tilde{r}+r-2,l'k'}}$,  it is easy to verify that   $ (k'-k'^0, l'-l'^0,i'^0-i')\in{\cal A}_{(f^{w_0}_{(1,S^{w_0})})^{-i'}_{\tilde{r}+r-2,k'l'}}$.
Moreover, by (\ref{jss}) and the facts that $f^{w_0}$ having $(\beta,0)$-type symmetric coefficients  and $S^{w_0}$ having  symmetric coefficients, it holds
  \begin{eqnarray*}
  &&\overline{(f^{w_0}_{(1,S^{w_0}
  )})^{i'(l'^0,k'^0,i'^0)}_{\tilde{r}+r-2,l'k'}} \nonumber\\
  &=&\overline{{\bf i}
 \sum_{ (l^0,k^0,i^0)\in D^{-1}(l'^0,k'^0,i'^0) } ( l_jK_j-k_jL_j)   (f^{w_0})_{\tilde{r},l k}^{i_1 (l^0,k^0,i^0)}   (S^{w_0})^{i_2}_{ {r},LK} }\nonumber\\
 &=&{\bf i}  \sum_{ (k-k^0,l-l^0,i^0-2i)\in D^{-1}(k'-k'^0,l'-l'^0, i'^0-2i')      } (L_jk_j- l_jK_j)   (f^{w_0})_{\tilde{r},kl}^{-i_1 ( k-k^0, l-l^0,i^0-2i  )} (S^{w_0})^{-i_2}_{{r},KL}\nonumber\\
 &=&(f^{w_0}_{(1,S^{w_0})})^{-i'(k'-k'^0,l'-l'^0,i'^0-2i')}_{\tilde{r}+r-2,
 k'l'},  \end{eqnarray*}
the last second equation is holding by the definition of $D$ in (\ref{5.51}) and
 $$
 {\cal M}(l'^0,k'^0)-\frac{i'^0}2={\cal M}(k'-k'^0,l'-l'^0)-(\frac{i'^0}2-i').
 $$ So $ f^{w_0}_{(1,S^{w_0}
  )}$ has $(\beta,0)$-type symmetric coefficients semi-bounded by $ C_{g^{w_0}_{r}}$.
   \noindent  By (\ref{jss}), (\ref{5.17}) in Lemma \ref{lem2}, it holds 
 \begin{small}
 \begin{eqnarray}\label{lki}
  && \sum_{(l'^0,k'^0,i'^0)\in{\cal A}_{(f^{w_0}_{(1,S^{w_0})})^{i'}_{\tilde{r}+r-2,l'k'} }}
  |(f^{w_0}_{(1,S^{w_0} )})^{ i'(l'^0,k'^0,i'^0)}_{\tilde{r}+r-2,l'k'} |\cdot \max\{\langle i'^0\rangle, \ \langle i'^0-2i'\rangle\}\nonumber\\
 &\leq&    \sum_{(l'^0,k'^0,i'^0)\in{\cal A}_{(f^{w_0}_{(1,S^{w_0} )})^{i'}_{\tilde{r}+r-2,l'k'}} } |\sum_{ (l^0,k^0,i^0)\in D^{-1}(l'^0,k'^0,i'^0)   }(k_jL_j-K_jl_j) (f^{w_0})_{\tilde{r},l k}^{ i_1( l^0,k^0,i^0) }     (S^{w_0})^{i_2}_{{r},LK}   | \nonumber\\
 && \cdot \max\{\langle i'^0\rangle, \ \langle i'^0-2i'\rangle \} \nonumber\\
 &\leq&     \frac{2r C^{r-2}_{g^{w_0}_{r}} N^{\alpha}}{\gamma} \sum_{((l,k,i_1), (L,K,i_2), j  )\in\Omega(l',k',i') }  \frac{\big|k_jL_j-K_j l_j\big|}{\langle i_2\rangle^{\beta} } \nonumber\\
 && \cdot \sum_{(l^0,k^0,i_1^0)\in  {\cal A}_{(f^{w_0})^{i_1}_{\tilde{r},l k}}} \big|(f^{w_0})_{\tilde{r},l k}^{i_1 (l^0,k^0,i_1^0)}\big| \cdot \big(\max\{ \langle i_1^0\rangle , \ \langle i_1^0-2i_1\rangle \} +N\big)\nonumber\\
 &\leq&  2r   \frac{N^{\alpha+1}C^{r-2}_{g^{w_0}_{r}} C^{\tilde{r}-2}_{f^{w_0}_{\tilde{r}}} }{\gamma  }    \sum_{((l,k,i_1), (L,K,i_2), j )\in\Omega(l',k',i' ),\ } \sum_{i_1,i_2\in\mathbb{Z}}\frac{|k_jL_j-K_jl_j|}{   \langle i_2\rangle^{\beta}\cdot \langle i_1\rangle^{\beta} } .
 \end{eqnarray}
\end{small}
  By (\ref{lki}) and Lemma \ref{jrm2}, it follows that
   \begin{eqnarray}
   &&\sum_{(l'^0,k'^0,i'^0)\in{\cal A}_{(f^{w_0}_{(1,S^{w_0})})^{i'}_{\tilde{r}+r-2,l'k'}} }| (f_{(1,S^{w_0})})^{ i'(l'^0,k'^0,i'^0)}_{\tilde{r}+r-2,l'k'} |\cdot \max\{\langle i'^0\rangle, \ \langle i'^0-2i'\rangle \}\nonumber\\
   &\leq&  2^{\beta+2}r^2 c(\tilde{r}+1)  N^{\alpha+1}C^{r-2}_{g^{w_0}_{r}} C^{\tilde{r}-2}_{f^{w_0}_{\tilde{r}}} \frac{ (2N+1)^{r-2}}{\gamma  \langle i'\rangle^{\beta}   }  .
    \end{eqnarray}
When $\theta =1$, the coefficients of $f^{w_1}_{(1,S^{w_1})}$ have the following form
\begin{eqnarray}
(f^{w_1}_{(1,S^{w_1})})^{i}_{\tilde{r}+r-2,l'k' }:=(\tilde{f}^{w_1}_{(1,S^{w_1})})^{i}_{\tilde{r}+r-2,l'k' }\prod_{t\in\mathbb{Z}^*} |t|^{\frac{l'_t+k'_t}{2}},
\end{eqnarray}
where
\begin{eqnarray}\label{jss-1}
(\tilde{f}^{w_1}_{(1,S^{w_1})})^{i}_{\tilde{r}+r-2,l'k' }: = \sum_{((l,k,i_1), (L,K,i_2), j)\atop{\in\Omega(l',k',i )}}
 {\bf i} \mbox{sgn}(j)  ( l_jK_j\!-\!L_jk_j)  ( \tilde{f}^{w_1})_{\tilde{r},l k}^{i_1}   \frac{|j|(\tilde{g}^{w_1})^{i_2}_{r,LK}}{{\bf i}\langle \omega^{w_1}, I_1( l-k)\rangle}.
   \end{eqnarray}
\vspace{4pt}
\noindent
Since $f^{w_1}(u,\bar{u})$ and $ g^{w_1}(u,\bar{u})$ have $(\beta,1)$-type symmetric coefficients ($g^{w_1}(u,\bar{u})$ is given in equation (\ref{47})), from (\ref{jss-1}) the coefficients of $ {f}^{w_1}_{(1,S^{w_1})}$ satisfy that 
  \begin{eqnarray*}
  &&\overline{(\tilde{f}^{w_1}_{(1,S^{w_1})})^{i}_{\tilde{r}+r-2,l'k'}}\nonumber\\
  &=&\overline{
  \sum_{((l,k,i_1), (L,K,i_2), j)\in\Omega(l',k',i )}{\bf i} \mbox{sgn}(j)  ( l_jK_j-L_jk_j)  ( \tilde{f}^{w_1})_{\tilde{r},l k}^{i_1}   \frac{|j|(\tilde{g}^{w_1})^{i_2}_{r,LK}}{{\bf i}\langle \omega^{w_1}, I_1( l-k)\rangle} }\nonumber\\
 &=&  \sum_{(k,l,-i_1),\ (K,L,-i_2), j)\in\Omega(k',l',-i )}{\bf i}\mbox{sgn}(j)  ( L_jk_j- l_jK_j)  ( \tilde{f}^{w_1})_{\tilde{r},k l}^{-i_1}   \frac{|j|(\tilde{g}^{w_1})^{-i_2}_{r,KL}}{{\bf i}\langle \omega^{w_1}, I_1( k-l)\rangle}\nonumber\\
 &=&(\tilde{f}^{w_1}_{(1,S^{w_1})})^{-i}_{\tilde{r}+r-2,k'l'}.
  \end{eqnarray*}
From (\ref{42}),  for any $(l,k)$ with nonzero $(S^{w_1})_{r,lk}^i$, $$  |\langle \omega^{w_1}, I_1( l-k)\rangle|>\frac{\gamma M_{l,k}}{N^{\alpha}} ,$$
which implies that  
     \begin{eqnarray*}
  &&| (\tilde{f}^{w_1}_{(1,S^{w_1})})^{i}_{\tilde{r}+r-2,l'k'} |\nonumber\\
  &\leq& \sum_{((l,k,i_1), (L,K,i_2), j)\in\Omega(l',k',i )}
   \big| (l_jK_j-L_jk_j)\cdot (\tilde{f}^{w_1})_{\tilde{r},l k}^{i_1}   \frac{j(\tilde{g}^{w_1})^{i_2}_{r,LK}}{{\bf i}\langle \omega^{w_1}, I_1( l-k)\rangle}\big|\nonumber\\
   &\leq& \frac{N^{\alpha}C^{r-2}_{g_r^{w_1}}C^{\tilde{r}-2}_{f_{\tilde{r}}^{w_1}}}{\gamma}\sum_{((l,k,i_1), (L,K,i_2), j)\in\Omega(l',k',i )}
     \frac{|l_jK_j-L_jk_j| }{ \langle i_1\rangle^{\beta}  \langle i_2\rangle^{\beta} }\nonumber\\
   &\leq&  (\tilde{r}+1)cr(2N )^{r-2}  \frac{N^{\alpha} 2^{\beta+1} C^{r-2}_{f^{w_1}_r} C^{\tilde{r}-2}_{g^{w_1}_{\tilde{r}}} }{\gamma \langle i\rangle^{\beta}} ,
  \end{eqnarray*}
the last inequality holds by  Lemma \ref{jrm2}.
 \end{proof}

The proof of Theorem \ref{Th2} is a purely technical matter and is relegated to Appendix.

\section{Proof of  Theorem \ref{T22}}

For any given integer $r_*\geq0$, using Theorem \ref{Th2}, there exists a transformation ${\cal T}^{(r_*)}_{w_{\theta}}$ changing the system (\ref{uj}) into
\begin{eqnarray}\label{5-1}
\left\{\begin{array}{ll}
\dot{\tilde{u}}_j=-{\bf i} \mbox{sgn}^{\theta}(j)\cdot \partial_{\bar{\tilde{u}}_j} H^{(r_*,w_{\theta})}(\tilde{u},\bar{\tilde{u}}),\\
\\
\dot{\overline{\tilde{u}}}_j=\ \ {\bf i} \mbox{sgn}^{\theta}(j) \cdot\partial_{\tilde{u}_j} H^{(r_*,w_{\theta})}(\tilde{u},\bar{\tilde{u}}),
\end{array}  \right.
\end{eqnarray}
with Hamiltonian
\begin{eqnarray}
H^{(r_*,w_{\theta})} (\tilde{u},\bar{\tilde{u}})=H_0^{w_{\theta}}+Z^{(r_*,w_{\theta})}(\tilde{u},\bar{\tilde{u}})
+{\cal R}^{N(r_*,w_{\theta})}(\tilde{u},\bar{\tilde{u}})+{\cal R}^{T(r_*,w_{\theta})}(\tilde{u},\bar{\tilde{u}}).
\end{eqnarray}
The solution $(\tilde{u},\bar{\tilde{u}})$ to (\ref{5-1}) satisfies
\begin{eqnarray}\label{a-4}
&&\frac{d}{dt}\|\tilde{u}(t)\|_p^2=\{ \|\tilde{u}\|_p^2 , H^{(r_*,w_{\theta})}(\tilde{u},\bar{\tilde{u}})\}_{w_{\theta}} \nonumber\\
&=&\{ \|\tilde{u}\|_p^2 , H_0^{w_{\theta}}+Z^{(r_*,w_{\theta})}(\tilde{u},\bar{\tilde{u}})
+{\cal R}^{N(r_*,w_{\theta})}(\tilde{u},\bar{\tilde{u}})+{\cal R}^{T(r_*,w_{\theta})}(\tilde{u},\bar{\tilde{u}})\}_{w_{\theta}}.
\end{eqnarray}
It is easy to get that
\begin{equation}\label{a-1}
\{ \|\tilde{u}\|_p^2 , H_0^{w_{\theta}}(\tilde{u},\bar{\tilde{u}})\}_{w_{\theta}}=0.
\end{equation}
Using Theorem \ref{Th2}, Proposition \ref{2.1} and Corollary \ref{4.3}, when  $N$ satisfies (\ref{N-c}), it holds  that
\begin{equation}\label{a-2}
\sup_{(\tilde{u},\bar{\tilde{u}})\in B_p(R/4)}
|\{ \|\tilde{u}\|_p^2, \ {\cal R}^{N(r_*,w_{\theta})}(\tilde{u},\bar{\tilde{u}})+{\cal R}^{T(r_*,w_{\theta})}(\tilde{u},\bar{\tilde{u}})\}_{w_{\theta}}|
 \leq  \frac{1}{2}C(\theta,p,r_*) R^{r_*+1},
\end{equation} the inequality is holding by the fact that for any $(\tilde{u},\bar{\tilde{u}})\in B_p(R)$
\begin{equation}\label{g-m}
\|\Gamma_{>N}\tilde{u}\|_3\leq  \frac{\|\Gamma_{>N}\tilde{u}\|_p}{N^{p-3}}\quad \mbox{and} \quad
\sum_{|i|>N}\frac{1}{\langle i\rangle^{\beta}}\leq \frac{1}{N^p}  (\sum_{|i|>N}\frac{1}{\langle i\rangle^{2}}), \ \mbox{as}\ \beta>2p+4.
\end{equation}
By Lemma \ref{6.2} and (\ref{g-m}), when $N$ satisfies  (\ref{N-c}), it follows that 
\begin{equation}\label{a-3}
 \sup_{(\tilde{u},\bar{\tilde{u}})\in B_p (R/4)}
|\{ \|\tilde{u}\|_p^2, \ Z^{(r_*,w_{\theta})}(\tilde{u},\bar{\tilde{u}})\}_{w_{\theta}}|
 \leq  \frac{1}{2}C(\theta, p,r_*) R^{r_*+1}.
\end{equation}
Suppose that the initial value to (\ref{uj}) satisfies $(u(0),\bar{u}(0))\in B_p(R/6)$. If $R$ is small enough, the initial value $(u(0),\bar{u}(0))\in B_p(R/6)$ is transformed into \begin{equation}\label{67}(\tilde{u}(0),\bar{\tilde{u}}(0))\in B_p(R/4).
\end{equation}
Together with (\ref{a-4})-(\ref{a-2}) and (\ref{a-3})-(\ref{67}), the following inequality holds true
\begin{eqnarray}
 \big|\|\tilde{u}(t)\|_p^2- \|\tilde{u}(0)\|_p^2 |\leq |\int_{0}^T \frac{d \|\tilde{u}(\tau)\|_p^2}{d\tau} d\tau|\leq C(\theta,p,r_*) R^{r_*+1}T,
\end{eqnarray}
where $T:= \min \{\ |t|\ | \ \|\big(\tilde{u}(t), \bar{\tilde{u}}(t)\big)\|_p=R/2\},$ which
 means that for any $|t|\leq T:=\frac{1}{144 C(\theta, p,r_*)R^{r_*-1}},$
\begin{equation}\label{last} \|\tilde{u}(t)\|_p\leq R/4,\qquad \|\big(\tilde{u}(t),\bar{\tilde{u}}(t)\big)\|_p\leq R/2.
\end{equation} 

From   Theorem \ref{Th2}, when $R\ll1$, ${\cal T}_{w_{\theta}}^{(r_*)}$ is an inverse transformation from $B_p(R/2)$ to $B_p(R)$.
Then by (\ref{last}),   the solution $(u(t),\bar{u}(t))$ to systems (\ref{uj}) with $(u(0),\bar{u}(0))\in B_p(R/6)$ satisfies $$ \|(u(t),\bar{u}(t))\|_p\leq R, \quad \mbox{for any } \ |t|\prec \frac{1}{R^{r_*-1}}. $$

\section{Proof of Theorem \ref{T1-1}  and Theorem \ref{Th1}} 
\subsection{Proof of Theorem \ref{T1-1} }

It is common knowledge that $(j^2)_{j\in\mathbb{Z}}  $ are  the eigenvalues of $-\partial_{xx}$ under periodic  boundary condition {$\psi(x,t)= \psi(x+2\pi,t)$} with the corresponding eigenfunctions  $\{ \phi_j(x):= \frac{e^{{\bf i}jx}}{\sqrt{2\pi}}  \}_{j\in\mathbb{Z}}$ .
 Take
\begin{equation}\label{psi}
 \psi(x,t)=\sum_{j\in\mathbb{Z}} u_j(t)\phi_j(x),\quad u_j:=\int_{0}^{2\pi} \psi(x,t) \phi_{-j}(x)dx
 \end{equation}
 into equation (\ref{2.23}) and obtain a Hamiltonian system,
\be\label{a.4}
\begin{cases}
\displaystyle \dot{u}_j = -{\bf i}\frac{\partial H^{w_0}}{\partial \bar{u}_j}(u,\bar{u}) , \medskip \\
\displaystyle \dot{\bar{u}}_j=\ \ {\bf i}\frac{\partial H^{w_0} }{\partial {u}_j}(u,\bar{u}),
\end{cases}\quad
 \mbox{for any}\quad j\in\mathbb{Z}
\ee
 with respect to  2-form $w_0$ in (\ref{sp1}), and the Hamiltonian function has the form
\begin{equation}\label{2}
H^{w_0}(u,\bar u):=H_0^{w_0} + P^{w_0}(u,\bar u),\qquad
\end{equation}
where  
\begin{equation}\label{6.3-a}
H_0^{w_0} := \sum_{j\in \mathbb Z}{\omega^0_j}|u_j|^2,\quad \omega^0_j:= -j^{2}+\widehat{V}_j=-j^{2}+\frac{v^{w_{0}}_j}{\langle j\rangle^m}.
\end{equation}
   Under  assumptions ${\bf A_1}$
 and ${\bf  A_2}$ in section 2.1,     the  power series  $P^{w_0}(u,\bar{u})$ has the following form
         $$P^{w_0}(u,\bar{u})=\sum\limits_{r\geq 3}  \sum\limits_{|k+l|=r  \atop{{\cal M}(l,k)=i\in M_{P^{w_0}_r}} }  \sum_{ (l',k',-i_1)\in {\cal A}_{(P^{w_0})^i_{r,lk}}} (P^{w_0})^{i(l',k',-i_1)}_{r,l k} \cdot ({\cal M}(l',k')-\frac{i_1}2) u^{l}\bar{u}^k,$$
where  $ {\cal A}_{(P^{w_0})^i_{r,lk}}:=\{ (l,0,-i)\ |\ i\in M_{P^{w_1}_r}\subset \mathbb{Z}\}$,  $M_{P^{w_1}_r}$ is a symmetric set and 
\begin{equation*}
 (P^{w_0})^{i(l,0,-i)}_{r,l k}:= \frac{\partial^{|l|}_{\psi^{|l|}}\partial^{|k|}_{\bar{\psi}^{|k|} }\widehat{F}|_{(0,0)}(-i) }{(-2\pi)^{\frac r2} l!k!}      .
\end{equation*}
Moreover, the following equation  holds true for any $l,k$ with $|l+k|=r$ and ${\cal M}(l,k)=i$ 
$$ \overline{(P^{w_0})^{i(l,0,-i)}_{r,l k}}= \overline{ \frac{\partial^{|l|}_{\psi^{|l|}}\partial^{|k|}_{\bar{\psi}^{|k|} }\widehat{F}|_{(0,0)}(-i) }{(-2\pi)^{\frac r2} l!k!}   }= \frac{\partial^{|k|}_{\psi^{|k|}}\partial^{|l|}_{\bar{\psi}^{|l|} }\widehat{F}|_{(0,0)}(i) }{(-2\pi)^{\frac r2} l!k!} =(P^{w_0})^{-i( k,0,-i)}_{r,kl }$$
 and  there exists a constant $C_1>0$ such that 
\begin{eqnarray*}
\sum_{(l^0,k^0,i_1)\in{\cal A}_{(P^{w_0})^i_{r,lk}}}\max\{\langle i_1\rangle,\ \langle 2i-i_1\rangle\}|(P^{w_0})^{i(l^0,k^0,i_1)}_{r, l k}|
  \leq   \frac{C_1^{r-2}}{\langle i\rangle^{\beta}} ,
\end{eqnarray*}
which means  that $P^{w_0}(u,\bar{u})$ has $(\beta,0)$-type symmetric coefficients semi-bounded by $C_1>0$. 

  \begin{Lemma}\label{6.1}
 For any given integers $r_*,N>0$ and real numbers $\alpha> 4m+r_*+8$, $1\gg\gamma>0$,
there exists an open subset $\widetilde{\Theta}^{\theta}_{m} \subset \Theta^{\theta}_m$ ($\Theta^{\theta}_m$  defined in (\ref{vset}) and (\ref{vset1}), respectively) such that for any $V\in \widetilde{\Theta}^{\theta}_{m}$ and any $(l,k)$ belongs to $ {O}^{w_{\theta}}_{r_*+3,N}$ defined in (\ref{o-set}),  
it satisfies
$$ |\langle \omega^{w_{\theta}}(V),I_{\theta}(l-k)\rangle|>\frac{\gamma M_{l,k}}{N^{\alpha}} , $$
where
\begin{equation}\label{om}
\omega^{w_{\theta}}(V)=(\omega_j^{w_{\theta}})_{j\in\mathbb{Z}^{\theta}},\quad \omega_j^{w_{\theta}}:=sgn^{\theta}(j)\cdot \big(-j^2+\frac{v^{w_{\theta}}_j}{\langle j \rangle^m}\big),\ v^{w_{\theta}}_j\in [-1/2,1/2].
\end{equation}
Moreover, $$\mbox{meas}\ (\Theta^{\theta}_m/\widetilde{\Theta}^{\theta}_{m })\leq \frac{  4^{r_*+4+m}r_*^{m+3}\gamma }{N^{\alpha-2m-r_*-3}} .$$

\end{Lemma}
\begin{Remark}
If $\gamma>0 $ is small enough, the set $\widetilde{\Theta}_m^{\theta}$ will have a
 positive measure. In particular, if $\gamma$ approaches to 0, then
the measure of $\widetilde{\Theta}_m^{\theta}$  will approach  to the measure of
${\Theta}_m^{\theta}$.
\end{Remark}
Now give the proof of  Lemma \ref{6.1}.
\begin{proof}

 Denote
 $$ \Theta^{\theta}_m/\widetilde{\Theta}^{\theta}_{m }:= \bigcup_{ 3\leq r\leq r_*+3\ } \big( \bigcup_{(l,k)\in O^{w_{\theta},r}_{r_*+3,N,0} \cup O^{w_{\theta},r}_{r_*+3,N,1} \cup O^{w_{\theta},r}_{r_*+3,N,2} }\Theta^{\theta}_{r,lk}   \big), $$
where $$\Theta^{\theta}_{r,lk}:=\bigg\{ V\in \Theta_m^{\theta} \ \bigg|\ \  \big|\langle\omega^{w_{\theta}}(V), I_{\theta}(l-k)\rangle\big| \leq \frac{\gamma M_{l,k}}{N^{\alpha}}  \bigg\} $$
and
$$ O^{w_{\theta},r}_{r_*+3,N,n}:= \bigg\{ (l,k)\in O^{w_{\theta}}_{r_*+3,N}\   \bigg| \ \ |\Gamma_{>N}(l+k)|=n,\ |l+k|=r   \bigg\},$$
for  $n\in\{0,1,2\}$ and $3\leq r\leq r_*+3$.

 I only give the estimate of  the measure of $\Theta^{\theta}_{r,lk}$ in the case $(l,k)\in O^{w_{\theta},r}_{r_*+3,N,2}$, which is more complex than the case $(l,k)\in O^{w_{\theta},r}_{r_*+3,N,0} \cup O^{w_{\theta},r}_{r_*+3,N,0}$.

When the multi-index $(l,k)\in O^{w_{\theta},r}_{r_*+3,N,2}$,  estimate the measure of $\Theta^{\theta}_{r,lk}$ in  two cases.

(1)The first case  $$(l,k)\in O^{w_{\theta},r}_{r_*+3,N,2a}:=\{ (l,k)\in O^{w_{\theta},r}_{r_*+3,N,2 }  \ | \  |\Gamma_{>N} {l}|=2 \mbox{ or } |\Gamma_{>N}{k}|=2\}.$$
In this case, there exists $ |j_1|,\ |j_0|>N$ such that  $ l_{j_1}= l_{j_0}=1$
or
  $ k_{j_1}=k_{j_0}=1$.
Without loss of  generality,  assume $|j_0|\geq|j_1|>N$ with $l_{j_0}=l_{j_1}=1$. So
$M_{l,k}=|j_0|$ and
$|\omega^{w_{\theta}}_{j_0}(V)|>  j_0^2-\frac12$, $|\omega^{w_{\theta}}_{j_1}(V)|>  j_1^2-\frac12$. The other
frequencies $\omega^{w_{\theta}}_j(V)(|j|\leq N)$ are bounded by $|\omega^{w_{\theta}}_j|\leq  N^2+1$.

If $|j_0|>4\sqrt{r}N $, it follows that
\begin{eqnarray*}
&&|\langle\omega^{w_{\theta}}(V), I_{\theta}(l-k)\rangle|=|l_{j_0} \omega^{w_{\theta}}_{j_0}+ l_{j_1} \omega^{w_{\theta}}_{j_1}+ \sum_{|j|\leq N}\omega^{w_{\theta}}_j(l_j-k_j)|\\
\geq&& |\omega^{w_{\theta}}_{j_0}+\omega^{w_{\theta}}_{j_1}|-|\sum_{j\neq
j_0,j_1,\ |j|\leq N }(l_j-k_j)\omega^{w_{\theta}}_j\ |\geq
  j_0^2-\frac12-( N^2+1)(r-1) \\
>&&\frac{\gamma |j_0|
}{N^{\alpha}}= \frac{\gamma M_{l,k}}{N^{\alpha}}.
\end{eqnarray*}
That means when $|j_0|>4 \sqrt{r}N$, the set ${\Theta}^{\theta}_{r,lk}$ is
empty. So it is only need to calculate the measure  of $\Theta^{\theta}_{r,lk}$ whose multi-index $(l,k)$
being in the following set,
\begin{equation}\label{s-1}
 \widetilde{O^{w_{\theta},r}_{r_*+3,N,2a}}:=\big\{ (l,k)\in O^{w_{\theta},r}_{r_*+3,N,2a}\  \big| N\leq M_{l,k}\leq 4\sqrt{r}N  \big\} \subset O^{w_{\theta},r}_{r_*+3,N,2a},
\end{equation} the number of which are bounded by
\begin{equation}\label{ow}
\sharp{\widetilde{O^{w_{\theta},r}_{r_*+3,N,2a}}}\leq 4\sqrt{r}(4N)^{r} .
\end{equation}
 For any fixed  $(l,k)\in \widetilde{O^{w_{\theta},r}_{r_*+3,N,2a}}$, there exists $4\sqrt{r}N \geq|j_0|>N$ fulfilling  
  $$
 \left\{
 \begin{array}{ll}  l_{j_0}+l_{-j_0}-k_{j_0}-k_{-j_0}\neq 0,& \theta=0 ,\\
 \\
  l_{j_0}-k_{j_0}\neq 0 \ \ \mbox{or}\ \ l_{-j_0}-k_{-j_0}\neq0, &\theta=1.
 \end{array}
 \right.
  $$
such that 
\begin{equation}\label{d-2}
\left\{
\begin{array}{ll}
\big| \frac{\partial g^{w_{0}}}{\partial
v^{w_{0}}_{j_0}}\big|= \frac{|l_{j_0}+l_{-j_0}-k_{j_0}-k_{-j_0}|}{|j_0|^m}\geq\frac{1}{|j_0|^m}\neq0, & \theta=0,
\\
\\
\big| \frac{\partial g^{w_{1}}}{\partial
v^{w_{1}}_{j_0}}\big|=\frac{|l_{j_0}-k_{j_0}|}{\langle j_0\rangle^m} \geq\frac{1}{\langle j_0\rangle^m}\neq0, \  ( \mbox{or }\big| \frac{\partial g^{w_{1}}}{\partial
v^{w_{1}}_{-j_0}}\big| = \frac{|l_{-j_0}-k_{-j_0}|}{\langle j_0\rangle^m}\geq\frac{1}{|j_0|^m}\neq0), &\theta=1.
\end{array}
\right.
\end{equation}
 The measure of $\Theta^{\theta}_{r,lk}$ has the following estimate by (\ref{d-2})
\begin{equation}\label{q1-2}
\mbox{meas}( {\Theta}^{\theta}_{r,lk})\leq
\frac{M_{l,k}\gamma}{N^{\alpha}}\bigg|\big(\frac{\partial g^{w_{\theta}}}{\partial
v^{w_{\theta}}_{j_0}}\big)^{-1}\bigg|\leq \frac{\gamma |j_0|^{m+1}}{N^{\alpha}}\leq
\frac{\gamma (4\sqrt{r})^{m+1}  }{N^{\alpha-m-1}} .
\end{equation}
From (\ref{ow}) and (\ref{q1-2}), it holds that
 \begin{eqnarray}\label{set0}
\mbox{meas}( \bigcup_{(l,k)\in O^{w_{\theta},r}_{r_*+3,N,2a}} {\Theta}^{\theta}_{r,lk}) =\mbox{meas}( \bigcup_{(l,k)\in\widetilde{O^{w_{\theta},r}_{r_*+3,N,2a}}} {\Theta}^{\theta}_{r,lk}) \leq
    \frac{2^{2m+2r+3}\sqrt{r}^{m+2} \gamma}{N^{\alpha-m-r-1}} .
\end{eqnarray}

(2)The second  case
 {\small $$\!(l,k)\!\in\! O^{w_{\theta},r}_{r_*+3,N,2b} :=\{ (l,k)\in O^{w_{\theta},r}_{r_*+3,N,2}  \ | \     \mbox{there exist } |j_0|\neq|j_1|>N,\ l_{j_0}\!=\!k_{j_1}\!=\!1 \mbox{ or}\ l_{j_1}\!=\!k_{j_0}\!=\!1  \}.$$}
 Without loss of generality, assume $|j_1|>|j_0|$ and  $M_{l,k}= |j_1|$. By (\ref{om}), it holds   $|\omega^{w_{\theta}}_{j_{1}}|\geq  j_1^2-\frac12$,
$|\omega^{w_{\theta}}_{j_0}|\leq j_0^2+\frac12$ and
$\omega^{w_{\theta}}_j\leq   N^2+1$, ($|j|\leq N$).
 If
$|j_{1}|>4rN^2$, the following inequality holds 
\begin{eqnarray*}
&&|\langle\omega^{w_{\theta}},I_{\theta}(l-k)\rangle|\geq
|\omega^{w_{\theta}}_{j_{1}}-\omega^{w_{\theta}}_{j_0}|-|\sum_{j\neq j_1,j_0} \omega^{w_{\theta}}_j(l_j-k_j)|\nonumber\\
\geq &&  ( |j_1|^2-( |j_0|^2+1))- (  N^2+1)(r-2)\nonumber\\
\geq&&  (|j_1|+|j_0|)(|j_1|-|j_0|)- ( N^2+1)(r-2)-1\nonumber\\
\geq && \frac{\gamma |j_1|}{N^{\alpha}}= \frac{\gamma
M_{l,k}}{N^{\alpha}}.
\end{eqnarray*}
That means when $M_{l,k}>4rN^2$, the set
$\Theta^{\theta}_{r,lk}$ is empty. It  only needs  to calculate the sum of the set $\Theta^{\theta}_{r,lk}$ with $(l,k)$ being in the following set
$$  \widetilde{O^{w_{\theta},r}_{r_*+3,N,2b} }:=\big\{ (l,k)\in O^{w_{\theta},r}_{r_*+3,N,2b} \  \big|  \  M_{l,k}\leq 4rN^2  \big\}  $$
which is bounded by
\begin{equation}\label{b-a} \sharp{\widetilde{O^{w_{\theta},r}_{r_*+3,N,2b} }}\leq4 r^2(4N)^{r+2}
\end{equation}
There exists $j$ with $4 rN^2 \geq|j|>N$ such that  $l_j+l_{-j}- k_j-k_{-j}\neq 0 $
and
\begin{equation}\label{a-b}
\big| \frac{\partial g^{w_{\theta}}(v)}{\partial v^{w_{\theta}}_j} \big|>\frac{1}{|j|^m}.
\end{equation}
 Denote  a set
$$  \widehat{\Theta}^{\theta}_{r,lk}:= \bigg\{V(x) \in
{\Theta}^{\theta}_m\  \bigg| \ |\langle\omega^{w_{\theta}}(V), I_{\theta}(l-k)\rangle| \leq \frac{4r\gamma
}{N^{\alpha-2}} \bigg\}.
 $$ 
 When $(l,k) \in\widetilde{O^{w_{\theta},r}_{r_*+3,N,2b} } ,$ using the fact $\Theta^{\theta}_{r,lk}\subset \widehat{\Theta}^{\theta}_{r,lk}$  and (\ref{a-b}), it implies that
\begin{equation}\label{b-b}\mbox{meas}\ \Theta_{r,lk}^{\theta}\leq\mbox{meas} \widehat{\Theta}_{r,lk}^{\theta}\leq \bigg|\big(\frac{\partial g^{w_{\theta}}(v)}{\partial v_j}\big)^{-1}\bigg|\frac{4r\gamma}{N^{\alpha-2}}\leq \frac{2(4r)^{m+1}\gamma}{N^{\alpha-2m-2}} .
\end{equation}
From (\ref{b-a}) and (\ref{b-b}), it holds that
\begin{equation}\label{set3}
\mbox{meas} \big( \bigcup_{(l,k)\in  O^{w_{\theta},r}_{r_*+3,N,2b}} \Theta^{\theta}_{r,lk}\big) = \mbox{meas}\big( \bigcup_{(l,k)\in \widetilde{O^{w_{\theta},r}_{r_*+3,N,2b}} } \Theta^{\theta}_{r,lk}\big)\leq \frac{4^{r+m+4}r^{m+3} \gamma}{N^{\alpha-2m-r-4}}.
\end{equation}
Using the same method, the following inequality holds true
 \begin{eqnarray}
\label{set-1}
\mbox{meas}( \bigcup_{(l,k)\in O^{w_{\theta},r}_{r_*+3,N,0}\cup O^{w_{\theta},r}_{r_*+3,N,1}} {\Theta}^{\theta}_{r,lk})\! \leq    \frac{2^{2m+2r+3}\sqrt{r}^{m+2} \gamma}{N^{\alpha-m-r-1}} .
\end{eqnarray}
In view of (\ref{set-1})  and (\ref{set3}), one has
{\small\begin{eqnarray*}
\mbox{meas} (\Theta_m^{\theta}\setminus\widetilde{\Theta}_{m,N}^{\theta})\leq
\sum_{r= 3}^{r_*+3}  \sum_{(l,k)\in(O^{w_{\theta},r}_{r_*+3,N,0} \cup O^{w_{\theta},r}_{r_*+3,N,1} \cup O^{w_{\theta},r}_{r_*+3,N,2})} \mbox{meas}\ {\Theta}^{\theta}_{r,lk} < \frac{  4^{r_*+4+m}r_*^{m+3}\gamma }{N^{\alpha-2m-r_*-4}}.
\end{eqnarray*}}
\end{proof}

Now  Theorem \ref{T1-1} is obtained by Theorem \ref{T22} and Lemma \ref{6.1}.
  The transformation $\psi(x,t)=\sum_{j\in
\mathbb{Z} }u_j(t)\phi_{j}(x) $ is   from
$ \ell_{p}^2$ to $H^p([0,2\pi],\mathbb{C}) $ and satisfies
$$\|u(t)\|_p \leq\|\psi(x,t)\|_{H^p([0,2\pi],\mathbb{C})}=\sup_{0\leq|n|\leq p}\|D_x^n \psi(x,t)\|_{L^2}\leq(2\pi)^p\|u(t)\|_p.$$
 Take $\widetilde{{\varepsilon}}  \ll 1$ small enough. For any $V(x)\in\widetilde{\Theta}_m^0$, if the  initial value of $\psi(x,t)$ to (\ref{2.23}) fulfilling  $\|\psi(x,0)\|_{H^p([0,2\pi],\mathbb{C})}\leq \varepsilon /8< \widetilde{\varepsilon} /8 , $ then it holds that
$\|\psi(x,t)\|_{H^p([0,2\pi],\mathbb{C})}\leq \varepsilon ,$ for any $|t| \prec\varepsilon^{-(r_*-1)} . $

\subsection{Proof of Theorem \ref{Th1}}
\noindent
The following statements deal with the solution to  equation (\ref{NLS-1}). 
It is obvious that $j^2$ $(j\in\mathbb Z^*)$ is the eigenvalue of $(-\partial_{xx})$  under periodic boundary condition and $\phi_{j}(x):= \frac{1}{\sqrt{2\pi}}e^{{\bf i}jx}$  is the corresponding eigenfunction.
Precisely,
 $$ (-\partial_{xx})\phi_{j}(x)=j^2 \phi_{j}(x)\qquad \forall~ j\in \mathbb{Z}^*. $$
For any $\psi\in H_0^{p+1/2}([0,2\pi],\mathbb{C})$,
 $$ \psi(x,t)=\sum_{j\in\mathbb{Z}^*}\widehat{\psi}_j(t) \phi_{j}(x), $$
where $ \widehat{\psi}_j(t) :=  {\int_0^{2\pi}}   \psi(x,t) \phi_{-j}(x)  \mathrm dx$. 
In order to transform equation (\ref{NLS-1}) into  an infinite dimensional Hamiltonian system under a standard symplectic form, I will use a tool given in  \cite{k-p}
\begin{equation}\label{cor-tr}
u(t)=(u_j(t))_{j\in\mathbb{Z}^*},\quad u_j(t):= \frac{\widehat{\psi}_j(t)}{|j|^{\frac12}},\quad \mbox{for any} \quad  j\in \mathbb{Z}^* .
\end{equation}
It is easy to get that if $\psi\in H_0^{p+1/2}([0,2\pi],\mathbb{C})$ then the corresponding Fourier coefficients  vector $(\widehat{\psi}_j)_{j\in\mathbb{Z}^*}\in \ell^2_{p+1/2}({\mathbb{Z}}^*,\mathbb C)$ and $u\in \ell^2_{p}({\mathbb{Z}}^*,\mathbb C)$. Moreover, there exist constants $\tilde{C}_1,\ \tilde{C}_2>0$   such that
\begin{equation}\label{qaz}
 \tilde{C}_1\|u\|_{p}\leq   \|\psi\|_{H^{p+1/2}_0([0,2\pi],\mathbb{C})}\leq \tilde{C}_2 \|u\|_p. 
 \end{equation}
 Under transformation (\ref{cor-tr}),  equation (\ref{NLS-1}) therefore can be written into the following Hamiltonian system with respect to symplectic from $w_1$ defined in (\ref{sp1}), for any $j\in\mathbb{Z}^*$,
\be\label{a.4}
\begin{cases}
\displaystyle \dot{u}_j = -{\bf i}\mbox{sgn}(j)\frac{\partial H^{w_1}(u,\bar{u})}{\partial \bar{u}_j} \medskip \\
\displaystyle \dot{\bar{u}}_j=\ \ {\bf i}\mbox{sgn}(j)\frac{\partial H^{w_1}(u,\bar{u})}{\partial {u}_j} \end{cases}
\ee
with the Hamiltonian
\begin{equation}\label{ex-2}
H^{w_1}(u,\bar u):= H_0^{w_1} + P^{w_1}(u,\bar u),
 \end{equation} where
 \begin{equation}\label{fr}  H_0^{w_1}:=\sum_{j\in \mathbb Z^*}\omega_j^1|u_j|^2, \quad \omega^1_j:=\mbox{sgn}(j)(-j^{2}+\widehat{V}_j)=\mbox{sgn}(j)(-j^{2}+\frac{v_j^{w_1}}{|j|^m})\in \mathbb{R}.
\end{equation}
By  assumptions ${\bf B_1}$-${\bf B_2} $ in  Theorem \ref{Th1}, $P^{w_1}(u,\bar{u})$ has a zero at origin  at last order 3 with the following form 
\begin{eqnarray*}
P^{w_1}(u,\bar{u}):=\sum_{r= 3}^{+\infty}\sum_{|l+k|=r\atop{{\cal M}(l,k)=i\in M_{P^{w_1}_r}\subseteq \mathbb{Z}} } (\tilde{P}^{w_1})^{i }_{r,lk}\prod_{t\in\mathbb{Z}^*} |t|^{\frac{l_t+k_t}{2}} u^l\bar{u}^k,\quad
\end{eqnarray*}
where   $ M_{P^{w_1}_r} \subset \mathbb{Z}$ is a symmetric set
 \begin{eqnarray*} {(\tilde{P}^{w_1})}^i_{r,lk}:=  \frac{1}{(2\pi)^{\frac{r-1}{2}}l!k!}  \hat{\frac{\partial^{l+k} F}{\partial \psi^l \partial\overline{\psi}^k } }\bigg|_{(0,0)}(-i) ,\quad l!=\prod_{j\in\mathbb{Z}^*} l_j!.
 \end{eqnarray*}
 And there exists a constant $C_2>0$ such that for any $r\geq3$ and any $i\in  M_{P^{w_1}_r}\subseteq\mathbb {Z}$,
\begin{equation*}
|  (\tilde{P}^{w_1})^i_{r,lk}|\leq \frac{C_2^{r-2 }}{\langle i\rangle^{\beta} }
\end{equation*}
and
\begin{equation*}
\overline{ (\tilde{P}^{w_1})^i_{r,lk}}=\overline{ \frac{1}{(2\pi)^{\frac{r-1}{2}}l!k!}  \hat{\frac{\partial^{l+k} F}{\partial \psi^l \partial\overline{\psi}^k } }\bigg|_{(0,0)}(-i)}= \frac{1}{(2\pi)^{\frac{r-1}{2}}l!k!}  \hat{\frac{\partial^{l+k} F}{\partial \psi^k \partial\overline{\psi}^l } }\bigg|_{(0,0)}(i)  =(\tilde{P}^{w_1})^{-i}_{r,kl}
\end{equation*}
which means that the power series $P^{w_1}(u,\bar{u})$ has $(\beta,1)$-type symmetric  coefficients semi-bounded by $C_2$.

From (\ref{fr}), the origin is the elliptic equilibrium point of the equation
(\ref{a.4}).
  Using Theorem \ref{T22}, Lemma \ref{6.1} and (\ref{qaz}), for any $V\in \tilde{\Theta}^1_m$,  there exists $\tilde{\varepsilon}\ll 1$ such that  for any $0<\varepsilon <\tilde{\varepsilon}$, if 
 $$ \|\psi(0,x)\|_{H^{p+1/2}_0([0,2\pi],\mathbb{C})} < {\varepsilon}<   \tilde{\varepsilon},$$
 then it satisfies
 $$ \|\psi(t,x)\|_{H^{p+1/2}_0([0,2\pi],\mathbb{C})}<2 \varepsilon,\quad \mbox{for any }\ |t|\prec  \varepsilon^{-r_*+1}. $$

\section{Appendix}
Now the proof of Theorem  \ref{Th2} is given in this section.
\begin{proof}For any $\theta\in\{0,1\}, $ denote
{\small\begin{equation*}
 g^{(-1,w_{\theta})}(u,\bar{u}):=\sum_{r=3}^{r_*+3} P^{w_{\theta}}_r(u,\bar{u}),\  {\cal R}^{N(-1,w_{\theta})}(u,\bar{u}):=0,\   {\cal R}^{T(-1,w_{\theta})}(u,\bar{u}):=\sum_{r=r_*+4}^{\infty} P^{w_{\theta}}_r(u,\bar{u}),
 \end{equation*}}
where $P^{w_{\theta}}_r(u,\bar{u})$ is an $r$-degree homogeneous polynomial  of $P^{w_{\theta}}(u,\bar{u})$.
Thus
(\ref{H-1}) can be rewritten  as
\begin{equation}\label{pr1}
H^{(-1,w_{\theta})}=H^{w_{\theta}}_0 +g^{(-1,w_{\theta})} + {\cal R}^{N(-1,w_{\theta})}+{\cal R}^{T(-1,w_{\theta})} ,\ \mbox{defined in } \ B_p(R_*).
 \end{equation}
  To start with, the results hold at rank  $r=0$.  For any $R<R_*$ and any $N$ satisfying (\ref{N-c}), I will look for a bounded  Lie-transformation ${\cal T}_0^{w_{\theta}}$ to eliminate
the non-normalized monomials of   $\Gamma^N_{\leq 2}g^{(-1,{w_{\theta}})}_{3}$.
 The Lie-transformation   ${\cal T}_0^{w_{\theta}}$  is constructed  from  1-time flow $\Phi_{S^{(0)}_{w_{\theta}}}^t$ of the following equations, 
\begin{equation*}
 \left\{\begin{array}{ll}
 \dot{u}_j=  -{\bf i}  \mbox{sgn}^{ \theta}(j) \nabla_{\bar{u}_j}S_{w_{\theta}}^{(0)} (u,\bar{u}),\\
 \\
\dot{\bar{u}}_j=\ \  {\bf i}  \mbox{sgn}^{ \theta}(j) \nabla_{u_j} S_{w_{\theta}}^{(0)}(u,\bar{u}),
\end{array}\right.\quad    j\in\mathbb{Z}^*,\ \theta\in\{0,1\},
 \end{equation*}
 where $S^{(0)}_{w_{\theta}}$ is undetermined.
Under transformation ${\cal T}_0^{w_{\theta}}$ the new Hamiltonian $H^{(0,{w_{\theta}})}$  has  the following  form,
{\footnotesize\begin{eqnarray}&&
H^{(0,{w_{\theta}})}=H^{(-1,{w_{\theta}})}\circ {\cal T}_0^{w_{\theta}}=( H_0^{w_{\theta}} +g^{(-1,{w_{\theta}})}+{\cal R}^{T(-1,w_{\theta})})
\circ{\Phi_{S_{w_{\theta}}^{(0)}}^{1}}\nonumber\\
&=&  H_0^{w_{\theta}} \nonumber\\
\label{4.31-A}
 &+& \{H_0^{w_{\theta}}  , S_{w_{\theta}}^{(0)}\}_{w_{\theta}}+g^{(-1,{w_{\theta}})}_3
\\\label{4.32-A}
 &+&\!\sum_{t=4}^{r_*+3}g^{(-1,{w_{\theta}})}_{t}\!+\! \sum_{\nu\geq2}(H_0^{w_{\theta}} )_{(\nu,S_{w_{\theta}}^{(0)})} +
\sum_{\nu\geq 1}\sum_{t=3}^{r_*+3}(g_t^{(-1,{w_{\theta}})})_{(\nu, S_{w_{\theta}}^{(0)})}+\sum_{\nu\geq 0}\sum_{t=r_*+4}^{\infty}({\cal R}_t^{T(-1,{w_{\theta}})})_{(\nu, S_{w_{\theta}}^{(0)})},\label{4.35-A}
\end{eqnarray}}
where $ (\ .\ )_{(\nu,S_{w_{\theta}}^{(0)})}$ is defined  in (\ref{d}).
 The auxiliary Hamiltonian function $S_{{w_{\theta}}}^{(0)}$
are obtained by solving the following homological  equation
 \begin{equation}\label{e.1-1}
\{H_0^{w_{\theta}} , S_{w_{\theta}}^{(0)}\}_{w_{\theta}}+\Gamma^N_{\leq2 } P^{w_{\theta}}_{3}=Z^{w_{\theta}}_{3}.
\end{equation}
Using  Remark
\ref{trunction}, $\Gamma^N_{>2} P_3^{w_{\theta}}$ and $\Gamma^N_{\leq2}P_3^{w_{\theta}}$ are still having $(\beta,\theta)$-type symmetric coefficients semi-bounded by $C(\theta,-1)=C_{\theta}>0.$ 
From Lemma \ref{lem2},
 $Z^{{w_{\theta}}}_{3}$ is  $(\theta,\gamma,\alpha,N)$-normal form of $ \Gamma^N_{\leq 2}P^{w_{\theta} }_{3}$
and the Hamiltonian vector field of $S_{{w_{\theta}}}^{(0)}$ satisfies
\begin{eqnarray}\label{j5}
 \|X^{w_{\theta}}_{S_{w_{\theta}}^{(0)}}\|_{p}\leq 
8  c^2 C_{\theta}3^{p+1}  \frac{
N^{\alpha}}{{\gamma}} \|u\|_{p}\|u\|_{2} \quad \mbox{for any }\ {(u,\bar{u})\in B_p(2R)} .
 \end{eqnarray}
 From (\ref{e.1-1}),  the following  holds true
$$ (\ref{4.31-A})=Z^{w_{\theta}}_{3}(u,\bar{u})+ \Gamma^N_{> 2}P^{w_{\theta} }_{3}.$$
The Lie-transformation ${\cal T}_0^{w_{\theta}} $ satisfies
 \begin{eqnarray}\label{18-A}
&&\sup_{(u,\bar{u})\in B_p(R)}\|{\cal T}_0^{w_{\theta}}(u,\bar{u})-(u,\bar{u})\|_{p} =\sup_{(u,\bar{u})\in B_p(R)} \|\Phi_{S_{w_{\theta}}^{(0)}}^1(u,\bar{u})-(u,\bar{u})\|_{p}\nonumber\\
&=&  \sup_{(u,\bar{u})\in B_p(R)}\big\|\int_{t=0}^1 X^{w_{\theta}}_{S_{w_{\theta}}^{(0)}}\circ \Phi_{S_{w_{\theta}}^{(0)}}^{\tau}(u,\bar{u})(\tau) d\tau \big\|_{p}.
  \end{eqnarray}
  Use the bootstrap method to estimate ${\cal T}_0^{w_{\theta}} $. First, assume that
 \begin{equation}\label{hjj}
 \Phi_{S_{w_{\theta}}^{(0)}}^t:B_p(R)\rightarrow B_p(2R),\quad  \mbox{ for any }\quad t\in[0,1].
 \end{equation}
By (\ref{j5})-(\ref{hjj}),  the following inequality holds true  
  \begin{eqnarray}\label{18-b}
&&\sup_{(u,\bar{u})\in B_p(R)}\| {\cal T}_0^{w_{\theta}}(u,\bar{u})-(u,\bar{u})\|_{p}\leq  \sup_{(u,\bar{u})\in B_p(2R)} \big\| \int_{t=0}^1 X^{w_{\theta}}_{S_{w_{\theta}}^{(0)}}(\tau) d\tau\big\|_{p}\nonumber\\
&\leq &
8 C_{\theta}  c^2  \frac{
N^{\alpha}}{{\gamma}}  3^{p+1} (2R)^2.
\end{eqnarray}
Since  $R$ is small enough,
  from  (\ref{N-c}) and (\ref{18-b}),  the transformation $ {\cal T}_0^{w_{\theta}} $
satisfies
 $$ \sup_{(u,\bar{u})\in B_p(R)}\| {\cal T}_0^{w_{\theta}}(u,\bar{u})-(u,\bar{u})\|_{p} \leq
  R,$$
which means
\begin{equation}\label{k}
   {\cal T}_0^{w_{\theta}} :B_p(R)\rightarrow B_p(2R).
  \end{equation}
      Denote $ {\cal T}^{w_{\theta}(0)}:={\cal T}_0^{w_{\theta}}$. By (\ref{N-c}), (\ref{18-A}) and (\ref{18-b}),
    it is  verified  that   (\ref{sss}) holds for rank  $r=0$:
 \begin{eqnarray*}
\sup_{(u,\bar{u})\in B_p(R)}\| {\cal T}^{w_{\theta}(0)}(u,\bar{u})-(u,\bar{u}) \|_{p} 
\leq C(\theta,p,r_*) R^{2-\frac{1}{r_*+1}}.
\end{eqnarray*}
Set
\begin{equation}\label{(13)-A}
Z^{(0,{w_{\theta}})}:=Z^{(-1,{w_{\theta}})}+Z^{w_{\theta}}_{3}, \quad {\cal R}^{N(0,{w_{\theta}})}:={\cal
R}^{N(-1,{w_{\theta}})}+ \Gamma^N_{>2}g^{(-1,{w_{\theta}})}_{3} .
\end{equation}
Since $Z^{w_{\theta}}_3$ and $ {g}^{(-1,{w_{\theta}})}_3$ having $(\beta, \theta)$-type symmetric coefficients, then $Z^{(0,w_{\theta})}$ and ${\cal R}^{N(0,{w_{\theta}})}$ are still having $(\beta,\theta)$-type symmetric coefficients.
Denote the $r_*+3$-degree polynomial of power series (\ref{4.35-A}) as $g^{(0,{w_{\theta}})}$ and the remainder  as ${\cal R}^{T(0,{w_{\theta}})}$, i.e., 
\begin{eqnarray*}
 g^{(0,{w_{\theta}})}:=\sum_{t= 4}^{r_*+3}g^{(0,{w_{\theta}})}_t,\quad
{\cal R}^{T(0,{w_{\theta}})}:=\sum_{t>r_*+3}{\cal R}^{T(0,{w_{\theta}})}_t ,
\end{eqnarray*}
where for any $4\leq t\leq r_*+3$,
\begin{eqnarray*}
 g^{(0,{w_{\theta}})}_t:=g^{(-1,{w_{\theta}})}_{t}+
(H_0^{w_{\theta}} )_{(t-2,S_{{w_{\theta}}}^{(0)})}+
\sum_{n'=1}^{t-3}(g^{(-1,{w_{\theta}})}_{t-n'})_{(n',S_{w_{\theta}}^{(0)})}
\end{eqnarray*}
and for any $t> r_*+3$
\begin{eqnarray*}
 {\cal R}^{T(0,{w_{\theta}})}_t :=(H_0^{w_{\theta}} )_{(t-2,S_{w_{\theta}}^{(0)})}+
\sum_{n'=1}^{t-3}(g^{(-1,w_{\theta})}_{t-n'})_{(n',S_{w_{\theta}}^{(0)})}+
\sum_{n'=1}^{t-3}({\cal R}^{T(-1,w_{\theta})}_{t-n'})_{(n',S_{w_{\theta}}^{(0)})}.
 \end{eqnarray*}
 From     Remark \ref{jjd} and Lemma \ref{lem4.2},  $g^{(0,w_{\theta})}$ and ${\cal R}^{T(0,w_{\theta})}$ have $(\beta,\theta)$-type symmetric coefficients.
In order to estimate them, one needs to estimate  the coefficients of functions
$(H_0^{w_{\theta}} )_{(t-2,S_{w_{\theta}}^{(0)})}$,
$\sum_{n'=1}^{t-3}(g^{(-1,w_{\theta})}_{t-n'})_{(n',S_{w_{\theta}}^{(0)})}$ and $\sum_{n'=1}^{t-3}({\cal R}^{T(-1,w_{\theta})}_{t-n'})_{(n',S_{w_{\theta}}^{(0)})}$.
By Remark \ref{jjd}, when $\theta=0$,  for any $|l+k|=t\geq4$, any $1\leq n'\leq t-3$ and any $i\in M_{ (g_{t-n'}^{(-1,w_0)})_{(n',S_{w_{0}}^{(0)})}}$,   it holds 
\begin{eqnarray}\label{sdd}
 &&\sum_{(l^0,k^0,i^0)\in {\cal A}_{ \big((g_{t-n'}^{(-1,w_0)})_{(n',S_{w_0}^{(0)})}\big)_{t,lk}^i}}\big|\big( (g_{t-n'}^{(-1,w_0)})_{(n',S_{w_0}^{(0)})}\big)^{i(l^0,k^0,i^0)}_{t,lk}\big|\cdot \max \{ \langle i^0\rangle,\ \langle 2i-i^0\rangle\}\nonumber\\
 &\leq& \frac{C_0^{t-2}}{\langle i\rangle^{\beta}} (\frac{144N^{\alpha+2}c2^{\beta}}{\gamma})^{n'} \frac{1}{n'!}\prod_{n=0}^{n'-1}(t-n'+n+1 )
 \end{eqnarray}
 and when $\theta=1$, it holds
\begin{eqnarray}\label{sdd}
 \big| \big( (\tilde{g}_{t-n'}^{(-1,w_1)})_{(n',S_{w_1}^{(0)})}\big)^{i}_{t,lk}\big|
 \leq \frac{C_1^{t-2}}{\langle i\rangle^{\beta} }(\frac{72cN^{\alpha+1}2^{\beta}}{\gamma})^{n'} \frac{1}{n'!}\prod_{n=0}^{n'-1}(t-n'+n+1 ).
 \end{eqnarray}
By equation (\ref{e.1-1}), when $\theta =0$ it follows
\begin{eqnarray}\label{sddj}
 &&\sum_{
   (l^0,k^0,i^0)\in{\cal A}_{\big(
       (  Z^{w_0}_3 )_{(t-3,S^{(0)}_{w_{0}})}
       \big)^i_{t,lk}
       }
        }|\big((Z^{w_0}_3)_{(t-3,S^{(0)}_{w_{0}})} \big)^{i(l^0,k^0,i^0)}_{t,lk} |\cdot \max \{\langle i^0\rangle,\ \langle 2i-i^0\rangle\}\nonumber\\
&\leq&  \frac{1}{\langle i\rangle^{\beta}}\big( 144\frac{N^{\alpha+2}c2^{\beta}}{\ \gamma\ }\big)^{t-3} \frac{C_0^{t-2}}{(t-3)!}\prod_{n=0}^{t-4}(4+n )  ;
\end{eqnarray}
%\color{black}
 when $\theta =1$ it follows
\begin{eqnarray}\label{sddj}
 |\big({(\tilde{Z}^{w_1}_3)}_{(t-3,S^{(0)}_{w_1})} \big)^{i}_{t,lk} |\leq\frac{ C_1^{t-2}}{\langle i\rangle^{\beta}}   \big( 72c\frac{N^{1+\alpha}2^{\beta}}{\ \gamma\ }\big)^{t-3} \frac{1}{(t-3)!}\prod_{n=0}^{t-4}(4+n ).
\end{eqnarray}
When $N$ satisfies (\ref{N-c}), using (\ref{sdd})-(\ref{sddj}),  in the case $\theta=0 $ it holds that
{\footnotesize \begin{eqnarray*}
&&\sum_{(l^0,k^0,i^0)\in {\cal A}_{(g^{(0,w_0)})_{t,lk}^i}}\big|(g^{(0,w_0)})^{i(l^0,k^0,i^0)}_{t,lk}\big|\cdot  \max \{\langle i^0\rangle,\ \langle 2i-i^0\rangle\}\nonumber\\
&\leq &  \sum_{(l^0,k^0,i^0)\in {\cal A}_{(g^{(-1,w_0)})_{t,lk}^i}} |(g^{(-1,w_0)})^{i(l^0,k^0,i^0)}_{t,lk}|\cdot  \max \{\langle i^0\rangle,\ \langle 2i-i^0\rangle\}\nonumber\\
 &&+ \sum_{{(l^0,k^0,i^0)\in{\cal A}_{ \big( (Z^{w_0}_3)_{(t-3,S_{w_0}^{(0)})}\big)^i_{t,lk}}} }\big| \big((Z^{w_0}_3)_{(t-3,S_{w_0}^{(0)})} \big)^{i(l^0,k^0,i^0)}_{t,lk} \big|\cdot  \max \{\langle i^0\rangle,\ \langle 2i-i^0\rangle\} \nonumber\\
&&+\sum_{(l^0,k^0,i^0)\in {\cal A}_{ \big((g_{t-n'}^{(-1,w_0)})^i_{(n',S_{w_0}^{(0)})}\big)_{t,lk}^i}}
\sum_{n'=1}^{t-3}\big|\big((g^{(-1,w_0)}_{t-n'})_{
(n',S_{w_0}^{(0)})}\big)^{i(l^0,k^0,i^0)}_{t,lk}\big|\cdot  \max \{\langle i^0\rangle,\ \langle i-i^0\rangle\}\nonumber\\
 & \leq & \frac{1}{\langle i\rangle^{\beta}}(C_0 c\frac{2^{\beta}N^{2\alpha}}{\gamma})^{t-2}   =:\frac{1}{\langle i\rangle^{\beta}}\big(C(0,0)\big)^{t-2}.  \end{eqnarray*}
  }
 When $\theta=1$ it follows 
\begin{eqnarray*}
&& \big| {(\widetilde{g}^{(0,w_1)})}^{i}_{t,lk}\big|\nonumber\\
&\leq &  |{(\widetilde{g}^{(-1,w_1)})}^{i}_{t,lk}|+ |{\big((\widetilde{Z^{w_1}_3})_{(t-3,S^{(0)}_{w_1})} \big)}^{i}_{t,lk} |+
\sum_{n'=1}^{t-3}\big|\big(
{(\widetilde{g}^{(-1,w_1)}_{t-n'})}_{(n',S^{(0)}_{w_{1}}}\big)^{i}_{t,lk}\big|\leq   \frac{\big(C(1,0)\big)^{t-2}}{\langle i\rangle^{\beta}}.
  \end{eqnarray*}
Similarly,  ${\cal R}^{T(0,w_{\theta})}(u,\bar{u})$ has still $(\beta,\theta)$-type symmetric coefficients semi-bounded by $ C(\theta,0)$.  

Now  assume  that the results hold for rank  $r< r_*$. By these assumptions, there exist a real
number $\tilde{R}<R_*$ and a Lie-transformation which changes
 Hamiltonian (\ref{pr1}) into the following form
$$ H^{(r,w_{\theta})}=H_0^{w_{\theta}} +Z^{(r,w_{\theta})}+{\cal R}^{N(r,w_{\theta})}+g^{(r,w_{\theta})}+{\cal R}^{T(r,w_{\theta})},$$
which is defined in $B_p(R_r)$ ($ R<\tilde{R}<R_*$), where $R_r:=\frac{2r_*-r}{2r_*}R$.
  One should construct a bounded  Lie-transformation ${\cal T}_r^{w_{\theta}}$ to eliminate
the non-normalized monomials of $\Gamma^N_{\leq 2}g^{(r,w_{\theta})}_{r+4}$.
 Because $g_{r+4}^{(r,w_{\theta})}$ have $(\beta,\theta)$-type symmetric  coefficients, by Remark \ref{trunction}, the coefficients of $\Gamma^N_{\leq 2}g^{(r,w_{\theta})}_{r+4}$ and $\Gamma^N_{>2}g^{(r,w_{\theta})}_{r+4}$ are $(\beta,\theta)$-type symmetric  coefficients
semi-bounded  by $C(\theta,r)$.
Make use of the 1-time flow of the following equation,
 for any $j\in\mathbb{Z}^*$
\begin{equation*}
 \left\{\begin{array}{ll}
 \dot{u}_j= \ \ {\bf i}  \mbox{sgn}^{\theta}(j)\partial_{\bar{u}_j}S_{w_{\theta}}^{(r)} (u,\bar{u}),\\
 \\
\dot{\bar{u}}_j=-{\bf i} \mbox{sgn}^{\theta}(j)\partial_{u_j} S_{w_{\theta}}^{(r)}(u,\bar{u}),
\end{array}\right.
 \end{equation*} to define a Lie-transformation ${\cal T}_r^{w_{\theta}}$, under which the new Hamiltonian  has  the
following  form formally,
{\footnotesize\begin{eqnarray}
&&H^{(r+1,w_{\theta})}:=H^{(r,w_{\theta})}\circ {\cal T}_r^{w_{\theta}}\nonumber\\
&=&H_0^{w_{\theta}} +Z^{(r,w_{\theta})}+{\cal R}^{N(r,w_{\theta})}\nonumber\\
\label{4.31} & + &\{H_0^{w_{\theta}} , S_{w_{\theta}}^{(r)}\}_{w_{\theta}}+g^{(r,w_{\theta})}_{r+4}
\\\label{4.32}
 &+& \sum_{t=r+5}^{r_*+3}g^{(r,w_{\theta})}_t + \sum_{\nu\geq2}(H_0^{w_{\theta}} )_{(\nu,S_{w_{\theta}}^{(r)})} +
\sum_{\nu\geq 1}(Z^{(r,w_{\theta})}+g^{(r,w_{\theta})} +{\cal
R}^{N(r,w_{\theta})})_{(\nu,S^{(r)}_{w_{\theta}})}+\sum_{\nu\geq0}({\cal R}^{T(r,w_{\theta})})_{(\nu,S^{(r)}_{w_{\theta}})}.\nonumber\\\label{4.35}
\end{eqnarray}}
The auxiliary Hamiltonian $S_{w_{\theta}}^{(r)}$
can be obtained by solving the following homological  equation
 \begin{equation}\label{e.1}
\{H_0^{w_{\theta}} , S_{w_{\theta}}^{(r)}\}_{w_{\theta}}+\Gamma^N_{\leq2}g^{(r,w_{\theta})}_{r+4}=Z_{r+4}.
\end{equation}
 From  Lemma \ref{lem2},  $Z_{r+4}$ is
$(\theta,\gamma,\alpha,N)$-normal form of $\Gamma^N_{\leq 2}g^{(r,w_{\theta})}_{r+4}$ and
$$ (\ref{4.31})=Z_{r+4}+\Gamma^N_{>2}g^{(r,w_{\theta})}_{r+4}.$$
 The Hamiltonian vector field $X^{w_{\theta}}_{S^{(r)}_{w_{\theta}}}$  satisfies
\begin{equation}\label{c1}
\sup_{(u,\bar{u})\in B_p(R_r)}\| X^{w_{\theta}}_{S_{w_{\theta}}^{(r)}}(u,\bar{u}) \|_{p}
  \leq 8(C(\theta,r))^{r+2} (r+4)^{p+1}c^{r+3} \frac{N^{\alpha}}{\gamma}R^{r+3} .
\end{equation}
Using (\ref{c1}) and  bootstrap method, suppose that   \begin{equation}\label{k}
\Phi_{S_{w_{\theta}}^{(r)}}^t:B_p(R_{r+1})\rightarrow B_p(R_r),
\end{equation}
 for any $t\in [0,1]$. 
 \begin{eqnarray}\label{18}
&& \sup_{(u,\bar{u})\in B_p (R_r)}\|{\cal T}_r^{w_{\theta}}(u,\bar{u})-(u,\bar{u})\|_{p} = \sup_{(u,\bar{u})\in B_p(R_r)}\|\Phi_{S_{w_{\theta}}^{(r)}}^1(u,\bar{u})-(u,\bar{u})\|_{p}\nonumber\\
 &=& \sup_{(u,\bar{u})\in B_p(R_r)}\big\| \int_{t=0}^1 X^{w_{\theta}}_{S_{w_{\theta}}^{(r)}}\circ\Phi_{S_{w_{\theta}}^{(r)}}^{\tau}(u,\bar{u})(\tau) d\tau \big\|_{p}\nonumber\\
&\leq& 16 (C(\theta, r))^{r+2} (r+4)^{p+1}c^{r+3}  \frac{N^{\alpha}}{\gamma}R^{r+3}.
 \end{eqnarray}
By (\ref{N-c}) and (\ref{18}),  the transformation ${\cal T}_r $ satisfies
 $$ \sup_{(u,\bar{u})\in B_p(R)}\|{\cal T}_r^{w_{\theta}}(u,\bar{u})-(u,\bar{u})\|_{p} \leq
  \delta/2=(R_{r}-R_{r+1})/2,$$
which verifies (\ref{k}).
  %the above inequality is holding when $N$ is big enough.
   Denote $ {\cal T}^{^{w_{\theta}}(r+1)}:={\cal T}^{^{w_{\theta}}(r)}\circ {\cal T}_r^{w_{\theta}}$. By (\ref{18}) and (\ref{k}),
    noting that $R < \tilde{R}<R_*<1$, it holds
 \begin{eqnarray}\label{s2}
&&\sup_{(u,\bar{u})\in B_p\big(({R_{r}+R_{r+1}})/2\big)}\| {\cal T}^{^{w_{\theta}}(r+1)}(u,\bar{u})-(u,\bar{u})\ \|_{p}\nonumber\\
&\leq &\sup_{(u,\bar{u})\in B_p\big(\frac{R_{r}+R_{r+1}}2\big)} \big(\|{\cal T}^{^{w_{\theta}}(r)}\circ {\cal T}_r^{w_{\theta}}(u,\bar{u})-{\cal
T}_r^{w_{\theta}}(u,\bar{u})\|_{p}+\|{\cal
T}_r^{w_{\theta}}(u,\bar{u})-(u,\bar{u})\|_{p}\big)\nonumber\\
&\leq &\sup_{(u,\bar{u})\in B_p (R_{r})}\|{\cal
T}^{(r)}(u,\bar{u})-(u,\bar{u})\|_{p}+16\big(C(\theta,r)\big)^{r+2} (r+4)^{p+1}c^{r+3} \frac{N^{\alpha}}{\gamma}  R^{r+3} . 
\end{eqnarray}
Because  $C(\theta,t)\leq C(\theta,t+1)$ for any positive integer $t$, from (\ref{s2}), one has that
{\footnotesize\begin{eqnarray*} 
&&\sup_{(u,\bar{u})\in B_p\big(({R_{r}+R_{r+1}})/2\big)}\| {\cal T}^{^{w_{\theta}}(r+1)}(u,\bar{u})-(u,\bar{u})\ \|_{p}\nonumber\\
 &\leq&16\sum_{t=3}^{r+3}\frac{N^{\alpha}}{\gamma}\big(C(\theta,t-3)\big)^{t-2}t^{p+1}c^{t-1} 
R^{t-1} +16\frac{N^{\alpha}}{\gamma}\big(C(\theta,r+1)\big)^{r+2} (r+4)^{p+1}c^{r+3} R_{r}^{r+3}  \nonumber\\
&\leq&
16\frac{N^{\alpha}}{\gamma}\sum_{t=3}^{r+4} \big(C(\theta,t-3)\big)^{t-2} \   t^{p+1}c^{t-1} 
R^{t-1} \leq R^{2-\frac{1}{r_*+1}}.
\end{eqnarray*}}
Denote
\begin{equation}\label{(13)}
Z^{(r+1,w_{\theta})}:=Z^{(r,w_{\theta})}+Z_{r+4}, \quad {\cal R}^{N(r+1,w_{\theta})}:={\cal
R}^{N(r,w_{\theta})}+\Gamma^N_{>2}g^{(r,w_{\theta})}_{r+4}.
\end{equation}
 By Remark \ref{jjd} and Lemma \ref{lem4.2},  $Z^{(r+1,w_{\theta})}$ and ${\cal R}^{N(r+1,w_{\theta})}$ have $(\beta,\theta)$-type symmetric coefficients. Denote
$$ g^{(r+1,w_{\theta})}=\sum_{t=r+5}^{r_*+3}g^{(r+1,w_{\theta})}_t,\quad {\cal R}^{T(r+1,w_{\theta})}=\sum_{t>r_*+3}{\cal R}^{T(r+1,w_{\theta})}_t,$$
where
\begin{eqnarray}\label{g+}
&&g^{(r+1,w_{\theta})}_t:=\nonumber
\\& & \left\{\begin{array}{lll}
g^{(r,w_{\theta})}_{t}+ (Z_{r+4}-\Gamma_{N\leq 2}g_{r+4}^{(r,w_{\theta})})_{(\frac{t-r-4}{r+2},S_{w_{\theta}}^{(r)})}
+\sum_{n'=0}^{[\frac{t-3}{r+2}]}({\cal
R}^{N(r,w_{\theta})}_{t-n'(r+2)})_{(n',S_{w_{\theta}}^{(r)})}\nonumber\\ +\sum_{n'=0}^{[\frac{t-(r+4)}{r+2}]}
(g^{(r,w_{\theta})}_{t-n'(r+2)})_{(n',S^{(r)}_{w_{\theta}})}
+\sum_{n'=2}^{[\frac{t-3}{r+2}]}(Z^{(r,w_{\theta})}_{t-n'(r+2)})_{(n',S_{w_{\theta}}^{(r)})},\ \mbox{when}\
(r+2) | (t-2);\nonumber\\
\\
g^{(r,w_{\theta})}_{t}
+\sum_{n'=0}^{[\frac{t-3}{r+2}]}({\cal
R}^{N(r,w_{\theta})}_{t-n'(r+2)})_{(n',S_{w_{\theta}}^{(r)})}+\sum_{n'=0}^{[\frac{t-(r+4)}{r+2}]}
(g^{(r,w_{\theta})}_{t-n'(r+2)})_{(n',S_{w_{\theta}}^{(r)})}\nonumber\\
+\sum_{n'=2}^{[\frac{t-3}{r+2}]}(Z^{(r,w_{\theta})}_{t-n'(r+2)})_{(n',S_{w_{\theta}}^{(r)})},\ \mbox{when}\
(r+2) \nmid (t-2);
\end{array}
\right.
 \end{eqnarray}
and
\begin{eqnarray}\label{R+}
&&{\cal R}^{T(r+1,w_{\theta})}_{t}:= \nonumber\\
&&\left\{
\begin{array}{ll}
  (Z_{r+4}-\Gamma^N_{\leq2}g_{r+4}^{(r,w_{\theta})})_{
  (\frac{t-r-4}{r+2},S_{w_{\theta}}^{(r)})}
+\sum_{n'=1}^{[\frac{t-3}{r+2}]}\big({\cal
R}^{N(r,w_{\theta})}_{t-n'(r+2)}\big)_{(n',S_{w_{\theta}}^{(r)})}\nonumber\\
+\sum_{n'=1}^{[\frac{t-(r+4)}{r+2}]}
\big(g^{(r,w_{\theta})}_{t-n'(r+2)}\big)_{(n',S^{(r)}_{w_{\theta}})}
+\sum_{n'=1}^{[\frac{t-3}{r+2}]}\big(Z^{(r,w_{\theta})}_{t-n'(r+2)}\big)_{(n',S_{w_{\theta}}^{(r)})}
\nonumber\\
+\sum_{n'=0}^{[\frac{t-r_*-4}{r+2}]}\big({\cal
R}^{T(r,w_{\theta})}_{t-n'(r+2)}\big)_{(n',S_{w_{\theta}}^{(r)})},\quad \quad \quad  \mbox{when}\ (r+2)|(t-2);\\
\\
 \sum_{n'=1}^{[\frac{t-3}{r+2}]}\big({\cal
R}^{N(r,w_{\theta})}_{t-n'(r+2)}\big)_{(n',S^{(r)}_{w_{\theta}})}+\sum_{n'=1}^{[\frac{t-(r+4)}{r+2}]}
\big(g^{(r,w_{\theta})}_{t-n'(r+2)}\big)_{(n',S^{(r)}_{w_{\theta}})}\nonumber\\
+\sum_{n'=1}^{[\frac{t-3}{r+2}]}\big(Z^{(r,w_{\theta})}_{t-n'(r+2)}\big)_{(n',S^{(r)}_{w_{\theta}})}
+\sum_{n'=0}^{[\frac{t-r_*-4}{r+2}]}\big({\cal
R}^{T(r,w_{\theta})}_{t-n'(r+2)}\big)_{(n',S^{(r)}_{w_{\theta}})},\ \mbox{when}\ (r+2)\nmid (t-2);
\end{array}
\right.
 \end{eqnarray}
 where $[a]$ denotes the integer part of the  real number $a$.
Using Lemma \ref{lem4.2} and Remark \ref{jjd},
 from the fact that   $ g^{(r,w_{\theta})}, \ {\cal R}^{T(r,w_{\theta})}, \ {\cal R}^{N(r,w_{\theta})}$ and $ Z^{(r,w_{\theta})}$ have $(\beta,\theta)$-type symmetric coefficients semi-bounded by $C(\theta,r)$,
then $g^{(r+1,w_{\theta})}$ and ${\cal R}^{T(r+1,w_{\theta})}$ also have $(\beta ,\theta)$-type symmetric coefficients.

\noindent 
When $\theta=0$, using Remark \ref{jjd},  the followings estimates hold: 
for any $|l+k|=t$  with ${\cal M}(l,k)=i\in M_{(g^{(r,w_{0})}_{r+4})_{(\frac{t-r-4}{r+2},S_{w_{0}}^{(r)})}}$,  
{\footnotesize\begin{eqnarray}\label{zbzd1}
&&  \sum_{(l^0,k^0,i^0)\in{\cal A}_{\big((Z_{r+4}-\Gamma^N_{\leq2}g_{r+4}^{(r,w_{0})})_{(\frac{t-r-4}{r+2} ,S_{w_{0}}^{(r)})} \big)_{t,lk}^i}} \big|\big((Z_{r+4}-\Gamma^N_{\leq2 }g_{r+4}^{(r,w_{0})})_{(\frac{t-r-4}{r+2} ,S_{w_{0}}^{(r)})} \big)^{i(l^0,k^0,i^0)}_{r,lk}\big|\nonumber\\
&&\cdot \max \{\langle i^0\rangle,\ \langle 2i-i^0\rangle\} \nonumber\\
&\leq&
  \frac{\big(C(\theta,r)\big)^{r+2}}{ \langle i\rangle^{\beta}} \big(2^{\beta+2}(r+4)^2\big(C(\theta,r)\big)^{r+2}c  \frac{N^{\alpha+1}}{\gamma}\big)^{\frac{t-r-4}{r+2}}\frac{
 (2N)^{(t-r-4)(r+2)}}{(\frac{t-r-4}{r+2})!}\prod_{n=1}^{\frac{t-r-4}{r+2}} \big(t+1-n(r+2)\big);\nonumber\\
\end{eqnarray}}
for any  $|l+k|=t$ with ${\cal M}(l,k)=i\in M_{(g^{(r,w_{0})}_{t-n'(r+2)})_{(n',S^{(r)}_{w_{0}})}}$,  
\begin{eqnarray}\label{zbzd2}
&&\sum_{(l^0,k^0,i^0)\in {\cal A}_{\big((g^{(r,w_0)}_{t-n'(r+2)})_{(n',S_{w_{0}}^{(r)})}\big)_{t,lk}^i}} \big|\big(
(g^{(r,w_0)}_{t-n'(r+2)})_{(n',S^{(r)}_{w_{0}})}\big)^{i(l^0,k^0,i^0)}_{t,lk}\big|\cdot  \max \{\langle i^0\rangle,\ \langle 2i-i^0\rangle\}\nonumber\\
&\leq&\!\frac{
\big(\!C(\theta,r)\!\big)^{t\! -\! n'\! (r+2\!)-2}}{\langle i\rangle^{\beta} }\big(\! 2^{\beta+2}(r+4)^2 \big(C(\theta,r)\big)^{r+2} c \frac{N^{\alpha+1}}{\gamma}\big)^{n'}\frac{(2N)^{(r+2)n'}
}{n'!} \!\prod_{n=1}^{n'}\! \big(\! t\! +\! 1\!-\! n(r+2)\big);\nonumber\\
\end{eqnarray}
for any  $|l+k|=t$  with ${\cal M}(l,k)=i\in M_{(Z^{(r,w_{0})}_{t-n'(r+2)})_{(n',S^{(r)}_{w_{0}})}} ,$  
{\small\begin{eqnarray}\label{zbzd3}
&&\sum_{(l^0,k^0,i^0)\in {\cal A}_{\big((Z^{(r,w_{0})}_{t-n'(r+2)})_{(n',S^{(r)}_{w_{0}})}\big)_{t,lk}^i}} \big|\big( (Z^{(r,w_{  0})}_{t-n'(r+2)})_{(n',S^{(r)}_{w_{0}})}\big)^{i(l^0,k^0,i^0)}_{t,lk}\big|\cdot \max \{\langle i^0\rangle,\ \langle 2i-i^0\rangle\}\nonumber\\
&\leq&
 \!\frac{\big(\! C(\theta,r\!)\big)^{\! t\!-\!n'\!(r+2\!)\!-2}}{\langle i\rangle^{\beta} }\big(2^{\beta+2}(r+4)^2\big(\!C(\!\theta,r\!)\big)^{r+2}c  \frac{N^{\alpha+1}}{\gamma}\big)^{n'}\frac{(2N)^{(r+2)n'}}{n'!}\prod_{n=1}^{n'} \big(\! t\!+\!1\!-\!n(\!r+2\!)\!\big),\nonumber\\
\end{eqnarray}
}
for any $|l+k|=t$  with ${\cal M}(l,k)= i\in M_{({\cal R}^{N(r,w_{0})}_{t-n'(r+2)})_{(n',S^{(r)}_{w_{0}})}}$ 
{\small \begin{eqnarray}\label{zbzd5}
&&\sum_{(l^0,k^0,i^0)\in {\cal A}_{\big(({\cal R}^{N(r,w_{0})}_{t-n'(r+2)})_{(n',S_{w_{0}}^{(r)})}\big)_{t,lk}^i}}\big|\big( ({\cal
R}^{N(r,w_{0})}_{t-n'(r+2)})_{(n',S_{w_{0}}^{(r)})}\big)^{i(l^0,k^0,i^0)}_{t,lk}\big|\cdot  \max \{\langle i^0\rangle,\ \langle 2i-i^0\rangle\}\nonumber\\
&\leq&
\frac{\big(\! C\!(\theta,r\!)\big)^{\!t\!-\!n'\!(\!r\!+\!2\!)\!-\!2}}{\langle i\rangle^{\beta}  }\big(2^{\beta+2} (r+4)^2\big(C(\theta,r)\big)^{r+2}c  \frac{N^{\alpha+1}}{\gamma}\big)^{n'}\frac{(2N)^{(r+2)n'}}{n'!} \prod_{n=1}^{n'} \big(t+1-n(r+2)\big).\nonumber\\
\end{eqnarray}}
and for any $|l+k|=t$ with ${\cal M}(l,k)= i\in M_{({\cal R}^{T(r,w_{0})}_{t-n'(r+2)})_{(n',S^{(r)}_{w_{0}})}}$  
{\small\begin{eqnarray}\label{zbzd4}
&&\sum_{(l^0,k^0,i^0)\in {\cal A}_{\big(({\cal R}^{T(r,w_{0})}_{t-n'(r+2)})_{(n',S_{w_{0}}^{(r)})}\big)_{t,lk}^i}}\big|\big( ({\cal
R}^{T(r,w_{0})}_{t-n'(r+2)})_{(n',S_{w_{0}}^{(r)})}\big)^{i(l^0,k^0,i^0)}_{t,lk}\big|\cdot  \max \{\langle i^0\rangle,\ \langle 2i-i^0\rangle\}\nonumber\\
&\leq&
\frac{\big(C(\theta,r)\big)^{t-n'(r+2)-2}}{\langle i\rangle^{\beta}  }\big(2^{\beta+2} (r+4)^2\big(C(\theta,r)\big)^{r+2}c  \frac{N^{\alpha+1}}{\gamma}\big)^{n'}\frac{(2N)^{(r+2)n'}}{n'!} \prod_{n=1}^{n'} \big(t+1-n(r+2)\big).\nonumber\\
\end{eqnarray}}
By (\ref{zbzd1})-(\ref{zbzd4}) and assumption (\ref{N-c}),  for any $r+5\leq t\leq r_*+3$, $|l+k|=t$ and $i\in M_{g^{(r+1,w_{\theta})}_{t}}$, the following estimate holds  
{\small\begin{eqnarray}
\sum_{(l^0,k^0,i^0)\in {\cal A}_{(g^{(r+1,w_{\theta})})_{t,lk}^i} }|(g^{(r+1,w_{
\theta})})^{i(l^0,k^0,i^0)}_{t,lk}|\cdot \max \{\langle i^0\rangle, \langle 2i-i^0\rangle\}\leq  \big(C(\theta,r+1)\big)^{t-2} \frac{1}{\langle i \rangle^{\beta} }  ,
 \end{eqnarray}}
 which means that $g^{(r+1,w_{0})}(u,\bar{u})$ has $(\beta,0)$-type symmetric coefficients semi-bounded by $C(\theta,r+1)>0$. 
 
 Similarly, $ {\cal R}^{T(r+1,w_{\theta})}(u,\bar{u})$ and ${\cal R}^{N(r+1,w_{1})}(u,\bar{u}) $ 
are also of $(\beta,\theta)$-type symmetric coefficients semi-bounded by $C(\theta,r+1)>0$.

\end{proof}

\section*{Acknowledgement} 
This paper is supported in part by Science and Technology Commission of Shanghai Municipality (No. 18dz2271000).

 \end{document}